\documentclass[preprint,3p,times]{elsarticle}
\usepackage[toc,page,title,titletoc,header]{appendix}
\usepackage{algorithm}
\usepackage{algorithmic}
\usepackage{amsmath}
\usepackage{amssymb}
\usepackage{amsthm}
\usepackage{enumerate}
\usepackage{enumitem}
\usepackage{mathrsfs}
\usepackage{graphicx}
\newtheorem{example}{Example}
\usepackage{tikz,pgfplots,tikz-3dplot}
\usepackage{microtype}
\usepackage{natbib}
\usetikzlibrary{fit}
\biboptions{sort&compress}

\newtheorem{thm}{Theorem}[section]
\newtheorem{lemma}{Lemma}[section]

\newtheorem{rem}{Remark}
\renewcommand{}

\newcommand{\bI}{\mathbb{I}}

\newcommand{\mS}{\mathcal{S}}

\usepackage{cases}

\def\epsilon{\varepsilon} 
\newcommand{\mat}[1]{\boldsymbol{#1}}

\allowdisplaybreaks

\begin{document}
\begin{frontmatter}
\title{
{Adaptive moving mesh methods for isotropic/anisotropic mean curvature flow with axisymmetric geometry}
}

\author[1]{Yiming Wang}
\author[1]{Meng Li*}
\address[1]{School of Mathematics and Statistics, Zhengzhou University,
Zhengzhou 450001, China}
\ead{limeng@zzu.edu.cn}

\begin{abstract}
This paper introduces adaptive moving mesh methods for the numerical simulation of axisymmetric mean curvature flow, addressing both isotropic and anisotropic cases. The methods are developed within the framework of the mesh equidistribution principle, where a carefully designed tangential velocity is employed to dynamically redistribute mesh points during the evolution. To accurately capture the key geometric features of the evolving interfaces, we select monitor functions based on the curvature $\varkappa$, its arc-length derivative $\varkappa_s$, and the squared curvature $\varkappa^2$. These monitor functions can be flexibly tailored to suit different problem settings and play a vital role in determining the resulting mesh quality and numerical accuracy. Spatial discretization is performed using central finite differences, while temporal integration is handled with first- and second-order time-stepping schemes, including the BDFk ($k=1,2$) and Crank-Nicolson methods. Additionally, a Lagrange multiplier approach is incorporated into the adaptive system to enforce the underlying geometric constraint, resulting in energy-stable numerical schemes.
Numerical experiments confirm the convergence and energy stability of the proposed methods. More importantly, the results clearly show that the proposed methods offer significant advantages in complex geometric evolutions: by utilizing appropriately designed monitor functions, the adaptive methods achieve dynamic redistribution of mesh points, efficiently capturing localized geometric features, significantly improving numerical accuracy, and effectively preventing mesh degeneration, particularly in anisotropic cases.

\end{abstract} 


\end{frontmatter}

\section{Introduction}\setcounter{equation}{0}
In numerous physical and engineering phenomena, interfaces evolve over time under the influence of surface tension to minimize the total surface energy. Typical examples include grain boundary motion in materials science, morphological evolution of thin films during solid-state dewetting, and interface smoothing in multiphase flows. A mathematical framework that characterizes such surface-tension-driven motion is the mean curvature flow (MCF), in which the normal velocity of a surface is proportional to its mean curvature. 
As one of the fundamental geometric evolution equations, the MCF has been widely applied in various scientific domains, including physics, image processing, and materials science.
For comprehensive introductions to the theory and analysis of MCF, we refer to \cite{ecker2004regularity,mantegazza2011lecture} and references therein.

In this paper, we consider a family of closed hypersurfaces $(\mS(t))_{t \geq 0} \subset \mathbb{R}^3$, whose evolution is governed by the geometric law: 
\begin{align}
    V_{n} = \mathcal{H}\quad \text{on } \mS(t)
\end{align}
where $V_{n}=\partial_t\mat X\cdot\mat n_{\mS}$ denotes the normal velocity of $\mS(t)$ with $\mat n_{\mS}$ the normal vector of the surface $\mS(t)$,  and $\mathcal{H}$ represents its mean curvature. 
This evolution can be regarded as the $L^2$–gradient flow of the surface area functional 
$A(t)=\int_{\mS(t)} 1 \text{d}\mS$, derived from the first variation of the interfacial energy. To simulate this curvature-driven evolution, numerous numerical techniques have been developed over the past few decades. A comprehensive overview of these methods can be found in \cite{deckelnick2005computation, dziuk2013finite}. In general, existing approaches fall into three primary formulations: parametric, level-set, and phase-field methods. Among them, parametric finite element methods (PFEMs) provide high accuracy in evaluating geometric quantities and have been widely applied to the MCF \cite{dziuk1990algorithm, barrett2008parametric, m2017approximations, kovacs2019convergent}. On the other hand, level-set techniques \cite{osher1988fronts, giga2006surface} naturally accommodate topological changes, while diffusion-generated schemes \cite{ruuth1998efficient} provide a computationally efficient curvature-regularized evolution. Additionally, phase-field models \cite{mullins1956two} have been adopted to describe interface morphology, offering a flexible framework for complex phenomena. 
Furthermore, numerical frameworks have been extended to simulate other geometric evolution equations, such as volume-preserving MCF \cite{athanassenas1997volume, ruuth2003simple, bao2022volume}, surface diffusion \cite{bernoff1998axisymmetric, barrett2019finite, bao2021structure} and Willmore flow \cite{mayer2002numerical, deckelnick2011error, barrett2021stable}.

In anisotropic materials, the surface energy is always direction-dependent and varies with the interface orientation. Mathematically, the anisotropic surface energy of an interface \( \mS(t) \) can be expressed as the functional
\begin{align} \label{eq:AMCF_energy}
	W(t) = \int_{\mS(t)} \gamma(\mat n_{\mS}) \, \mathrm{d}\mS,
\end{align}
where \( \gamma(\mat n_{\mS}) > 0 \) characterizes the orientation-dependent surface energy density.
From this energy functional, the interface evolves via the $L^2$-gradient flow
\begin{align}\label{eqn:vt}
	V_{n} = \mathcal{H}_\gamma\quad \text{on } \mS(t),
\end{align}
where $\mathcal{H}_\gamma$ refers to the weighted mean curvature. 
Numerical studies of the anisotropic MCF have attracted significant attention in recent decades. From a theoretical perspective, variational formulations based on anisotropic surface energies, such as the Finsler-type framework of Bellettini and Paolini~\cite{bellettini1996anisotropic}, and subsequent analytical developments~\cite{giga2006surface, chen1999uniqueness} have contributed to the understanding of anisotropic curvature-driven evolution. On the computational side, the PFEMs constitute the predominant numerical approach. Numerical schemes for weighted curvature evolution were investigated in~\cite{deckelnick1999discrete, deckelnick2002fully}, and structure-preserving numerical schemes, in particular parametric finite element approximations, were developed to enhance stability in the presence of anisotropy~\cite{barrett2008numerical, barrett2008variational, bao2024structure}. More recent developments have focused on improving geometric accuracy and handling curvature concentration, highlighting the importance of appropriate discretization strategies in practical simulations.

In numerous practical scenarios, the evolving surface exhibits rotational symmetry, allowing the geometric flow to be reformulated as a one-dimensional problem. This dimensional reduction facilitates theoretical analysis (see, e.g.,~\cite{dziuk1991rotational, athanassenas1997volume, bernoff1998axisymmetric}) and contributes to enhanced numerical tractability. 
Early analytical studies focused on rotationally symmetric surfaces, including investigations of geometric behavior and singularity formation~\cite{dziuk1991rotational, athanassenas1997volume}. Subsequent extensions addressed boundary conditions and stability aspects~\cite{matioc2007boundary, hartley2016stability}.
 Although Jeffres~\cite{jeffres2009gauss} considered Gauss curvature flow, his analysis under axisymmetric settings remains relevant to understanding similar evolution mechanisms. 
On the numerical side, one of the earliest computational studies in axisymmetric curvature-driven evolution was conducted by Mayer and Simonett~\cite{mayer2002numerical} using a finite-difference discretization, where reduced accuracy near the symmetry axis was observed due to geometric degeneracy in cylindrical coordinates. Subsequent finite-difference formulations have been explored in related curvature-driven problems~\cite{osher1988fronts, sethian1999level, smereka2003semi, kimura1994accurate, ruuth1998efficient}, suggesting their potential under regularized parameterizations. Alternatively, the parametric finite element framework introduced by Dziuk~\cite{dziuk1990algorithm} laid a foundation for surface evolution approximations. Building on this formulation, Deckelnick and collaborators~\cite{deckelnick2003error, deckelnick2011error} developed weak approximations for axisymmetric surface diffusion and Willmore flow, and later Barrett, Garcke, and Nürnberg~\cite{barrett2019variational} proposed variational discretizations that achieve enhanced stability for axisymmetric curvature evolution. 
In~\cite{deckelnick2021error}, Deckelnick and N\"urnberg studied
finite difference approximations for isotropic axisymmetric MCF 
and derived corresponding error estimates for genus-$0$ surfaces.
Existing studies in the literature have primarily focused on isotropic axisymmetric MCF, including its theoretical analysis and numerical simulation. In contrast, research on anisotropic axisymmetric MCF, encompassing both theoretical investigation and numerical methods, remains largely unexplored and has not yet been systematically developed. 
In the work of Li and Zhao~\cite{li2024parametric}, an anisotropic surface diffusion model with axisymmetric geometry was formulated, and corresponding structure-preserving PFEMs were developed.
Subsequently, Li and Zhou~\cite{li2025structure} extended this framework to the axisymmetric anisotropic solid-state dewetting and proposed a structure-preserving parametric finite element approximation that preserves both volume conservation and energy stability. 
Inspired by these developments, the present work investigates the axisymmetric MCF with anisotropic effects and accordingly develops an appropriate numerical methodology.

In this study, for the sake of brevity, we only consider the MCF with a genus-1 axisymmetric structure; however, the adaptive moving mesh method we develop in this work is also applicable to other boundary conditions. Specifically, we denote $\bI=\mathbb R/\mathbb Z$ with $\partial \bI=\emptyset$. Assuming rotational symmetry with respect to the \(x_2\)-axis (see Fig.~\ref{Fig.1}), the evolving surface \(\mathcal{S}(t)\) can be generated by rotating a planar curve
\[
\mat{X}(\rho,t) = \left(r(\rho,t),\, z(\rho,t)\right)^{T},
\qquad \rho \in \mathbb{I},
\]
around the \(x_2\)-axis. 
Obviously, under the assumption, we have 
\begin{align*}
	 \mat{X}(\rho,t)\cdot\mat{e}_{1} > 0 
  \quad \forall\,\rho\in \overline{\mathbb{I}},
\end{align*}
where $\mat{e}_{1}=(1,0)^{T}$. The planar curve \(\Gamma(t) = \mat{X}(\mathbb{I}, t)\) is referred to as the generating curve of the axisymmetric surface. The associated surface parametrization is given by
\[
(\rho, \theta, t) \mapsto 
\left(
r(\rho,t)\cos\theta,\,
r(\rho,t)\sin\theta,\,
z(\rho,t)
\right)^{T},
\qquad \theta \in [0, 2\pi].
\]
\begin{figure} [!t]
    \centering
    \includegraphics[width=0.9\linewidth]{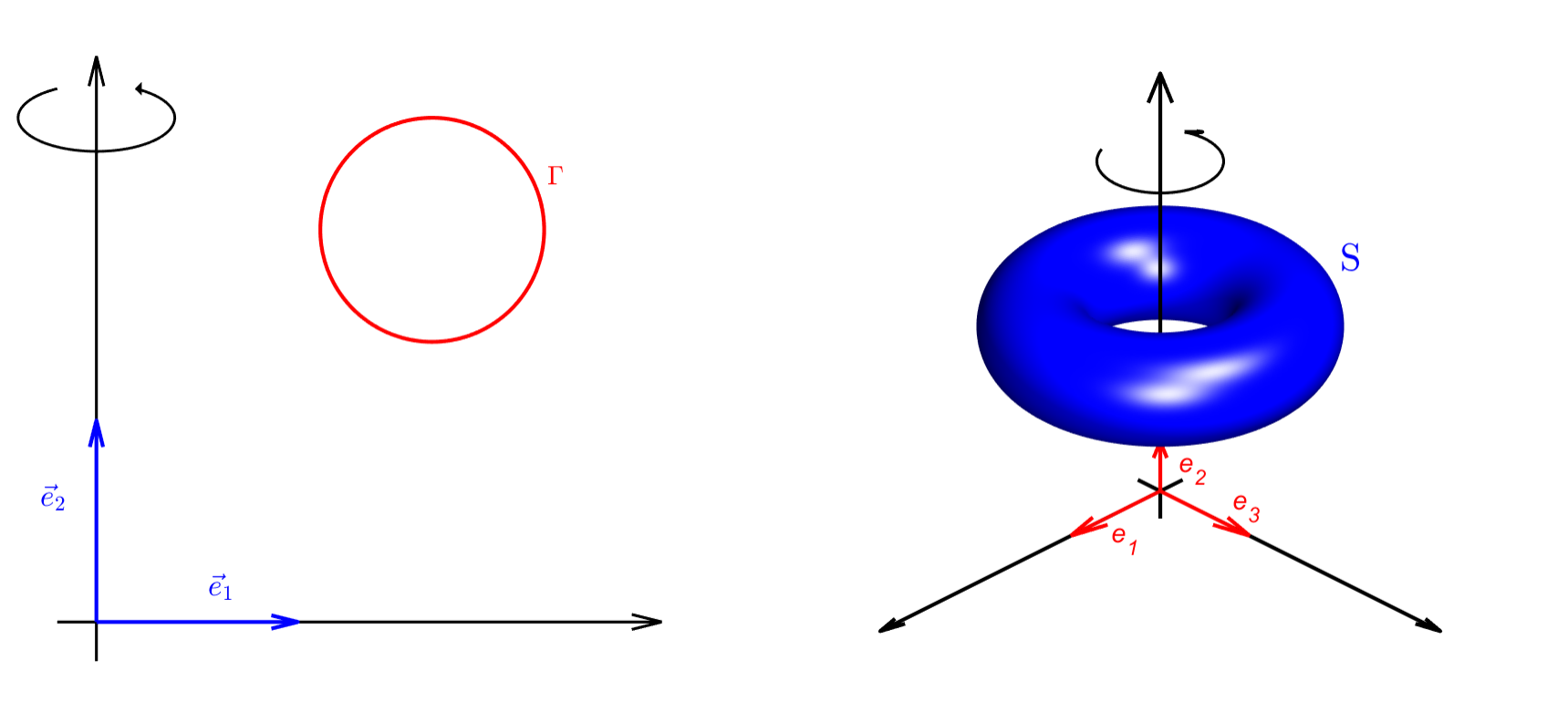}
    \caption{Sketch of $\Gamma$ and $S$, as well as the unit vectors $\mat e_1$, $\mat e_2$, and $\mat e_3$.}
    \label{Fig.1}
\end{figure}
Let $s$ denote the arc-length parameter along $\Gamma(t)$, and let
$\mat{n}$ be the associated unit normal. For the isotropic axisymmetric MCF, the normal velocity of the generating curve satisfies~\cite{barrett2019variational}
\begin{align}\label{eqn:iso_model}
     \partial_{t} \mat{X}\cdot\mat{n}
    =
    \varkappa
    -
    \frac{\mat{n}\cdot\mat{e}_{1}}
         {\mat{X}\cdot\mat{e}_{1}},
\end{align}
where $\varkappa=\partial_{ss}\mat{X}\cdot\mat{n}$ denotes the curvature of the 
generating curve $\Gamma(t)$.

We further assume that the anisotropy function $\gamma(\mat n_{\mS})$ is independent of the azimuthal angle and thus can be reduced to a function of $\mat n$, which we still denote by $\gamma(\mat n)$ with a slight abuse of notations. In this case,
the
normal velocity of the generating curve satisfies~\cite{li2024parametric}
\begin{align}\label{eq:aniso_model}
    \partial_{t} \mat{X}\cdot\mat{n}
    =
    \varkappa_{\gamma}
    -
    \frac{\left[\gamma(\theta)\mat n+\gamma'(\theta)\mat{\tau}\right]\cdot\mat{e}_{1}}
         {\mat{X}\cdot\mat{e}_{1}},
\end{align}
where $\gamma(\theta):=\gamma(\mat{n})>0$ denotes the surface energy density of the curve, and 
the anisotropic curvature of $\Gamma(t)$ is defined by
\begin{align}\label{eq:H_gamma}
    \varkappa_{\gamma}=
\left(\gamma(\theta)+\gamma''(\theta)\right)\varkappa.
\end{align}

Despite the considerable progress in improving the stability and accuracy of numerical schemes for curvature-driven interface evolution, traditional methods still require highly refined uniform meshes to achieve reliable solutions in the presence of strong curvature variations or sharp morphological changes, resulting in excessive computational cost. Adaptive moving mesh methods have therefore emerged as an effective strategy for enhancing computational efficiency while maintaining high accuracy. 
Huang, Ren, and Russell \cite{huang1994moving} systematically developed the moving mesh partial differential equation (PDE) framework based on the equidistribution principle, which is widely regarded as an important foundation of the moving mesh methods.  
Beyond the moving mesh PDE framework, the variational mesh optimization approach of Tourigny and Hulsemann \cite{tourigny1998new} offers an alternative viewpoint by formulating mesh adaptation as the minimization of a discrete energy functional. The harmonic map-based method of Li, Tang, and Zhang \cite{LI2001562} provides another direction, constructing adaptive meshes via iterative approximation of harmonic mappings and producing smooth, high-quality meshes in multidimensional settings. 
Mackenzie et al.\ \cite{mackenzie2019adaptive} employed an adaptive finite difference method (FDM) for the forced curve shortening flow by adopting a tangential redistribution strategy with a curvature-based monitor function, thereby accurately resolving highly localized geometric features while maintaining a relatively small number of mesh points.
For a comprehensive overview of these methodologies, including their theoretical foundations, practical implementation strategies, and several representative applications, we refer the reader to \cite{budd2009adaptivity, huang2010adaptive}. 
Despite these advances, extending adaptive moving mesh techniques to higher-dimensional or anisotropic geometric flow problems remains highly challenging, which motivates the developments presented in this work. In addition, the construction of energy-stable adaptive moving mesh methods is itself an important problem, further underscoring the need for our study.

In this work, we develop several adaptive moving mesh methods for simulating the axisymmetric MCF, considering both isotropic and anisotropic cases. The evolution of the generating curve is expressed through its normal and tangential components as follows:
\begin{equation}\label{eq:velocity-decomposition}
	\partial_t \mat X = \left(\partial_t \mat X \cdot \mat n\right)\mat n + \left(\partial_t \mat X \cdot \mat\tau\right)\mat\tau,
\end{equation}
where the tangential velocity is derived as the gradient flow of a mesh energy functional. This functional incorporates monitor functions based on the curvature $\varkappa$, its arc-length derivative $\varkappa_s$, and the squared curvature $\varkappa^2$. These monitor functions enable the mesh points to automatically concentrate in regions with significant geometric variation, while preserving the normal component of the axisymmetric MCF.
Building upon the new coupled system based on the total velocity equation, we develop adaptive moving mesh methods by employing central finite differences for spatial discretization and using first- and second-order implicit time integrators (BDFk, $k=1,2$, and Crank-Nicolson). The results demonstrate a flexible and accurate discretization method capable of addressing curvature-induced stiffness and effectively capturing localized features on dynamically evolving axisymmetric surfaces.
As an enhancement to the adaptive method, and inspired by the approach in \cite{garcke2025structure}, we further develop several energy-stable methods based on a Lagrange multiplier formulation. These methods ensure the monotonic decay of the discrete energy, significantly improving the robustness of the fully discrete algorithm. Numerical experiments confirm the convergence and energy stability of the proposed methods. More importantly, the results clearly demonstrate that the adaptive methods offer substantial advantages in handling complex geometric evolutions. By utilizing well-designed monitor functions, the adaptive moving mesh methods presented in this work facilitate dynamic mesh point redistribution, effectively capturing localized geometric features. This greatly enhances numerical accuracy while preventing mesh degeneration. Specifically, for anisotropic curvature evolution problems, the adaptive methods maintain high mesh quality, ensuring stable and accurate long-term numerical simulations of the evolution process.

The paper is structured as follows. Section 2 introduces a new continuous coupling system for the isotropic axisymmetric MCF, and based on this system, several adaptive finite difference approximations are established using the BDFk and Crank-Nicolson methods for time discretization. Furthermore, Section 2 extends the framework by developing area-decay methods based on the adaptive moving mesh approach.
Section 3 extends the adaptive moving mesh method for isotropic problems to anisotropic MCF and further establishes energy-stable methods. Section 4 focuses on the design principles of the curvature-based monitor function employed throughout the work. Section 5 presents numerical results illustrating the performance of the proposed methods. Finally, Section 6 concludes with remarks and future directions.


\section{Numerical approximations for the isotropic axisymmetric MCF}\label{sec2}
In this section, we will first introduce a new continuous coupling system for the isotropic axisymmetric MCF and, based on this, establish several adaptive finite difference approximations by using the BDFk and Crank-Nicolson methods in time.
\subsection{The strong formulation}
We introduce the arc-length coordinate $s(\rho,t) = \int_{0}^{\rho} \left|\partial_{q}\mat X(q,t)\right|\,\mathrm{d}q$ with $s \in [0,L]$, 
so that the corresponding arc-length element is $\mathrm{d}s = \left|\partial_\rho\mat X(\rho,t)\right|\,\mathrm{d}\rho, \partial_{s} = \frac{1}{\left|\partial_\rho\mat X(\rho,t)\right|}\,\partial_{\rho}.$
The unit tangent and normal vectors associated with the curve $\Gamma(t)$ are defined by
\begin{equation} \label{eq:tau_continue}
	\mat \tau(\rho,t) := \partial_{s} \mat{X}(\rho,t)
	= \frac{\partial_\rho\mat X(\rho,t)}{\left|\partial_\rho\mat X(\rho,t)\right|},
	\qquad
	\mat n(\rho,t) := -\mat \tau^{\perp}(\rho,t),
\end{equation}
where $(\cdot)^{\perp}$ denotes the clockwise rotation of a vector by $\frac{\pi}{2}$.  
The axisymmetric MCF can be written as the following
second-order geometric system. Given $\mat X(\rho,0)=\mat X_0(\rho)$, we find
$\left(\mat X(\rho,t),\,\varkappa(\rho,t)\right)$, $(\rho,t)\in\mathbb I\times(0,+\infty)$,
such that
\begin{subequations}\label{IAMCF}
	\begin{align}
		&\label{eq:x_t.n}\partial_t \mat X \cdot \mat{n} = \varkappa
		- \frac{\mat n\cdot\mat e_{1}}{\mat X\cdot\mat e_{1}}, \\[4pt]
		&\label{eq:hn}\varkappa\,\mat n= \partial_{ss}\mat X
		= \frac{\partial_{\rho\rho}\mat X}{\left|\partial_{\rho}\mat X\right|^{2}}
		-\frac{\mat \tau\cdot\partial_{\rho\rho}\mat X}{\left|\partial_{\rho}\mat X\right|^{3}}\,
		\partial_{\rho}\mat X .
	\end{align}    
\end{subequations}
The numerical approximations based on the original system \eqref{IAMCF} typically exhibit good asymptotic equidistribution grid quality. To account for the adaptive effects related to curvature, we further investigate a new coupled system, which incorporates a tangential velocity component by introducing a monitor function related to the curvature $\varkappa$, its arc-length derivative $\varkappa_s$, and the squared curvature $\varkappa^2$. To this end, motivated by \cite{huang1994moving,huang2010adaptive,tourigny1998new,mackenzie2019adaptive}, we recall a tangential velocity equation by adopting the first variation of the mesh energy.
\begin{lemma}\label{lemma:tangential}
    Let \(\Gamma(t)\) be a smooth closed curve parametrized by the arc-length coordinate \(s\), and let \(\rho = \rho(s,t)\) be a reparametrization function. Given a strictly positive monitor function \(M(\rho,t)\), define the mesh energy by
    \begin{equation}\label{eq:mesh-energy}
        I[\rho(t)] = \frac{1}{2} \int_{\Gamma(t)} \frac{1}{M(\rho,t)} \left( \partial_s \rho(s,t) \right)^{2} \, \mathrm{d}s.
    \end{equation}
    An evolution equation is obtained from the gradient flow equation
    \begin{equation}\label{eq:rho-evolution}
        \partial_t \rho = -\frac{P}{\mathcal{J}} \frac{\delta I}{\delta \rho},
    \end{equation}
    where \(\mathcal{J} > 0\) is a relaxation time constant, and \(P\) is a positive operator. Then, the total velocity is given by
    \begin{equation}\label{eq:total-velocity}
      \partial_t \mat X
		= V\mat n
		+ \frac{P}{\mathcal J}\,
		\left(M\left|\partial_\rho \mat X\right|\right)^{-2}
		\,\partial_\rho\!\left(M\left|\partial_\rho \mat X\right|\right)\,\mat\tau.
    \end{equation}
\end{lemma}

\begin{proof}
    We first compute the first variation of the mesh energy \eqref{eq:mesh-energy}. For a perturbation \(\rho^\varepsilon = \rho + \varepsilon \eta\), where \(\eta\) is a smooth function, and treating \(M(\rho,t)\) as fixed during the relaxation step, we write
    \begin{align}\label{lemma:tangential_pf1}
        I\left[\rho^\varepsilon\right]
        = \frac12 \int_{\Gamma(t)} \frac{1}{M} \left|\partial_s \left( \rho + \varepsilon \eta \right)\right|^2 \mathrm{d}s
        = \frac12 \int_{\Gamma(t)} \frac{1}{M} \left( \left|\partial_s \rho\right|^2 + 2\varepsilon\, \partial_s \rho \,\partial_s \eta + \varepsilon^2 \left|\partial_s \eta\right|^2 \right) \mathrm{d}s.
    \end{align}
    Differentiating \eqref{lemma:tangential_pf1} at \(\varepsilon = 0\) yields
    \begin{equation}\label{lemma:tangential_pf2}
        \left.\frac{\mathrm{d}}{\mathrm{d}\varepsilon} I\left[\rho^\varepsilon\right]\right|_{\varepsilon=0}
        = \int_{\Gamma(t)} \frac{1}{M} \, \partial_s \rho \, \partial_s \eta \, \mathrm{d}s.
    \end{equation}
    Since \(\Gamma(t)\) is a closed curve, we use integration by parts in \eqref{lemma:tangential_pf2}, which gives
    \begin{align}\label{lemma:tangential_pf3}
        \left.\frac{\mathrm{d}}{\mathrm{d}\varepsilon} I\left[\rho^\varepsilon\right]\right|_{\varepsilon=0}
        = -\int_{\Gamma(t)} \eta \, \partial_s \left( \frac{1}{M} \partial_s \rho \right) \mathrm{d}s.
    \end{align}
    Hence, we have 
    \begin{equation}\label{lemma:tangential_pf4}
        \frac{\delta I}{\delta \rho} = -\partial_s \left( \frac{1}{M} \partial_s \rho \right).
    \end{equation}
    Substituting \eqref{lemma:tangential_pf4} into \eqref{eq:rho-evolution} gives
    \begin{equation}\label{lemma:tangential_pf5}
        \partial_t \rho = \frac{P}{\mathcal{J}} \, \partial_s \left( \frac{1}{M} \partial_s \rho \right).
    \end{equation}
    Using the relation
    \[
    \partial_\rho s = \left|\partial_\rho \mat{X}\right| = \left(\partial_s \rho\right)^{-1},
    \]
    along with the identity \(\partial_s = \left|\partial_\rho \mat{X}\right|^{-1} \partial_\rho\), we obtain
    \begin{align}\label{lemma:tangential_pf6}
        \partial_t \rho = \frac{P}{\mathcal{J}} \, \frac{1}{\left|\partial_\rho \mat{X}\right|} \partial_\rho \left( \frac{1}{M \left|\partial_\rho \mat{X}\right|} \right).
    \end{align}
    From \eqref{lemma:tangential_pf6}, we obtain the tangential component
    \begin{equation}\label{eq:tangential-final}
        \partial_{t} \mat{X} \cdot \mat \tau = \partial_t s = \frac{P}{\mathcal{J}} \,
        \left(M \left|\partial_\rho \mat{X}\right|\right)^{-2} \partial_\rho \left( M \left|\partial_\rho \mat{X}\right| \right),
    \end{equation}
    which, together with the normal velocity equation \eqref{eq:x_t.n}, leads to the total velocity equation \eqref{eq:total-velocity}. Therefore, the proof is complete.
\end{proof}

From the vector equation \eqref{eq:hn}, we can deduce the following scalar equation
\begin{align}\label{eqn:scalar}
    \varkappa = \frac{\partial_{\rho \rho} \mat X\cdot\mat n}{\left|\partial_{\rho}\mat X\right|^2}.
\end{align}
From \eqref{eq:total-velocity} and  \eqref{eqn:scalar}, we derive the following new coupled system, which is to find $\mat X$, $V$ and $\varkappa$, such that 
\begin{subequations}\label{eq:isotropic-AMCF-system}
	\begin{align}
		&\label{eq:iso_x_t}
        \partial_t \mat X
		= V\mat n
		\;+\;
		\frac{P}{\mathcal J}\,
		\left(M\left|\partial_\rho \mat X\right|\right)^{-2}
		\,\partial_\rho\!\left(M\left|\partial_\rho \mat X\right|\right)\,\mat\tau ,\\[4pt]
		&\label{eq:V}V = \left(
		\varkappa - \frac{\mat n \cdot \mat e_1}{\mat X \cdot \mat e_1}
		\right) , \\[4pt]
		&\varkappa = \frac{\partial_{\rho \rho} \mat X\cdot\mat n}{\left|\partial_{\rho}\mat X\right|^2}.
	\end{align}
\end{subequations}

The area of the axisymmetric surface, determined by the generating curve governed by the system \eqref{eq:isotropic-AMCF-system}, can be demonstrated to be dissipative.
\begin{lemma}
	\label{lem:area_dissipation}
	Let $A(t)$ denote the surface area of the axisymmetric surface generated by
	$\Gamma(t)$ governed by \eqref{eq:isotropic-AMCF-system}. Then, $A(t)$ is dissipative in the sense that 
	\begin{equation}\label{eqn:area_dissipation}
		\frac{\mathrm d}{\mathrm dt} A(t)
		=
		-\,2\pi \int^1_{0} (\mat X\cdot \mat e_1)\,V^{2}\,\left|\partial_\rho\mat X\right| \,\mathrm d\rho
		\leq 0.
	\end{equation}
\end{lemma}

	\begin{proof}
		Differentiating the area with respect to time yields
		\begin{align} \label{eqn:area_dissipation_pf1}
			\frac{\mathrm d}{\mathrm dt}A(t)
			=
			2\pi\frac{\mathrm d}{\mathrm dt}\int^1_{0}\mat X\cdot \mat e_1\left|\partial_\rho\mat X\right|\mathrm d\rho
			=2\pi\int^1_{0}\partial_t\mat X\cdot \mat e_1\left|\partial_\rho\mat X\right|\mathrm d \rho+2\pi\int^1_{0}\mat X\cdot \mat e_1\frac{\partial}{\partial t}\left|\partial_\rho\mat X\right|\mathrm d\rho.
		\end{align}
		By using 
		\begin{align*}
			\frac{\partial}{\partial t}\left|\partial_\rho\mat X\right|=\partial_t\partial_\rho\mat X\cdot \mat \tau 
		\end{align*}
		in \eqref{eqn:area_dissipation_pf1}, we get 
		\begin{align}\label{eqn:area_dissipation_pf2}
			\frac{\mathrm d}{\mathrm dt}A(t)
			&=
			2\pi\int^1_{0}
			\partial_t\mat X\cdot \mat e_1
			\left|\partial_\rho\mat X\right|
			\,\mathrm d\rho
			+
			2\pi\int^1_{0}
			\left(\mat X\cdot \mat e_1\right)
			\partial_\rho\partial_t\mat X\cdot \mat\tau
			\,\mathrm d\rho .
		\end{align}
By applying integration by parts to the second term on the right-hand side of \eqref{eqn:area_dissipation_pf2}, we obtain
	\begin{align}\label{eqn:area_dissipation_pf3}
		\frac{\mathrm d}{\mathrm dt}A(t)
		&=
		2\pi\int^1_{0}
		\partial_t\mat X\cdot \mat e_1
		\left|\partial_\rho\mat X\right|
		\,\mathrm d\rho
		-
		2\pi\int^1_{0}
		\partial_\rho\!\left(
		\left(\mat X\cdot \mat e_1\right)\mat\tau
		\right)\cdot
		\partial_t\mat X
		\,\mathrm d\rho .
	\end{align}
	Here the boundary term vanishes due to periodicity for closed curves.
	Thanks to the relation $\partial_\rho\mat\tau
		= \varkappa\,\left|\partial_\rho\mat X\right| \mat n,$
		we have 
		\begin{align}\label{eqn:area_dissipation_pf4}
		   \partial_\rho\!\left(
		\left(\mat X\cdot \mat e_1\right)\mat\tau
		\right)
		=
		\left(\partial_\rho\mat X\cdot \mat e_1\right)\mat\tau
		+
		\left(\mat X\cdot \mat e_1\right)
		\varkappa
		\left|\partial_\rho\mat X\right|
		\mat n . 
		\end{align}
        Substituting \eqref{eqn:area_dissipation_pf4} into \eqref{eqn:area_dissipation_pf3}, and using \eqref{eq:iso_x_t}, we obtain
\begin{align}\label{eqn:area_dissipation_pf5}
\frac{\mathrm d}{\mathrm dt}A(t)
&=
2\pi\int^1_{0}
\partial_t\mat X\cdot \mat e_1
\left|\partial_\rho\mat X\right|
\,\mathrm d\rho
-
2\pi\int^1_{0}
\left(\partial_\rho\mat X\cdot \mat e_1\right)
\left(\partial_t\mat X\cdot\mat \tau\right)
\,\mathrm d\rho \nonumber
-
2\pi\int^1_{0}
\left(\mat X\cdot \mat e_1\right)
\varkappa
\left|\partial_\rho\mat X\right|
\left(\partial_t\mat X\cdot\mat n\right)
\,\mathrm d\rho \nonumber\\
&=
2\pi\int_0^{1}
\bigg[
\partial_t\mat X\cdot \mat e_1
-
\left(\mat\tau\cdot \mat e_1\right)
\left(\partial_t\mat X\cdot\mat\tau\right)
-
\left(\mat X\cdot \mat e_1\right)\varkappa
\left(\partial_t\mat X\cdot\mat n\right)
\bigg]\left|\partial_\rho\mat X\right|
\,\mathrm d\rho \nonumber\\
&=
2\pi\int_0^{1}
\left(\partial_t\mat X\cdot\mat n\right)
\left(\mat n\cdot \mat e_1\right)\left|\partial_\rho\mat X\right|
\,\mathrm d\rho
-
2\pi\int_0^{1}
\left(\mat X\cdot \mat e_1\right)
\varkappa
\left(\partial_t\mat X\cdot\mat n\right)\left|\partial_\rho\mat X\right|
\,\mathrm d\rho \nonumber\\
&=
2\pi\int_0^{1}
V
\bigg[
\mat n\cdot \mat e_1
-
\left(\mat X\cdot \mat e_1\right)\varkappa
\bigg]\left|\partial_\rho\mat X\right|
\,\mathrm d\rho .
\end{align}
		Finally, from \eqref{eq:V}, 
		we have
\begin{align}\label{eqn:area_dissipation_pf6}
    		\mat n\cdot \mat e_1
		-
		\left(\mat X\cdot \mat e_1\right)\varkappa
		=
		-\left(\mat X\cdot \mat e_1\right)V.
\end{align}
		Therefore, substituting \eqref{eqn:area_dissipation_pf6} into \eqref{eqn:area_dissipation_pf5}
 concludes that 
		\[
		\frac{\mathrm d}{\mathrm dt}A(t)
		=
		-\,2\pi
		\int^1_{0}
		\left(\mat X\cdot \mat e_1\right)
		V^{2}
		\left|\partial_\rho\mat X\right|
		\,\mathrm d\rho
		\le 0,
		\]
		which completes the proof.
	\end{proof}

\subsection{Numerical approximations}

Let $h=1/M$ and $\Delta t=T/N$ denote the spatial mesh width and the time step, respectively, where $M,N\in\mathbb{N}$.
We employ a uniform discretization of the parameter domain by setting $\rho_i=ih$ for $i=0,1,\dots,M-1$, and impose periodicity through the identification $\rho_M\equiv\rho_0$.
The temporal grid is given by $t^n=n\,\Delta t$ for $n=0,1,\dots,N$.
At each time level $t^n$, the discrete curve is represented by the nodal values $\mat X_i^n\approx \mat X(\rho_i,t^n)\in\mathbb{R}^2$ for $i=0,1,\dots,M-1$.
To define the finite difference operators, let \( f \) be a scalar-valued function sampled at the grid points \( \rho_i \) and assumed to satisfy periodicity, i.e., \( f_{i+M} = f_i \).
The forward, backward, and centered first-order differences, together with the
standard centered second-order difference used for curvature and other geometric
quantities, are defined by
\begin{align*}
	\delta_{\rho}^+ f_i &= \frac{f_{i+1}-f_i}{h}, 
	\qquad
	\delta_{\rho}^- f_i = \frac{f_i-f_{i-1}}{h}, \\[4pt]
	\delta_{\rho} f_i &= \frac{f_{i+1}-f_{i-1}}{2h},
	\qquad
	\delta_{\rho\rho} f_i = \frac{f_{i+1}-2f_i+f_{i-1}}{h^2}.
\end{align*}
All operators act componentwise when \(f\) is vector-valued.

In the discrete framework, the geometric quantities are evaluated through finite difference
approximations along the periodic parameter~\(\rho\).
The derivative of the parametrization with respect to \(\rho\) is approximated
by the centered difference
\[
\partial_\rho \mat X(\rho_i,t^n)
\;\approx\;
\delta_\rho \mat X_i^n
:=
\frac{\mat X_{i+1}^n - \mat X_{i-1}^n}{2h}.
\] 
The unit tangent and normal vectors are denoted by 
\begin{align*}
\boldsymbol{\tau}_i^n
= \frac{\delta_\rho \mat X_i^n}
{\left|\delta_\rho \mat X_i^n\right|};
\qquad
\boldsymbol{n}_i^n
= R\,\boldsymbol{\tau}_i^n,
\qquad
R =
\begin{pmatrix}
	0 & -1 \\
	1 & \;\;0
\end{pmatrix}.
\end{align*}
To obtain the fully discrete methods of \eqref{eq:anisotropic-AMCF-system}, 
we discretize the spatial derivatives using second-order centered finite differences, and approximate the temporal derivative by the BDF1, BDF2 and
Crank-Nicolson methods. 
For the Crank-Nicolson method, we additionally define
\[
\widehat{f}_i^{\,n+\frac12}
:= \frac{f_i^{\,n+1}+f_i^{\,n}}{2},
\]
and apply the same notation componentwise to vector-valued quantities.

Given the solution values $\left\{
\left(\mat X_i^{\,n-p},\, V_i^{\,n-p},\, \varkappa_i^{\,n-p}\right)
\right\}_{p=0}^{r-1}$ or $\left\{
\left(\mat X_i^{\,n-p}, \mathcal B_i^{n-p}, V_i^{\,n-p}, \varkappa_i^{\,n-p}\right)
\right\}_{p=0}^{r-1}$ 
from the previous \(r\) time levels, where \(r=1\) for the BDF1 and Crank-Nicolson methods, and \(r=2\) for the BDF2 method, the time-stepping procedure computes the updated solution
$\left(\mat X_i^{\,n+1},\, V_i^{\,n+1},\, \varkappa_i^{\,n+1}\right)$ or $\left(\mat X_i^{\,n+1}, \mathcal B_i^{n+1},  V_i^{\,n+1}, \varkappa_i^{\,n+1}\right)$
by solving the following fully discrete approximations. By employing different time discretization strategies, we obtain the fully discrete schemes for the new coupled system \eqref{eq:isotropic-AMCF-system}, including 
\begin{itemize}
	\item [(i)]\textbf{BDFk method ($k=1,2$) for the isotropic system (referred to the ISO-A-BDFk-FDMs)}:
	\begin{subequations}\label{eq:isotropic-BDFk}
		\begin{align}
			&\frac{1}{\Delta t}\sum_{\rho=0}^{k}
			\alpha_\rho\,\mat X_i^{\,n+1-\rho}
			=
			V_i^{n+1}\,\mat n_i^{n+1}
			+
			\frac{P_i^{n+1}}{\mathcal J}
			\left(
			M_i^{n+1}\left|\delta_\rho\mat X_i^{n+1}\right|
			\right)^{-2}
			\delta_\rho\!\left(
			M_i^{n+1}\left|\delta_\rho\mat X_i^{n+1}\right|
			\right)\mat\tau_i^{n+1},
			\\[4pt]
			&V_i^{n+1}
			=
			\varkappa_i^{n+1}
			-
			\frac{\mat n_i^{n+1}\cdot\mat e_1}
			{\mat X_i^{n+1}\cdot\mat e_1},
			\\[4pt]
			&\varkappa_i^{n+1}
			=
			\frac{
				\delta_{\rho\rho}\mat X_i^{n+1}\cdot \mat n_i^{n+1}}
			{\left|\delta_\rho\mat X_i^{n+1}\right|^2}.
		\end{align}
	\end{subequations}
	where the coefficients $\{\alpha_\rho\}_{\rho=0}^{k}$ correspond to the k-step BDFk method. For instance, the commonly used temporal first-order and second-order schemes are given by
	\begin{align*}
		&\text{BDF1 }(k=1): 
		&& \alpha_0 = 1,\quad \alpha_1 = -1; \\[4pt]
		&\text{BDF2 }(k=2): 
		&& \alpha_0 = \tfrac{3}{2},\quad \alpha_1 = -2,\quad \alpha_2 = \tfrac{1}{2}. 
	\end{align*}
	\item [(ii)]\textbf{Crank-Nicolson method for the isotropic system (referred to the ISO-A-CN-FDM)}: 
	\begin{subequations}\label{eq:isotropic-CN}
		\begin{align}
			&\frac{\mat X_i^{n+1}-\mat X_i^{n}}{\Delta t}
			=
			\widehat{V}_i^{n+\frac12}\,
			\widehat{\mat n}_i^{\,n+\frac12}
			+
			\widehat{\mathcal B}_i^{n+\frac12}	\widehat{\mat \tau}_i^{\,n+\frac12},
			\\[4pt]
			& \mathcal B_i^{n+1}=	\frac{P_i^{n+1}}{\mathcal J}
			\left(
			M_i^{n+1}\left|\delta_\rho\mat X_i^{n+1}\right|
			\right)^{-2}
			\delta_\rho\!\left(
			M_i^{n+1}\left|\delta_\rho\mat X_i^{n+1}\right|
			\right),\\
			&V_i^{n+1}
			=
			\varkappa_i^{n+1}
			-
			\frac{
				\mat n_i^{\,n+1}\cdot\mat e_1}
			{\mat X_i^{\,n+1}\cdot\mat e_1},
			\\[4pt]
			&\varkappa_i^{n+1}
			=
			\frac{
				\delta_{\rho\rho}\mat X_i^{n+1}\cdot \mat n_i^{n+1}}
			{\left|\delta_\rho\mat X_i^{n+1}\right|^2}.
		\end{align}
	\end{subequations}
	
\end{itemize}

To treat the nonlinearity in \eqref{eq:isotropic-BDFk} and \eqref{eq:isotropic-CN}, we adopt a Picard-type iterative strategy at each time step. 
Given the iteration $\left(\mat X_i^{n+1,m},\, V_i^{n+1,m},\, \varkappa_i^{n+1,m}\right)$ or $\left(\mat X_i^{n+1,m}, \mathcal B^{n+1,m}, V_i^{n+1,m},  \varkappa_i^{n+1,m}\right)$, the next iteration $\left(\mat X_i^{n+1,m+1},\, V_i^{n+1,m+1},\, \varkappa_i^{n+1,m+1}\right)$ or $\left(\mat X_i^{n+1,m+1}, \mathcal B^{n+1,m+1}, V_i^{n+1,m+1},\, \varkappa_i^{n+1,m+1}\right)$ is obtained by solving the following linearized system:
\begin{itemize}
	\item [(i)]\textbf{Picard-iteration method for the ISO-A-BDFk-FDMs}:
	\begin{subequations}\label{eq:isotropic-BDFk-Picard}
		\begin{align}
			&\frac{\alpha_0 \mat X_i^{\,n+1,m+1}+\sum_{p=1}^{k} \alpha_p \mat X_i^{\,n+1-p} } {\Delta t}
			=
			V_i^{n+1,m+1}\,\mat n_i^{n+1,m}
			+
			\frac{P_i^{n+1}}{\mathcal J}
			\left(
			M_i^{n+1,m}\left|\delta_\rho\mat X_i^{n+1,m}\right|
			\right)^{-2}
			\delta_\rho\!\left(
			M_i^{n+1,m}\left|\delta_\rho\mat X_i^{n+1,m}\right|
			\right)\mat\tau_i^{n+1,m},
			\\[4pt]
			&V_i^{n+1,m+1}
			=
			\varkappa_i^{n+1,m+1}
			-
			\frac{\mat n_i^{n+1,m}\cdot\mat e_1}
			{\mat X_i^{n+1,m}\cdot\mat e_1},
			\\[4pt]
			&\varkappa_i^{n+1,m+1}
			=
			\frac{
				\delta_{\rho\rho}\mat X_i^{n+1,m+1}\cdot \mat n_i^{n+1,m}}
			{\left|\delta_\rho\mat X_i^{n+1,m}\right|^2}.
		\end{align}
	\end{subequations}
	\item [(ii)]\textbf{Picard-iteration method for the ISO-A-CN-FDM}:
	\begin{subequations}\label{eq:isotropic-CN-Picard}
		\begin{align}
			&\frac{\mat X_i^{n+1,m+1}-\mat X_i^{n}}{\Delta t}
			=
			\widehat{V}_i^{\,n+\frac12,m+1}\,
			\widehat{\mat n}_i^{\,n+\frac12,m}
			+
			\widehat{\mathcal B}_i^{\,n+\frac12,m+1}\,
			\widehat{\mat \tau}_i^{\,n+\frac12,m},
			\\[4pt]
			&\mathcal B_i^{n+1,m+1}
			=
			\frac{P_i^{n+1,m}}{\mathcal J}
			\left(
			M_i^{n+1,m}\left|\delta_\rho\mat X_i^{n+1,m}\right|
			\right)^{-2}
			\delta_\rho\!\left(
			M_i^{n+1,m}\left|\delta_\rho\mat X_i^{n+1,m}\right|
			\right),
			\\[4pt]
			&V_i^{n+1,m+1}
			=
			\varkappa_i^{n+1,m+1}
			-
			\frac{
				\mat n_i^{\,n+1,m}\cdot\mat e_1
			}{
				\mat X_i^{\,n+1,m}\cdot\mat e_1
			},
			\\[4pt]
			&\varkappa_i^{n+1,m+1}
			=
			\frac{
				\delta_{\rho\rho}\mat X_i^{n+1,m+1}\cdot \mat n_i^{n+1,m}
			}{
				\left|\delta_\rho\mat X_i^{n+1,m}\right|^2
			}.
		\end{align}
	\end{subequations}
\end{itemize}
\begin{rem}
At the fully discrete level, even for the isotropic axisymmetric MCF, establishing an area dissipation property analogous to
Lemma~\ref{lem:area_dissipation} is highly nontrivial.
Standard time discretizations, such as BDFk or Crank--Nicolson schemes,
do not generally admit discrete counterparts of the transport identities
and integration--by--parts arguments used in the continuous analysis.
As a result, the monotonic decay of the surface area is not automatically
preserved, especially for higher-order time discretizations. 
In the next subsection, we adopt the Lagrange multiplier approach to establish several area-decay approximations for the system \eqref{eq:isotropic-AMCF-system}.
\end{rem}

\subsection{Area-decay methods}
For the axisymmetric MCF \eqref{eq:isotropic-AMCF-system}, the continuous model naturally satisfies an area dissipation property, as shown in Lemma \ref{lem:area_dissipation}. While standard time discretizations, such as BDFk or Crank-Nicolson methods, typically exhibit area decay in practical computations, this property may not be directly demonstrated at the fully discrete level, particularly for higher-order temporal schemes.
To address this issue, motivated by \cite{garcke2025structure}, we adopt a Lagrange multiplier approach and introduce a
scalar function into the normal velocity, leading to a formulation that is
amenable to the construction of fully discrete schemes which can preserve the area-decay property at the discrete level, independently of the time step size.
This idea leads to the following reformulated isotropic axisymmetric MCF:
\begin{subequations}\label{eq:isotropic-AMCF-LM-system}
	\begin{align}
		&\partial_t \mat X
		= \left(
		1 - \lambda(t)
		\right)V\mat n
		\;+\;
		\frac{P}{\mathcal J}\,
		\left(M\left|\partial_\rho \mat X\right|\right)^{-2}
		\,\partial_\rho\!\left(M\left|\partial_\rho \mat X\right|\right)\,\mat\tau ,
		\label{eq:LM-evolution} \\[4pt]
		&V
		= \varkappa
		- \frac{\mat n \cdot \mat e_1}{\mat X \cdot \mat e_1},
		\label{eq:LM-V} \\[4pt]
		&\varkappa
		= \frac{\partial_{\rho \rho} \mat X\cdot\mat n}
		{\left|\partial_{\rho}\mat X\right|^2},
		\label{eq:LM-kappa} \\[4pt]
		&\frac{\mathrm d}{\mathrm dt}A(t)
		= -\,2\pi\int^1_{0}
		r\,V^{2}\left|\partial_\rho\mat X\right|
		\,\mathrm d\rho .
		\label{eq:ISO-LM-energy}
	\end{align}
\end{subequations}

\begin{lemma}
	\label{lem:xx}
The new system \eqref{eq:isotropic-AMCF-LM-system} is consistent with the original system \eqref{eq:isotropic-AMCF-system}.
\end{lemma}

\begin{proof}
For convenience, we denote $V_n := \partial_t \mat X\cdot \mat n$ and $r:=\mat X\cdot \mat e_1$. Differentiating $A(t)$ with respect to $t$ obtains  
  \begin{equation}\label{eq:xx_pf1}
		\frac{\mathrm d}{\mathrm dt}A(t)
		= -\,2\pi\int^1_{0} r\,V\,V_n\,\left|\partial_\rho\mat X\right|
		\,\mathrm d\rho.
	\end{equation}
 Thanks to \eqref{eq:LM-evolution}, we have 
	\(V_n = V-\lambda(t) V\).
	Substituting this into \eqref{eq:xx_pf1} gives 
	\begin{align}\label{eq:xx_pf2}
		\frac{\mathrm d}{\mathrm dt}A(t)
		&= -\,2\pi\int^1_{0} r\,V\bigg[V-\lambda(t)\,V\bigg]\,\left|\partial_\rho\mat X\right|
		\,\mathrm d\rho \nonumber\\
		&= -\,2\pi\int^1_{0} r\,V^{2}\,\left|\partial_\rho\mat X\right|
		\,\mathrm d\rho
		\;+\;2\pi\lambda(t)\int^1_{0} r\,V^{2}\,\left|\partial_\rho\mat X\right|
		\,\mathrm d\rho.
	\end{align}
	Comparing \eqref{eq:xx_pf2} with \eqref{eq:ISO-LM-energy} implies $\lambda(t)\equiv 0$. 
	Hence, The new system \eqref{eq:isotropic-AMCF-LM-system} is consistent with the original system \eqref{eq:isotropic-AMCF-system}. Therefore, we have completed the proof. 
\end{proof}

Based on the above equivalent reformulation \eqref{eq:isotropic-AMCF-LM-system}, we construct the following area-decay methods.
\begin{itemize}
	\item[(i)]\textbf{BDFk methods ($k=1,2$) for the isotropic system with
		Lagrange multiplier (referred to the ISO-A-LM-BDFk-FDMs)}:
	\begin{subequations}\label{eq:isotropic-BDFk-LM}
		\begin{align}
			&\frac{1}{\Delta t}\sum_{\rho=0}^{k}\alpha_\rho\,\mat X_i^{\,n+1-\rho}
			=
			\left(
			V_i^{\,n+1}-\lambda^{n+1}\,V_i^{\,n+1}
			\right)\mat n_i^{\,n+1}
			+
			\frac{P_i^{\,n+1}}{\mathcal J}
			\left(
			M_i^{\,n+1}\left|\delta_\rho\mat X_i^{\,n+1}\right|
			\right)^{-2}
			\delta_\rho\!\left(
			M_i^{\,n+1}\left|\delta_\rho\mat X_i^{\,n+1}\right|
			\right)\mat\tau_i^{\,n+1},
			\label{eq:iso-BDFk-LM-evolution}
			\\[4pt]
			&V_i^{\,n+1}
			=
			\varkappa_i^{\,n+1}
			-
			\frac{\mat n_i^{\,n+1}\cdot\mat e_1}
			{\mat X_i^{\,n+1}\cdot\mat e_1},
			\label{eq:iso-BDFk-LM-V}
			\\[4pt]
			&\varkappa_i^{\,n+1}
			=
			\frac{
				\delta_{\rho\rho}\mat X_i^{\,n+1}\cdot \mat n_i^{\,n+1}}
			{\left|\delta_\rho\mat X_i^{\,n+1}\right|^2},
			\label{eq:iso-BDFk-LM-kappa}
			\\[4pt]
			&\frac{1}{\Delta t}\sum_{\rho=0}^{k}\alpha_\rho\,A^{n+1-\rho}
			=
			-\,2\pi\int^1_{0}
			\left(\mat X^{n+1}\cdot\mat e_1\right)\,
			\left|V^{n+1}\right|^{2}\,
			\left|\partial_\rho \mat X^{n+1}\right|\,
			\mathrm d\rho .
			\label{eq:iso-BDFk-LM-energy}
		\end{align}
	\end{subequations}
	
	\item [(ii)]\textbf{Crank-Nicolson method for the isotropic system with
		Lagrange multiplier (referred to as the ISO-A-LM-CN-FDM)}:
	\begin{subequations}\label{eq:isotropic-CN-LM}
		\begin{align}
			&\frac{\mat X_i^{n+1}-\mat X_i^{n}}{\Delta t}
			=
			\left(
			\widehat{V}_i^{\,n+\frac12}
			-
			\lambda^{n+\frac12}\,
			\widehat{V}_i^{\,n+\frac12}
			\right)
			\widehat{\mat n}_i^{\,n+\frac12}
			+
			\widehat{\mathcal B}_i^{\,n+\frac12}\,
			\widehat{\mat \tau}_i^{\,n+\frac12},
			\label{eq:iso-CN-LM-evolution}
			\\[4pt]
			&\mathcal B_i^{n+1}
			=
			\frac{P_i^{n+1}}{\mathcal J}
			\left(
			M_i^{n+1}\left|\delta_\rho\mat X_i^{n+1}\right|
			\right)^{-2}
			\delta_\rho\!\left(
			M_i^{n+1}\left|\delta_\rho\mat X_i^{n+1}\right|
			\right),
			\label{eq:iso-CN-LM-B}
			\\[4pt]
			&V_i^{n+1}
			=
			\varkappa_i^{n+1}
			-
			\frac{
				\mat n_i^{\,n+1}\cdot\mat e_1}
			{\mat X_i^{\,n+1}\cdot\mat e_1},
			\label{eq:iso-CN-LM-V}
			\\[4pt]
			&\varkappa_i^{n+1}
			=
			\frac{
				\delta_{\rho\rho}\mat X_i^{n+1}\cdot \mat n_i^{n+1}}
			{\left|\delta_\rho\mat X_i^{n+1}\right|^2},
			\label{eq:iso-CN-LM-kappa}
			\\[4pt]
			&\frac{A^{n+1}-A^{n}}{\Delta t}
			=
			-\,2\pi\int^1_{0}
			\left(\mat X^{n+\frac12}\cdot\mat e_1\right)\,
			\left|\widehat{V}^{\,n+\frac12}\right|^{2}\,
			\left|\partial_\rho \mat X^{\,n+\frac12}\right|\,
			\mathrm d\rho .
			\label{eq:iso-CN-LM-energy}
		\end{align}
	\end{subequations}
\end{itemize}

\begin{thm} \label{thm:energy_dissipation}
The fully discrete schemes~\eqref{eq:isotropic-BDFk-LM} and~\eqref{eq:isotropic-CN-LM} are unconditionally area-decay, in the sense that 
\begin{align}\label{eqn:area-decay}
    A^{n+1}\leq A^n,\qquad n\geq 0. 
\end{align}
\end{thm}

\begin{proof}
For the ISO-A-LM-BDF1-FDM, it follows from~\eqref{eq:iso-BDFk-LM-energy} that
\begin{align}\label{eqn:area-decay-pf1}
    \frac{A^{n+1}-A^{n}}{\Delta t}
    =
    -\,2\pi \int^1_{0}
    \left(\mat X^{n+1}\cdot\mat e_1\right)\,
    \left|V^{n+1}\right|^{2}\,
    \left|\partial_\rho \mat X^{n+1}\right|\,
    \mathrm d\rho .
\end{align}
Under the assumptions that $\mat X^{n+1}\cdot\mat e_1>0$ and $\left|\partial_\rho \mat X^{n+1}\right|>0$, the right-hand side of~\eqref{eqn:area-decay-pf1} is non-positive, which immediately implies the area-decay property~\eqref{eqn:area-decay}.

In addition, for the ISO-A-LM-BDF2-FDM, the first time step is computed by the ISO-A-LM-BDF1-FDM. 
According to the above analysis, this yields $A^1 \le A^0$. 
By using~\eqref{eq:iso-BDFk-LM-energy}, we further obtain
\begin{align}\label{eqn:area-decay-pf2}
    \frac{1}{\Delta t}
    \left(
        \frac{3}{2} A^{n+1}
        - 2 A^{n}
        + \frac{1}{2} A^{n-1}
    \right)
    =
    -\,2\pi \int^1_{0}
    \left(\mat X^{n+1}\cdot\mat e_1\right)\,
    \left|V^{n+1}\right|^{2}\,
    \left|\partial_\rho \mat X^{n+1}\right|\,
    \mathrm d\rho
    \le 0 .
\end{align}
The inequality \eqref{eqn:area-decay-pf2} further implies that
\begin{align}\label{eqn:area-decay-pf3}
    \frac{3}{2}\left(A^{n+1}-A^{n}\right)
    \le
    \frac{1}{2}\left(A^{n}-A^{n-1}\right).
\end{align}
We now proceed by mathematical induction. 
Since $A^1 \le A^0$, assume that
\[
A^{j} \le A^{j-1}, \qquad j = 1,2,\dots,n .
\]
Then, \eqref{eqn:area-decay-pf3} immediately implies that $A^{n+1} \le A^{n}$. 
Therefore, the area-decay property~\eqref{eqn:area-decay} holds for all $n \ge 1$.

Finally, for the ISO-A-LM-CN-FDM, 
\eqref{eq:iso-CN-LM-energy} directly implies
\[
\frac{A^{n+1}-A^{n}}{\Delta t}
=
-\,2\pi \int^1_{0}
\left(\widehat{\mat X}^{n+\frac12}\cdot\mat e_1\right)\,
\left|\widehat{V}^{\,n+\frac12}\right|^{2}\,
\left|\partial_\rho \mat X^{\,n+\frac12}\right|\,
\mathrm d\rho
\;\le\;0,
\]
and hence $A^{n+1}\le A^{n}$. 
Therefore, the area-decay  property~\eqref{eqn:area-decay} holds for all $n \ge 1$. 
This completes the proof.
\end{proof}


\section{Numerical approximations for the anisotropic axisymmetric MCF} \label{sec3}
In this section, we first review the strong formulation of the axisymmetric MCF with anisotropic surface energy. Based on this formulation, we then develop adaptive moving mesh methods along with their corresponding energy-stable numerical schemes.

\subsection{The strong formulation}
In contrast to the isotropic case, where the axisymmetric formulation of MCF is well-established, explicit derivations of anisotropic axisymmetric models based on a surface energy functional are less commonly documented in the literature. For details on surface diffusion and solid-state dewetting, see \cite{zhao2019sharp, li2024parametric}. For completeness, we provide a simplified derivation of the axisymmetric anisotropic MCF.

Consider the generating curve $\Gamma(t)$ parametrized by the arc-length variable $s$.
We introduce a normal perturbation of the curve in the form
\[
\Gamma^\varepsilon = \mat X + \varepsilon\,\mat X^{(1)},
\]
where $\varepsilon$ is a small perturbation parameter and
$\mat X^{(1)} = \left(r^{(1)}(s),\, z^{(1)}(s)\right) \in \left(\mathrm{Lip}[0,L]\right)^2$.
If $s$ and $\varphi$ are taken as the parameters of the perturbed surface $S^\varepsilon$,
the associated metric tensor is given by
\[
g^\varepsilon
=
\left[\left(r_s^\varepsilon\right)^2 + \left(z_s^\varepsilon\right)^2\right]
\left(r^\varepsilon\right)^2 .
\]
Based on \eqref{eq:AMCF_energy}, the total free energy of the perturbed system can be written as
\begin{equation}\label{eq:energyvar}
    W^{\varepsilon}
    =
    2\pi \int_{\Gamma^{\varepsilon}}
    r^{\varepsilon}\,\gamma\!\left(\theta^{\varepsilon}\right)\,
    \left| \partial_{s} \mat X^{\varepsilon}\right|
    \,\mathrm{d}s .
\end{equation}
The unit tangential and outer normal vectors of the curve $\Gamma$ are defined in \eqref{eq:tau_continue}.
We expand the following quantities at $\varepsilon = 0$:
\begin{subequations}\label{eq:eps-expansion}
\begin{align}
&r^\varepsilon
= r + r^{(1)}\,\varepsilon + O(\varepsilon^2), \label{eq:eps-expansion-a} \\[2mm]
&\left|\partial_{s} \mat X^\varepsilon\right|
= 1 + (\mat\tau \cdot \partial_{s} \mat X^{(1)})\,\varepsilon + O(\varepsilon^2), \label{eq:eps-expansion-b} \\[2mm]
&\gamma(\theta^\varepsilon)
= \gamma(\theta)
+ \gamma'(\theta)\,(\mat n \cdot \partial_{s} \mat X^{(1)})\,\varepsilon
+ O(\varepsilon^2). \label{eq:eps-expansion-c}
\end{align}
\end{subequations}
To compute the first variation of the total energy, we rewrite the energy functional by Taylor's formula
\begin{equation}\label{eq:w_varepsilon}
    W^{\varepsilon} = W + \varepsilon\, W^{(1)} + O(\varepsilon^2),
\end{equation}
where $W^{(1)}$ denotes the coefficient of the $O(\varepsilon)$ term.
Substituting \eqref{eq:eps-expansion-a}-\eqref{eq:eps-expansion-c} into
\eqref{eq:energyvar} and performing integration by parts, we obtain
\begin{equation}\label{eq:2.10}
\begin{aligned}
W^{(1)}
&=
2\pi \int_0^L
\left[
\gamma'(\theta)\left(\mat n \cdot \partial_s \mat X^{(1)}\right) r
+
\gamma(\theta)\left(\mat\tau \cdot \partial_s \mat X^{(1)}\right) r
+
\gamma(\theta) r^{(1)}
\right] \,\mathrm ds \\[2mm]
&=
2\pi \int_0^L
\left[
-
\partial_s\!\left(\gamma(\theta)\mat\tau + \gamma'(\theta)\mat n\right)
\cdot \mat X^{(1)}\, r
-
\left(\gamma(\theta)\mat\tau + \gamma'(\theta)\mat n\right)
\cdot \partial_s \mat X^{(1)}
+
\gamma(\theta) r^{(1)}
\right] \,\mathrm ds .
\end{aligned}
\end{equation}
Using the geometric identities
\[
\partial_s\mat n = \varkappa\,\mat\tau,
\qquad
\partial_s\theta = -\varkappa,
\qquad
\partial_s\mat\tau = -\varkappa\,\mat n,
\qquad
\mat\tau \cdot \partial_s \mat X^{(1)} - r^{(1)}
=
\partial_s z\, \mat X^{(1)} \cdot \mat n ,
\]
we arrive at
\begin{equation}
\begin{aligned}
W^{(1)}
&=
2\pi \int_0^L
\left[
\left(\gamma(\theta) + \gamma''(\theta)\right)\varkappa\,
\left(\mat n \cdot \mat X^{(1)}\right) r
-
\left(\gamma(\theta)\partial_s z + \gamma'(\theta)\partial_s r\right)
\left(\mat X^{(1)} \cdot \mat n\right)
\right]
\,\mathrm ds \\[2mm]
&=
2\pi \int_0^L
\left[
\left(\gamma(\theta)+\gamma''(\theta)\right)\varkappa
-
\frac{
\left(\gamma(\theta)\,\mat n + \gamma'(\theta)\,\mat\tau\right)\cdot \mat e_1
}{\mat X\cdot\mat e_1}
\right]
\left(\mat X^{(1)} \cdot \mat n\right)
\, r\,\mathrm ds .
\end{aligned}
\end{equation}
Similar to \cite{zhao2019sharp}, we define
\begin{equation}\label{eq:aniso-AMCF}
	\mu=
\left(\gamma(\theta)+\gamma''(\theta)\right)\varkappa- \frac{\left(\gamma(\theta)\,\mat n
		+ \gamma'(\theta)\,\mat\tau
		\right)\cdot\mat e_1
	}{\mat X\cdot\mat e_1}.
\end{equation}

We introduce the tangential component of the velocity via the variational mesh redistribution strategy outlined in Lemma~\ref{lemma:tangential}. The resulting coupled system for the anisotropic axisymmetric MCF is given by:
\begin{subequations}\label{eq:anisotropic-AMCF-system}
    \begin{align} 
        &\partial_t \mat X
    =\mu\mat n
  + \frac{P}{\mathcal J}\,
    \left(M\left|\partial_\rho \mat X\right|\right)^{-2}
    \partial_\rho\!\left(M\left|\partial_\rho \mat X\right|\right)\mat\tau ,\label{eq:unified_velocity} \\[4pt]
       &\mu=
\left(\gamma(\theta)+\gamma''(\theta)\right)\varkappa- \frac{\left(\gamma(\theta)\,\mat n
		+ \gamma'(\theta)\,\mat\tau
		\right)\cdot\mat e_1
	}{\mat X\cdot\mat e_1},\label{eq:mu}     \\[4pt]
    &\varkappa = \frac{\partial_{\rho \rho} \mat X\cdot\mat n}{\left|\partial_{\rho}\mat X\right|^2}.
    \end{align}
\end{subequations}

\begin{lemma}
\label{lem:aniso-energy-dissipation}
Let $W(t)$ denote the anisotropic surface energy of the axisymmetric surface
generated by $\Gamma(t)$ governed by \eqref{eq:anisotropic-AMCF-system}.
Then the energy is dissipative in the sense that 
\begin{equation}\label{eqn:aniso_energy_dissipation}
\frac{\mathrm d}{\mathrm dt} W(t)
=
-\,2\pi
\int^1_{0}
\left(\mat X\cdot \mat e_1\right)\,
\mu^{2}\,
\left|\partial_\rho\mat X\right|
\,\mathrm d\rho
\;\le\; 0 .
\end{equation}
\end{lemma}

\begin{proof}
The anisotropic surface energy of the axisymmetric surface generated by
$\Gamma(t)$ can be written as 
\begin{equation}\label{eq:W_def}
W(t)
=
2\pi \int^1_{0}
\left(\mat X\cdot\mat e_1\right)\,
\gamma(\theta)\,
\left|\partial_\rho\mat X\right|
\,\mathrm d\rho .
\end{equation}
Differentiating \eqref{eq:W_def} with respect to $t$, we obtain
\begin{align}
\frac{\mathrm d}{\mathrm dt}W(t)
&=
2\pi\int^1_{0}
\bigg[
\left(\partial_t\mat X\cdot\mat e_1\right)\,\gamma(\theta)
+
\left(\mat X\cdot\mat e_1\right)\,\gamma'(\theta)\,\partial_t\theta
\bigg]
\left|\partial_\rho\mat X\right|
\,\mathrm d\rho
\nonumber\\
&\quad
+
2\pi\int^1_{0}
\left(\mat X\cdot\mat e_1\right)\,\gamma(\theta)\,
\mat \tau\cdot\partial_\rho(\partial_t\mat X)
\,\mathrm d\rho.
\label{eq:dWdt_step1}
\end{align}
By using 
\begin{equation}\label{eq:theta_t_identity}
\partial_t\theta\,\left|\partial_\rho\mat X\right|
=
\partial_\rho\!\left(\partial_t\mat X\cdot\mat n\right)
+
\partial_\rho\theta\,\left(\partial_t\mat X\cdot\mat \tau\right),\quad
\mat \tau\cdot\partial_\rho(\partial_t\mat X)
=
\partial_\rho\left(\partial_t\mat X\cdot\mat \tau\right)
-
\left(\partial_\rho\mat \tau\right)\cdot\partial_t\mat X,
\end{equation}
we rewrite \eqref{eq:dWdt_step1} as
\begin{align}
\frac{\mathrm d}{\mathrm dt}W(t)
&=
2\pi\int^1_{0}
\left(\partial_t\mat X\cdot\mat e_1\right)\,\gamma(\theta)
\left|\partial_\rho\mat X\right|
\,\mathrm d\rho
+
2\pi\int^1_{0}
\left(\mat X\cdot\mat e_1\right)\,\gamma'(\theta)\,
\partial_\rho\!\left(\partial_t\mat X\cdot\mat n\right)
\,\mathrm d\rho
\nonumber\\
&\quad
+
2\pi\int^1_{0}
\left(\mat X\cdot\mat e_1\right)\,\gamma'(\theta)\,\partial_\rho\theta
\left(\partial_t\mat X\cdot\mat \tau\right)
\,\mathrm d\rho
+
2\pi\int^1_{0}
\left(\mat X\cdot\mat e_1\right)\,\gamma(\theta)\,
\partial_\rho\!\left(\partial_t\mat X\cdot\mat \tau\right)
\,\mathrm d\rho
\nonumber\\
&\quad
-
2\pi\int^1_{0}
\left(\mat X\cdot\mat e_1\right)\,\gamma(\theta)\,
\left(\partial_\rho\mat\tau\cdot\partial_t\mat X\right)
\,\mathrm d\rho .
\label{eq:dWdt_step2}
\end{align}
By applying integration by parts to the second and the forth terms of \eqref{eq:dWdt_step2}, we have 
\begin{align}
\frac{\mathrm d}{\mathrm dt}W(t)
&=
2\pi\int^1_{0}
\left(\partial_t\mat X\cdot\mat e_1\right)\,\gamma(\theta)
\left|\partial_\rho\mat X\right|
\,\mathrm d\rho
-
2\pi\int^1_{0}
\partial_\rho\!\left(
\left(\mat X\cdot\mat e_1\right)\gamma'(\theta)
\right)
\left(\partial_t\mat X\cdot\mat n\right)
\,\mathrm d\rho \nonumber\\
&\quad
+
2\pi\int^1_{0}
\left(\mat X\cdot\mat e_1\right)\,\gamma'(\theta)\,\partial_\rho\theta
\left(\partial_t\mat X\cdot\mat \tau\right)
\,\mathrm d\rho -
2\pi\int^1_{0}
\partial_\rho\!\left(
\left(\mat X\cdot\mat e_1\right)\gamma(\theta)
\right)
\left(\partial_t\mat X\cdot\mat \tau\right)
\,\mathrm d\rho
\nonumber\\
&\quad
-
2\pi\int^1_{0}
\left(\mat X\cdot\mat e_1\right)\gamma(\theta)
\left(\partial_\rho\mat \tau\cdot\partial_t\mat X\right)
\,\mathrm d\rho.
\label{eq:dWdt_step4}
\end{align}
By using $\partial_\rho \theta = \varkappa\,|\partial_\rho \mat X|$
and $\partial_\rho\mat\tau=\varkappa\,|\partial_\rho\mat X|\,\mat n$, 
the normal-tangential decomposition \eqref{eq:velocity-decomposition}, together with
\[
\partial_t\mat X\cdot\mat e_1
=
(\partial_t\mat X\cdot\mat\tau)(\mat\tau\cdot\mat e_1)
+
(\partial_t\mat X\cdot\mat n)(\mat n\cdot\mat e_1),
\]
we have 
\begin{align}
\frac{\mathrm d}{\mathrm dt}W(t)
&=
2\pi\int_{\Gamma(t)}
(\mat X\cdot\mat e_1)\left(\gamma(\theta)+\gamma''(\theta)\right)\varkappa
\left(\partial_t\mat X\cdot\mat n\right)
|\partial_\rho\mat X|\,
\mathrm d\rho
\nonumber\\
&\quad
-2\pi\int_{\Gamma(t)}
\left(\gamma(\theta)\mat n+\gamma'(\theta)\mat\tau\right)\cdot\mat e_1
\left(\partial_t\mat X\cdot\mat n\right)
|\partial_\rho\mat X|\,
\mathrm d\rho .
\label{eq:dWdt_final}
\end{align}
By virtue of \eqref{eq:mu}, we can rewrite \eqref{eq:dWdt_final} as
\begin{equation}\label{eq:dWdt_mu}
\frac{\mathrm d}{\mathrm dt}W(t)
=
-\,2\pi\int^1_{0}
\left(\mat X\cdot\mat e_1\right)\,
\mu\,
\left(\partial_t\mat X\cdot\mat n\right)
\left|\partial_\rho\mat X\right|
\,\mathrm d\rho .
\end{equation}
Finally, taking the normal component of the evolution law \eqref{eq:unified_velocity}
and substituting it into \eqref{eq:dWdt_mu}, we obtain
\[
\frac{\mathrm d}{\mathrm dt}W(t)
=
-\,2\pi\int^1_{0}
\left(\mat X\cdot\mat e_1\right)\,
\mu^2
\left|\partial_\rho\mat X\right|
\,\mathrm d\rho
\le 0,
\]
which completes the proof.
\end{proof}

\subsection{Numerical discretizations}\label{sec3-aniso-discretizations}
In this subsection, we intend to construct fully discrete schemes for the anisotropic coupled system
\eqref{eq:anisotropic-AMCF-system}. To this end, 
central finite differences are employed for the spatial discretization, while
BDFk ($k=1,2$) and the Crank-Nicolson methods are used for time integration.
An adaptive tangential redistribution term is incorporated to achieve moving
mesh refinement in regions with strong anisotropic effects.
Below, we present the BDFk and Crank-Nicolson methods for the coupled system \eqref{eq:anisotropic-AMCF-system},
where the resulting nonlinear systems are solved
by a Picard-type iteration.

\begin{itemize}
    \item [(i)]\textbf{BDFk methods ($k=1,2$) for the anisotropic system (referred to the ANISO-A-BDFk-FDMs)}:
\begin{subequations}\label{eq:anisotropic-BDFk}
\begin{align}
&\frac{1}{\Delta t}\sum_{\rho=0}^{k}
\alpha_\rho\,\mat X_i^{\,n+1-\rho}
=
\mu_i^{\,n+1}\,\mat n_i^{\,n+1}
+
\frac{P_i^{n+1}}{\mathcal J}
\left(
  M_i^{\,n+1}\left|\delta_\rho\mat X_i^{\,n+1}\right|
\right)^{-2}
\delta_\rho\!\left(
  M_i^{\,n+1}\left|\delta_\rho\mat X_i^{\,n+1}\right|
\right)\mat\tau_i^{\,n+1},
\\[4pt]
&\mu_i^{\,n+1}
=
\left(\gamma(\theta_i^{\,n+1,m})+\gamma''(\theta_i^{\,n+1,m})\right)\,
\varkappa_i^{\,n+1}
-
\frac{
  \left(
    \gamma(\theta_i^{\,n+1,m})\,\mat n_i^{\,n+1}
    + \gamma'(\theta_i^{\,n+1,m})\,\mat\tau_i^{\,n+1}
  \right)\cdot\mat e_1
}{
  \mat X_i^{\,n+1}\cdot\mat e_1
},
\\[4pt]
&\varkappa_i^{\,n+1}
=
\frac{
  \delta_{\rho\rho}\mat X_i^{\,n+1}\cdot \mat n_i^{\,n+1}
}{
  \left|\delta_\rho\mat X_i^{\,n+1}\right|^2
}.
\end{align}
\end{subequations}

\item [(ii)]\textbf{Crank-Nicolson method for the anisotropic system (referred to the ANISO-A-CN-FDM)}: 
\begin{subequations}\label{eq:anisotropic-CN}
		\begin{align}
			&\frac{\mat X_i^{n+1}-\mat X_i^{n}}{\Delta t}
			=
			\widehat{\mu}_i^{\,n+\frac12}\,
			\widehat{\mat n}_i^{\,n+\frac12}
			+
			\widehat{\mathcal B}_i^{n+\frac12}	\widehat{\mat \tau}_i^{\,n+\frac12},
			\\[4pt]
			& \mathcal B_i^{n+1}=	\frac{P_i^{n+1}}{\mathcal J}
			\left(
			M_i^{n+1}\left|\delta_\rho\mat X_i^{n+1}\right|
			\right)^{-2}
			\delta_\rho\!\left(
			M_i^{n+1}\left|\delta_\rho\mat X_i^{n+1}\right|
			\right),\\
			&\mu_i^{\,n+1}
			=
			\left(\gamma(\theta_i^{\,n+1,m})+\gamma''(\theta_i^{\,n+1,m})\right)\,
			\varkappa_i^{\,n+1}
			-
			\frac{
				\left(
				\gamma(\theta_i^{\,n+1,m})\,\mat n_i^{\,n+1}
				+ \gamma'(\theta_i^{\,n+1,m})\,\mat\tau_i^{\,n+1}
				\right)\cdot\mat e_1
			}{
				\mat X_i^{\,n+1}\cdot\mat e_1
			},
			\\[4pt]
			&\varkappa_i^{\,n+1}
			=
			\frac{
				\delta_{\rho\rho}\mat X_i^{\,n+1}\cdot \mat n_i^{\,n+1}
			}{
				\left|\delta_\rho\mat X_i^{\,n+1}\right|^2
			}.
		\end{align}
	\end{subequations}
\end{itemize}

Analogous to the isotropic case, we treat the nonlinearity in \eqref{eq:anisotropic-BDFk} and \eqref{eq:anisotropic-CN} by a Picard-type
iteration. Given the iteration \(\left(\mat X_i^{n+1,m},\, \mu_i^{n+1,m},\, \varkappa_i^{n+1,m}\right)\) or \(\left(\mat X_i^{n+1,m}, \mathcal B^{n+1,m}, \mu_i^{n+1,m},\, \varkappa_i^{n+1,m}\right)\), the next iteration is obtained by solving the following linearized systems.

\begin{itemize}
    \item [(i)]\textbf{Picard-iteration method for the ANISO-A-BDFk-FDMs}:
\begin{subequations}\label{eq:anisotropic-BDFk-Picard}
\begin{align}
&\frac{\alpha_0 \mat X_i^{\,n+1,m+1}+\sum_{p=1}^{k} \alpha_p \mat X_i^{\,n+1-p} } {\Delta t}
=
\mu_i^{\,n+1,m+1}\,\mat n_i^{\,n+1,m}
+
\frac{P_i^{n+1}}{\mathcal J}
\left(
  M_i^{\,n+1,m}\left|\delta_\rho\mat X_i^{\,n+1,m}\right|
\right)^{-2}
\delta_\rho\!\left(
  M_i^{\,n+1,m}\left|\delta_\rho\mat X_i^{\,n+1,m}\right|
\right)\mat\tau_i^{\,n+1,m},
\\[4pt]
&\mu_i^{\,n+1,m+1}
=
\left(
  \gamma\!\left(\theta_i^{\,n+1,m}\right)
  +\gamma''\!\left(\theta_i^{\,n+1,m}\right)
\right)\varkappa_i^{\,n+1,m+1}
-
\frac{\left(\gamma\!\left(\theta_i^{\,n+1,m}\right)\,\mat n_i^{\,n+1,m}
    + \gamma'\!\left(\theta_i^{\,n+1,m}\right)\,\mat\tau_i^{\,n+1,m}
  \right)\cdot\mat e_1
}{
  \mat X_i^{\,n+1,m}\cdot\mat e_1
},
\\[4pt]
&\varkappa_i^{\,n+1,m+1}
=
\frac{
  \delta_{\rho\rho}\mat X_i^{\,n+1,m+1}\cdot \mat n_i^{\,n+1,m}
}{
  \left|\delta_\rho\mat X_i^{\,n+1,m}\right|^2
}.
\end{align}
\end{subequations}

\item [(ii)]\textbf{Picard-iteration method for the ANISO-A-CN-FDM}:
\begin{subequations}\label{eq:anisotropic-CN-Picard}
\begin{align}
&\frac{\mat X_i^{n+1,m+1}-\mat X_i^{n}}{\Delta t}
=
\widehat{\mu}_i^{\,n+\frac12,m+1}\,
\widehat{\mat n}_i^{\,n+\frac12,m}
+
\widehat{\mathcal B}_i^{\,n+\frac12,m+1}\,
\widehat{\mat \tau}_i^{\,n+\frac12,m},
\\[4pt]
&\mathcal B_i^{n+1,m+1}
=
\frac{P_i^{n+1,m}}{\mathcal J}
\left(
M_i^{n+1,m}\left|\delta_\rho\mat X_i^{n+1,m}\right|
\right)^{-2}
\delta_\rho\!\left(
M_i^{n+1,m}\left|\delta_\rho\mat X_i^{n+1,m}\right|
\right),
\\[4pt]
&\mu_i^{\,n+1,m+1}
=
\left(\gamma\!\left(\theta_i^{\,n+1,m}\right)+\gamma''\!\left(\theta_i^{\,n+1,m}\right)\right)\,
\varkappa_i^{\,n+1,m+1}
-
\frac{
\left(
\gamma\!\left(\theta_i^{\,n+1,m}\right)\,\mat n_i^{\,n+1,m}
+ \gamma'\!\left(\theta_i^{\,n+1,m}\right)\,\mat\tau_i^{\,n+1,m}
\right)\cdot\mat e_1
}{
\mat X_i^{\,n+1,m}\cdot\mat e_1
},
\\[4pt]
&\varkappa_i^{\,n+1,m+1}
=
\frac{
\delta_{\rho\rho}\mat X_i^{\,n+1,m+1}\cdot \mat n_i^{\,n+1,m}
}{
\left|\delta_\rho\mat X_i^{\,n+1,m}\right|^2
}.
\end{align}
\end{subequations}
\end{itemize}

\subsection{Energy-stable methods}

The same Lagrange multiplier framework can be applied to the anisotropic
axisymmetric MCF considered in this work.
In this case, the evolution of the generating curve is given by
\begin{subequations}\label{eq:anisotropic-AMCF-LM-system}
\begin{align}
&\partial_t \mat X
=
\left(\mu-\lambda(t)\,\mu\right)\,\mat n
+\frac{P}{\mathcal J}\,
\left(M\left|\partial_\rho \mat X\right|\right)^{-2}
\partial_\rho\!\left(M\left|\partial_\rho \mat X\right|\right)\mat\tau,
\label{eq:aniso-LM-evolution}\\[4pt]
&\mu
=
\left(\gamma(\theta)+\gamma''(\theta)\right)\varkappa
-\frac{\left(\gamma(\theta)\,\mat n+\gamma'(\theta)\,\mat\tau\right)\cdot\mat e_1}
{\mat X\cdot\mat e_1},
\label{eq:aniso-mu}\\[4pt]
&\varkappa
=
\frac{\partial_{\rho\rho}\mat X\cdot\mat n}{\left|\partial_\rho \mat X\right|^2},
\label{eq:aniso-kappa}\\[4pt]
&\frac{\mathrm d}{\mathrm dt}W(t)=
-\,2\pi\int^1_{0}
		r\, \mu^{2}\,\left|\partial_\rho \mat X\right|\mathrm d \rho .
\label{eq:aniso-LM-energy}
\end{align}
\end{subequations}

\begin{lemma}
\label{lem:anisotropic-lm-consistency}
The anisotropic Lagrange multiplier system \eqref{eq:anisotropic-AMCF-LM-system}
is consistent with the original anisotropic system
\eqref{eq:anisotropic-AMCF-system}.
\end{lemma}

\begin{proof}
We denote the normal velocity by
\( V_n := \partial_t \mat X \cdot \mat n \).
For anisotropic axisymmetric MCF, the surface energy
\( W(t) \) satisfies
\[
\frac{\mathrm d}{\mathrm dt}W(t)
=
-\,2\pi \int^1_{0} r\,\mu\,V_n \,\left|\partial_\rho \mat X\right|\mathrm d \rho,
\]
where \( \mu \) is the anisotropic chemical potential defined in
\eqref{eq:anisotropic-AMCF-system}.
In the Lagrange multiplier formulation
\eqref{eq:anisotropic-AMCF-LM-system}, the normal velocity is
\( V_n = (1-\lambda(t))\,\mu \).
Substituting this relation gives
\[
\frac{\mathrm d}{\mathrm dt}W(t)
=
-\,2\pi \int^1_{0} r\,\mu^2 \,\left|\partial_\rho \mat X\right|\mathrm d \rho
+
2\pi\,\lambda(t)
\int^1_{0} r\,\mu^2 \,\left|\partial_\rho \mat X\right|\mathrm d \rho.
\]
Comparing with the energy dissipation law
\eqref{eq:aniso-LM-energy}, we can obtain
\( \lambda(t) \equiv 0 \).
This proves the consistency of the anisotropic Lagrange multiplier
formulation.
\end{proof}

Similar to the isotropic case, based on the anisotropic Lagrange multiplier
formulation, we now construct fully discrete finite difference schemes for the
anisotropic system. 
\begin{itemize}
\item[(i)]\textbf{BDFk schemes ($k=1,2$) for the anisotropic system with
Lagrange multiplier (referred to the ANISO-A-LM-BDFk-FDMs)}:
\begin{subequations}\label{eq:anisotropic-BDFk-LM}
\begin{align}
&\frac{1}{\Delta t}\sum_{\rho=0}^{k}
\alpha_\rho\,\mat X_i^{\,n+1-\rho}
=
\left(
\mu_i^{\,n+1}-\lambda^{n+1}\,\mu_i^{\,n+1}
\right)\mat n_i^{\,n+1}
+
\frac{P_i^{\,n+1}}{\mathcal J}
\left(
M_i^{\,n+1}\left|\delta_\rho\mat X_i^{\,n+1}\right|
\right)^{-2}
\delta_\rho\!\left(
M_i^{\,n+1}\left|\delta_\rho\mat X_i^{\,n+1}\right|
\right)\mat\tau_i^{\,n+1},
\label{eq:aniso-BDFk-LM-evolution}
\\[4pt]
&\mu_i^{\,n+1}
=
\left(
\gamma\!\left(\theta_i^{\,n+1}\right)
+\gamma''\!\left(\theta_i^{\,n+1}\right)
\right)\varkappa_i^{\,n+1}
-
\frac{
\left(
\gamma\!\left(\theta_i^{\,n+1}\right)\,\mat n_i^{\,n+1}
+\gamma'\!\left(\theta_i^{\,n+1}\right)\,\mat\tau_i^{\,n+1}
\right)\cdot\mat e_1
}{
\mat X_i^{\,n+1}\cdot\mat e_1
},
\label{eq:aniso-BDFk-LM-mu}
\\[4pt]
&\varkappa_i^{\,n+1}
=
\frac{
\delta_{\rho\rho}\mat X_i^{\,n+1}\cdot \mat n_i^{\,n+1}
}{
\left|\delta_\rho\mat X_i^{\,n+1}\right|^2
},
\label{eq:aniso-BDFk-LM-kappa}
\\[4pt]
&\frac{1}{\Delta t}\sum_{\rho=0}^{k}
\alpha_\rho\,W^{\,n+1-\rho}
=
-\,2\pi\int^1_{0}
\left(\mat X^{n+1}\cdot\mat e_1\right)\,\left(
\mu^{\,n+1}\right)^{2}\,
\left|\partial_\rho \mat X^{n+1}\right|\,
\mathrm d\rho .
\label{eq:aniso-BDFk-LM-energy}
\end{align}
\end{subequations}

\item [(ii)]\textbf{Crank-Nicolson scheme for the anisotropic system with Lagrange multiplier
(referred to as the ANISO-A-LM-CN-FDM)}:
\begin{subequations}\label{eq:anisotropic-CN-LM}
\begin{align}
&\frac{\mat X_i^{n+1}-\mat X_i^{n}}{\Delta t}
=
\left(
\widehat{\mu}_i^{\,n+\frac12}
-
\lambda^{n+\frac12}\,
\widehat{\mu}_i^{\,n+\frac12}
\right)
\widehat{\mat n}_i^{\,n+\frac12}
+
\widehat{\mathcal B}_i^{\,n+\frac12}\,
\widehat{\mat\tau}_i^{\,n+\frac12},
\label{eq:aniso-CN-LM-evolution}
\\[4pt]
&\mathcal B_i^{n+1}
=
\frac{P_i^{n+1}}{\mathcal J}
\left(
M_i^{n+1}\left|\delta_\rho\mat X_i^{n+1}\right|
\right)^{-2}
\delta_\rho\!\left(
M_i^{n+1}\left|\delta_\rho\mat X_i^{n+1}\right|
\right),
\label{eq:aniso-CN-LM-B}
\\[4pt]
&\mu_i^{\,n+1}
=
\left(\gamma\!\left(\theta_i^{\,n+1}\right)+\gamma''\!\left(\theta_i^{\,n+1}\right)\right)\,
\varkappa_i^{\,n+1}
-
\frac{
\left(
\gamma\!\left(\theta_i^{\,n+1}\right)\,\mat n_i^{\,n+1}
+\gamma'\!\left(\theta_i^{\,n+1}\right)\,\mat\tau_i^{\,n+1}
\right)\cdot\mat e_1
}{
\mat X_i^{\,n+1}\cdot\mat e_1
},
\label{eq:aniso-CN-LM-mu}
\\[4pt]
&\varkappa_i^{\,n+1}
=
\frac{
\delta_{\rho\rho}\mat X_i^{\,n+1}\cdot \mat n_i^{\,n+1}
}{
\left|\delta_\rho\mat X_i^{\,n+1}\right|^2
},
\label{eq:aniso-CN-LM-kappa}
\\[4pt]
&\frac{W^{n+1}-W^{n}}{\Delta t}
=
-\,2\pi\int^1_{0}
\left(\widehat{\mat X}_i^{n+\frac12}\cdot\mat e_1\right)\,
\left(\widehat{\mu}_i^{\,n+\frac12}\right)^{2}\,
\left|\partial_\rho \widehat{\mat X}_i^{\,n+\frac12}\right|\,
\mathrm d\rho .
\label{eq:aniso-CN-LM-energy}
\end{align}
\end{subequations}
\end{itemize}

\begin{thm}
\label{thm:aniso-energy}
The fully discrete schemes
\eqref{eq:anisotropic-BDFk-LM} and \eqref{eq:anisotropic-CN-LM}
are unconditionally energy-stable, in the sense that
\[
W^{n+1} \le W^{n}, \qquad n \ge 0,
\]
where \( W^n \) denotes the discrete anisotropic surface energy.
\end{thm}

\begin{proof}
For the ANISO-A-LM-BDF1-FDM, it follows from
\eqref{eq:anisotropic-BDFk-LM} that
\[
\frac{W^{n+1}-W^{n}}{\Delta t}
=
-\,2\pi
\int^1_{0}
\left(\mat X^{n+1}\cdot\mat e_1\right)
\left(\mu^{n+1}\right)^{2}
\left|\partial_\rho \mat X^{n+1}\right|
\,\mathrm d\rho .
\]
Under the assumptions that
\( \mat X^{n+1}\cdot\mat e_1 > 0 \) and
\( \left|\partial_\rho \mat X^{n+1}\right| > 0 \),
the right-hand side is non-positive, which immediately implies
\( W^{n+1} \le W^{n} \).
In addition, for the ANISO-A-LM-BDF2-FDM, the first time step is computed
by the ANISO-A-LM-BDF1-FDM, yielding \( W^{1} \le W^{0} \).
By using \eqref{eq:anisotropic-BDFk-LM}, we further obtain
\[
\frac{1}{\Delta t}
\left(
\frac{3}{2}W^{n+1}-2W^{n}+\frac{1}{2}W^{n-1}
\right)
=
-\,2\pi
\int^1_{0}
\left(\mat X^{n+1}\cdot\mat e_1\right)
\left(\mu^{n+1}\right)^{2}
\left|\partial_\rho \mat X^{n+1}\right|
\,\mathrm d\rho
\le 0 .
\]
This inequality implies
\[
\frac{3}{2}\left(W^{n+1}-W^{n}\right)
\le
\frac{1}{2}\left(W^{n}-W^{n-1}\right).
\]
We now proceed by mathematical induction.
Since \( W^{1} \le W^{0} \), assume that
\(
W^{j} \le W^{j-1}
\)
for \( j=1,2,\ldots,n \).
Then the above inequality immediately implies
\( W^{n+1} \le W^{n} \).
Therefore, the energy dissipation property holds for all \( n \ge 1 \).

Finally, for the ANISO-A-LM-CN-FDM,
\eqref{eq:anisotropic-CN-LM} directly implies
\[
\frac{W^{n+1}-W^{n}}{\Delta t}
=
-\,2\pi
\int^1_{0}
\left(\widehat{\mat X}^{n+\frac12}\cdot\mat e_1\right)
\left(\mu^{n+\frac12}\right)^{2}
\left|\partial_\rho \mat X^{n+\frac12}\right|
\,\mathrm d\rho
\le 0 ,
\]
and hence \( W^{n+1} \le W^{n} \).
This completes the proof.
\end{proof}

\section{Design of the monitor function}\label{sec4}
The performance of adaptive moving mesh methods is governed by the choice of the monitor function, as it dictates how mesh points are redistributed in response to evolving geometric features. An ill-suited monitor function can lead to excessive clustering or insufficient resolution, thereby compromising both accuracy and robustness. Consequently, this section focuses on the systematic design of monitor functions based on the geometric properties of evolving curves.

In the absence of reliable error estimates for the numerical approximation $\Gamma_h(t)$, the construction of such monitor functions is typically guided by geometric interpolation error analysis. The primary objective is to employ equidistribution to rigorously control the deviation between the smooth curve and its piecewise linear approximation. Following this paradigm, Mackenzie et al.~\cite{mackenzie2019adaptive} utilized the maximal distance between the curve and its polygonal interpolant as an error metric to derive a curvature-based monitor function.Building upon this framework, we adopt the same error criterion but conduct a more granular analysis of its local structure. For sufficiently smooth curves, retaining higher-order terms in the local expansion yields a sharper error estimate. This result provides a solid theoretical justification for incorporating higher-order geometric information, specifically the spatial variation of curvature, into the monitor function design.

To proceed, in Fig.\ref{fig:monitor_fun} we let $\Gamma$ be a smooth planar curve and consider an arc segment $\Gamma_i \subset \Gamma$ with endpoints $\mat X_i$ and $\mat X_{i+1}$.
Denote by $\Gamma_{h,i}$ the straight line segment connecting these two points. We define the geometric error at a point $\mat X \in \Gamma_i$ as its orthogonal distance to $\Gamma_{h,i}$.
Since this distance is invariant under rigid motions, we may assume, without loss of generality, that $\mat X_i$ is located at the origin and that $\Gamma_{h,i}$ is aligned with the horizontal axis.
Under this assumption, the curve segment $\Gamma_i$ can locally be represented as the graph of a smooth function $\overline{\Gamma}(\overline{x})$ over the interval $[0,h_i]$, where $h_i = \| \mat X_{i+1} - \mat X_i \|$.
The maximal deviation between $\Gamma_i$ and its linear interpolant is therefore given by
\[
\max_{0 \le \overline{x} \le h_i} \left| \overline{\Gamma}(\overline{x}) \right|.
\]
   \begin{figure} [H]
		\centering
		\includegraphics[width=0.8\linewidth]{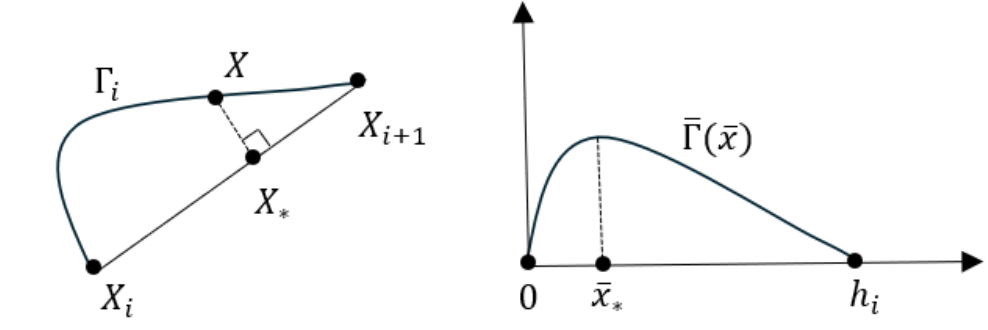}
		\caption{Left: Segment of a smooth curve $\Gamma_i$ and the interpolating linear approximation connecting mesh points $X_i$ and $X_{i+1}$. The geometric error at point $X$ is defined as the orthogonal distance from $X$ to $X_*$. Right: The translated and rotated segment is transformed into the graph $\overline{\Gamma}(\overline{x})$. The maximal distance corresponds to the absolute maximum value at the critical point $\overline{x}_* \in [0, h_i]$.}
		\label{fig:monitor_fun}
	\end{figure}
Assume that this maximum is attained at an interior point $\overline{x}_* \in (0,h_i)$. At this critical point, the tangent to the curve is parallel to the chord $\Gamma_{h,i}$, which implies the condition $\overline{\Gamma}'(\overline{x}_*) = 0$.
A local Taylor expansion of $\overline{\Gamma}(\overline{x})$ about $\overline{x}_*$, evaluated at the endpoint $\overline{x}=0$ (where $\overline{\Gamma}(0)=0$), yields the asymptotic estimate:
\begin{equation}
	\overline{\Gamma}(\overline{x}_*) \approx \frac{\overline{x}_*^2}{2}\overline{\Gamma}''(\overline{x}_*) - \frac{\overline{x}_*^3}{6}\overline{\Gamma}'''(\overline{x}_*).
	\label{eq:taylor_expansion}
\end{equation}
Taking the absolute value and retaining terms up to the third order, we can analyze the error components.
First, recall that the curvature magnitude for a planar graph is given by 
\begin{equation} \label{eq:kappa_gamma}
    \left|\varkappa\right| = \frac{\left|\overline{\Gamma}''\right|}{\left(1 + \left(\overline{\Gamma}'\right)^2\right)^{\frac32}}.
\end{equation}
At $\overline{x}_*$, the vanishing first derivative simplifies \eqref{eq:kappa_gamma} to:
\begin{equation}
	\left|\overline{\Gamma}''(\overline{x}_*)\right| = \left|\varkappa(\overline{x}_*)\right|.
	\label{eq:second_order_relation}
\end{equation}
To estimate the third-order term, let the curve be parametrized by arc-length $s$ as $\mat X(s)$.
Differentiation of the Frenet-Serret formulas for planar curves, $\mat X''(s) = \varkappa(s)\mathbf{n}(s)$, leads to:
\begin{equation}
	\mat X'''(s) = \partial_s \varkappa(s)\mathbf{n}(s) - \varkappa^2(s)\mathbf{t}(s),
\end{equation}
where $\mathbf{t}(s)$ and $\mathbf{n}(s)$ denote the unit tangent and normal vectors, respectively.
Using the triangle inequality, we obtain the bound:
\begin{equation}
	\left|\mat X'''(s)\right| \le \left|\partial_s \varkappa(s)\right| + \varkappa^2(s).
\end{equation}
Under the assumption of locally small slopes, the third derivative in the graph representation, $\overline{\Gamma}'''(\overline{x}_*)$, is of the same order as the intrinsic derivative $\mat X'''(s)$.
Combining this with (\ref{eq:taylor_expansion}) and (\ref{eq:second_order_relation}), and noting that $\overline{x}_* \le h_i$, we derive the following upper bound for the local geometric interpolation error:
\begin{equation}
	\max_{\overline{x} \in [0, h_i]} \left|\overline{\Gamma}(\overline{x})\right| \le C_1 h_i^2 \left|\varkappa\right| + C_2 h_i^3 \left|\partial_s \varkappa\right| + C_3 h_i^3 \varkappa^2,
	\label{eq:error_bound}
\end{equation}
where $C_1, C_2,$ and $C_3$ are positive constants depending only on the regularity of the curve.
This estimate highlights the hierarchical composition of the error. In regions characterized by high curvature or rapid curvature variation (large $|\partial_s \varkappa|$), higher-order geometric terms become dominant. 

Motivated by the local error bound derived in (\ref{eq:error_bound}), we propose a novel adaptive monitor function designed to control the geometric approximation error. Since our analysis reveals that the error is dominated by terms involving the curvature magnitude $\varkappa$, its derivative $\partial_s \varkappa$, and the square of the curvature $\varkappa^2$, we construct the monitor function as follows:\begin{equation}M(s,t) = M_{floor} + \sqrt{a |\varkappa| + b |\partial_s \varkappa| + c \varkappa^2},\label{eq:new_monitor}\end{equation}where $a, b,$ and $c$ are positive weighting parameters that balance the contribution of each geometric feature. Similar to the approach in \cite{mackenzie2019adaptive}, a positive regularization parameter $M_{floor}$ is included to prevent mesh starvation in flat regions where the curvature and its derivative vanish, thereby ensuring the stability of the numerical scheme.
\begin{rem}\label{rem:monitor_function}
    The proposed monitor function offers distinct advantages over standard curvature-based methods. First, by explicitly including the derivative term $|\partial_s \varkappa|$, the method automatically refines the mesh in regions of rapid curvature variation—such as transition zones near inflection points—rather than focusing solely on regions of high curvature magnitude. This directly addresses the higher-order geometric error components identified in our analysis. Second, the combination of these terms aligns with the equidistribution principle, facilitating a more efficient minimization of the global geometric interpolation error compared to uniform arc-length or standard curvature-based meshes.
\end{rem}

\begin{rem}
All adaptive numerical schemes in this paper are based on fully implicit formulations.
While iterative solvers are required, the iteration counts observed in practice are moderate, and the resulting computational overhead is limited due to the low cost of each finite-difference-based iteration.
Numerical experiments with linearized schemes indicate comparable performance, with adaptivity continuing to provide clear benefits.
In view of the inherent stability of implicit schemes and for brevity, we do not pursue linearized methods further, but emphasize that they remain a feasible alternative.
\end{rem}

\section{Numerical Experiments}\label{sec5}

In this section, we present a series of numerical experiments to validate the effectiveness and robustness of the proposed numerical schemes for both isotropic and anisotropic axisymmetric MCF.
The numerical results are designed to examine key properties of the methods, including convergence behavior, energy stability,  and the performance of adaptive mesh redistribution. 
By comparing the numerical solutions with exact solutions or highly resolved reference solutions, we systematically assess the accuracy and efficiency of different time-stepping schemes, BDFk $(k = 1,2)$, and the Crank-Nicolson method, and demonstrate the advantages of the adaptive approach in resolving geometric features.
The initial computational mesh is generated using the de Boor algorithm to ensure a smooth and well-distributed parametrization of the initial curve \cite{huang1994moving,huang2010adaptive,mackenzie2019adaptive}.
The stopping criterion for the iteration is set to tol = $10^{-8}$, and we select the relax time constant $\mathcal{J}=0.5$.

When testing the errors and corresponding convergence rates, a sequence of refined meshes is employed to measure the numerical errors. Let \(\left(h_\ell,\Delta t_\ell\right)\) denote the spatial mesh size and time step at refinement level \(\ell\). Let \(\mat U_{h_\ell}^{N_\ell}\) be the corresponding numerical solution at \(t=T\), where \(N_\ell=T/\Delta t_\ell\) and \(\mat U\) denotes the combined discrete unknowns, defined by
\[
\mat U =
\left\{
\begin{array}{ll}
	\left(\mat X,\, V,\, \varkappa\right),
	& \text{ISO-A-BDFk-FDMs}, \\[4pt]
	\left(\mat X,\, V,\, \mathcal{B},\, \varkappa\right),
	& \text{ISO-A-CN-FDM}, \\[4pt]
	\left(\mat X,\, V,\, \varkappa,\, \mu\right),
	& \text{ANISO-A-BDFk-FDMs}, \\[4pt]
	\left(\mat X,\, V,\, \mathcal{B},\, \varkappa,\, \mu\right),
	& \text{ANISO-A-CN-FDM}.
\end{array}
\right.
\]
The numerical error is defined by comparing two consecutive refinement levels,
\begin{align} \label{eq:error_definition}
	e_\ell^h
	=
	\left\|
	\mat U_{h_\ell}^{N_\ell}
	-
	\mat U_{h_{\ell+1}}^{N_{\ell+1}}
	\right\|_{\infty},
\end{align}
and the convergence order is computed as
\begin{align} \label{eq:order_definition}
	\mathrm{order}_\ell
	=
	\frac{
		\log\!\left(
		e_\ell^h / e_{\ell+1}^h
		\right)
	}{
		\log\!\left(
		\Delta t_\ell / \Delta t_{\ell+1}
		\right)
	}.
\end{align}
To facilitate the simultaneous measurement of temporal and spatial convergence rates, a coupled refinement strategy between the spatial mesh size and the time step is imposed. Accordingly, the pair \(\left(h_\ell,\Delta t_\ell\right)\) is updated as
\begin{align} \label{eq:refinement_strategy}
	\left(h_{\ell+1},\Delta t_{\ell+1}\right)=
	\begin{cases}
		\left(h_\ell/2,\;\Delta t_\ell/4\right), & \text{BDF1 method},\\[4pt]
		\left(h_\ell/2,\;\Delta t_\ell/2\right), & \text{BDF2 method},\\[4pt]
		\left(h_\ell/2,\;\Delta t_\ell/2\right), & \text{Crank-Nicolson method}.
	\end{cases}
\end{align}

To better showcase the advantages of the adaptive methods, we also introduce the corresponding finite-difference discretizations without explicit adaptive tangential mesh redistribution. The methods include 

\begin{itemize}
	\item [(i)]\textbf{BDFk methods ($k=1,2$) for the isotropic system 
		(referred to as the ISO-BDFk-FDMs)}:
	\begin{subequations}\label{eq:isotropic-BDFk-uniform}
		\begin{align}
			&\frac{1}{\Delta t}\sum_{\rho=0}^{k}
			\alpha_\rho\,\mat X_i^{\,n+1-\rho}\cdot\mat n_i^{n+1}
			=
			V_i^{n+1},
			\\[4pt]
			&V_i^{n+1}
			=
			\varkappa_i^{n+1}
			-
			\frac{\mat n_i^{n+1}\cdot\mat e_1}
			{\mat X_i^{n+1}\cdot\mat e_1},
			\\[4pt]
			&\varkappa_i^{\,n+1}\mat{n}_i^{\,n+1}
			= \frac{\delta_{\rho\rho}\mat X_i^{\,n+1}}{\left|\delta_{\rho}\mat X_i^{\,n+1}\right|^2}
			-
			\frac{\mat{\tau}_i^{\,n+1}\cdot 
				\delta_{\rho\rho}\mat X_i^{\,n+1}}
			{\left|\delta_{\rho}\mat X_i^{\,n+1}\right|^{3}}
			\,\delta_{\rho}\mat X_i^{\,n+1}.
		\end{align}
	\end{subequations}
	
	\item [(ii)]\textbf{Crank-Nicolson method for the isotropic system 
		(referred to as the ISO-CN-FDM)}:
	\begin{subequations}\label{eq:isotropic-CN-uniform}
		\begin{align}
			&\frac{\mat X_i^{n+1}-\mat X_i^{n}}{\Delta t}\cdot
			\widehat{\mat n}_i^{\,n+\frac12}
			=
			\widehat{V}_i^{n+\frac12},
			\\[4pt]
			&V_i^{n+1}
			=
			\varkappa_i^{n+1}
			-
			\frac{
				\mat n_i^{\,n+1}\cdot\mat e_1}
			{\mat X_i^{\,n+1}\cdot\mat e_1},
			\\[4pt]
			&\widehat{\varkappa}_i^{\,n+\frac{1}{2}}\widehat{\mat{n}}_i^{\,n+\frac{1}{2}}
			= \frac{\delta_{\rho\rho}\widehat{\mat X_i}^{\,n+\frac{1}{2}}}{\left|\delta_{\rho}\widehat{\mat X_i}^{\,n+\frac{1}{2}}\right|^2}
			-
			\frac{\mat{\widehat{\tau}}_i^{\,n+\frac{1}{2}}\cdot 
				\delta_{\rho\rho}\widehat{\mat X_i}^{\,n+\frac{1}{2}}}
			{\left|\delta_{\rho}\widehat{\mat X_i}^{\,n+\frac{1}{2}}\right|^{3}}
			\,\delta_{\rho}\widehat{\mat X_i}^{\,n+\frac{1}{2}}.
		\end{align}
	\end{subequations}
	
		\item [(iii)]
		\textbf{BDFk methods ($k=1,2$) for the anisotropic system 
			(referred to as the ANISO-BDFk-FDMs)}:
		\begin{subequations}\label{eq:aniso-BDFk-BGN}
			\begin{align}
				&\frac{1}{\Delta t}\sum_{\rho=0}^{k}
				\alpha_\rho\,\mat X_i^{\,n+1-\rho}\cdot\mat n_i^{\,n+1}
				=
				\mu_i^{\,n+1},
				\\[4pt]
				&\mu_i^{\,n+1}
				=
				\left(\gamma(\theta_i^{\,n+1,m})+\gamma''(\theta_i^{\,n+1,m})\right)\,
				\varkappa_i^{\,n+1}
				-
				\frac{
					\left(
					\gamma(\theta_i^{\,n+1,m})\,\mat n_i^{\,n+1}
					+ \gamma'(\theta_i^{\,n+1,m})\,\mat\tau_i^{\,n+1}
					\right)\cdot\mat e_1
				}{
					\mat X_i^{\,n+1}\cdot\mat e_1
				},
				\\[4pt]
				&\varkappa_i^{\,n+1}\mat{n}_i^{\,n+1}
				= \frac{\delta_{\rho\rho}\mat X_i^{\,n+1}}{\left|\delta_{\rho}\mat X_i^{\,n+1}\right|^2}
				-
				\frac{\mat{\tau}_i^{\,n+1}\cdot 
					\delta_{\rho\rho}\mat X_i^{\,n+1}}
				{\left|\delta_{\rho}\mat X_i^{\,n+1}\right|^{3}}
				\,\delta_{\rho}\mat X_i^{\,n+1}.
			\end{align}
		\end{subequations}
		
		\item [(iv)]\textbf{Crank-Nicolson method for the anisotropic system 
			(referred to as the ANISO-CN-FDM)}:
		\begin{subequations}\label{eq:anisotropic-CN-standard}
			\begin{align}
				&\frac{\mat X_i^{n+1}-\mat X_i^{n}}{\Delta t}\cdot\widehat{\mat n}_i^{\,n+\frac12}
				=
				\widehat{\mu}_i^{\,n+\frac12},
				\\[4pt]
				&\mu_i^{\,n+1}
				=
				\left(\gamma(\theta_i^{\,n+1,m})+\gamma''(\theta_i^{\,n+1,m})\right)\,
				\varkappa_i^{\,n+1}
				-
				\frac{
					\left(
					\gamma(\theta_i^{\,n+1,m})\,\mat n_i^{\,n+1}
					+ \gamma'(\theta_i^{\,n+1,m})\,\mat\tau_i^{\,n+1}
					\right)\cdot\mat e_1
				}{
					\mat X_i^{\,n+1}\cdot\mat e_1
				},
				\\[4pt]
				&\widehat{\varkappa}_i^{\,n+\frac{1}{2}}\widehat{\mat{n}}_i^{\,n+\frac{1}{2}}
				= \frac{\delta_{\rho\rho}\widehat{\mat X_i}^{\,n+\frac{1}{2}}}{\left|\delta_{\rho}\widehat{\mat X_i}^{\,n+\frac{1}{2}}\right|^2}
				-
				\frac{\mat{\widehat{\tau}}_i^{\,n+\frac{1}{2}}\cdot 
					\delta_{\rho\rho}\widehat{\mat X_i}^{\,n+\frac{1}{2}}}
				{\left|\delta_{\rho}\widehat{\mat X_i}^{\,n+\frac{1}{2}}\right|^{3}}
				\,\delta_{\rho}\widehat{\mat X_i}^{\,n+\frac{1}{2}}.
			\end{align}
		\end{subequations}
\end{itemize}
To assess the mesh quality quantitatively, several quantitative measures are employed.
Specifically, we consider the ratio
\[
R_1(\Delta s) = \frac{\max_i \Delta s_i}{\min_i \Delta s_i},
\]
which measures the uniformity of the discrete arc-length distribution, and the
monitor-weighted ratio
\[
R_2(M_f,\Delta s) =
\frac{\max_i \left(M_{f,i}\,\Delta s_i\right)}
{\min_i \left(M_{f,i}\,\Delta s_i\right)},
\]
which evaluates the equidistribution of the mesh with respect to the monitor
function $M_f$.
Clearly, values of $R_1$ and $R_2$ close to one indicate good mesh quality,
while larger values reflect increasing mesh distortion.

\subsection{Numerical experiments for the isotropic axisymmetric MCF} \label{sec5-isotropic}

\begin{example}[Convergence tests]
	In this example, we investigate the temporal-spatial convergence behavior
	of the proposed numerical schemes for the isotropic axisymmetric MCF.
	The initial generating curve is chosen as
	\begin{align} \label{eq:initial_curve}
		x\!\left(\rho,0\right)=4+2\cos\!\left(2\pi\rho\right),\qquad
		y\!\left(\rho,0\right)=\sin\!\left(2\pi\rho\right),\qquad \rho\in[0,1],
	\end{align}
	and the final time is set to \(T=0.4\). Fig.~\ref{fig:error_time_ISO} presents the log-log plots of the numerical errors versus the time step for the isotropic case. The results demonstrate that the ISO-A-BDF1-FDM achieves first-order convergence in time, while both the ISO-A-BDF2-FDM and ISO-A-CN-FDM achieve second-order convergence in time. Additionally, the spatial accuracy for all schemes reaches second-order. These results are entirely consistent with our expectations.
	
	\begin{figure}[htbp]\label{fig:error_time_ISO}
		\centering
		\includegraphics[width=0.97\linewidth]{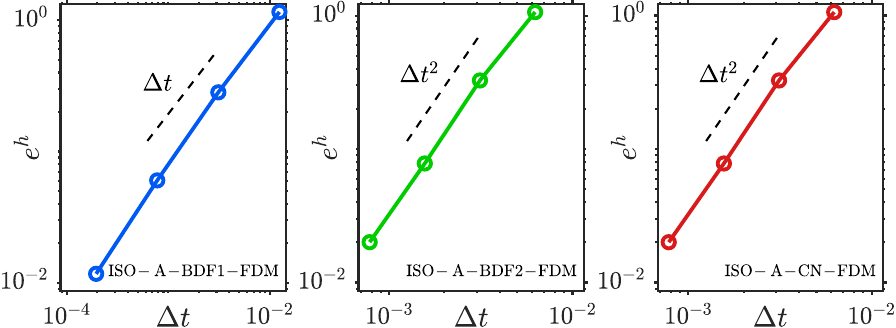}
		\hspace{0.08\linewidth} 
		\caption{
			Plots of the numerical error \(e^h\) versus the time step \(\Delta t\) are presented.
			From left to right, the results obtained with the ISO-A-BDF1-FDM, ISO-A-BDF2-FDM, and ISO-A-CN-FDM schemes are shown, respectively.
		}
	\end{figure}
\end{example}

\begin{example}  
	In this example, we investigate the mesh quality for several representative initial curves and highlight the importance of adaptive moving mesh methods. To simplify, we compare only the ISO-BDF1-FDM and ISO-A-BDF1-FDMs. The time evolution of the curves and the corresponding mesh quality indicators are presented in Figs.~\ref{fig:example2.1_mesh_quality}-\ref{fig:example2.3_mesh_quality}.
	The tests in this example include:
	\begin{itemize}
		\item First, we consider the evolution based on the initial curve with a localized geometric feature shown in Fig.~\ref{fig:example2.1_evolution}. As the evolution proceeds, the ISO-BDF1-FDM gradually develops a highly nonuniform distribution of mesh points. In contrast, the ISO-A-BDF1-FDM dynamically redistributes mesh points and maintains a significantly more uniform resolution along the curve. This qualitative difference is quantitatively confirmed by the mesh quality indicators in Fig.~\ref{fig:example2.1_mesh_quality}, where the values of \( R_1(\Delta s) \) and \( R_2(M_f, \Delta s) \) associated with the adaptive method remain significantly smaller than those of the non-adaptive scheme and gradually approach a steady value close to one.
		
		\item Next, we examine a more oscillatory initial curve, the evolution of which is shown in Fig.~\ref{fig:example2.2_evolution}. Due to the presence of rapidly varying curvature, the ISO-BDF1-FDM quickly suffers from severe mesh distortion, with mesh points clustering in some regions while becoming excessively sparse in others. As a result, the corresponding mesh quality indicators increase rapidly over time, as reported in Fig.~\ref{fig:example2.2_mesh_quality}. In contrast, the ISO-A-BDF1-FDM effectively concentrates mesh points in regions of large curvature while preserving a reasonable distribution elsewhere. As a result, the corresponding mesh quality indicators decrease steadily over time and stabilize at values close to one, indicating an increasingly uniform arc-length distribution and approximate monitor-based equidistribution. We can clearly see that, unlike ISO-A-BDF1-FDM, the evolution curve of ISO-BDF1-FDM is completely incorrect.

		\item Finally, we consider the evolution based on a highly nonconvex initial curve with complex geometric features, as illustrated in Fig.~\ref{fig:example2.3_evolution}. The mesh quality indicators are shown in Fig.~\ref{fig:example2.1_mesh_quality}. For this challenging configuration, the ISO-BDF1-FDM suffers from severe mesh degeneration at relatively early times, which eventually leads to numerical singularity and causes the simulation to break down at approximately \( t \approx 0.38 \). Consequently, no reliable mesh quality indicators can be obtained beyond this time for the non-adaptive method. In contrast, the ISO-A-BDF1-FDM remains numerically stable and allows the computation to proceed to much later times, up to about \( t = 0.7 \). The corresponding mesh quality indicators continue to decrease and gradually approach unity as time increases.
		
	\end{itemize}
	
	Overall, these numerical results demonstrate that the proposed ISO-A-BDF1-FDM provides an effective mesh redistribution mechanism and yields consistently superior mesh quality compared with the classical ISO-BDF1-FDM, particularly for curves with strong geometric complexity.

	\begin{figure}[htbp]
		\centering
		\includegraphics[width=0.38\linewidth]{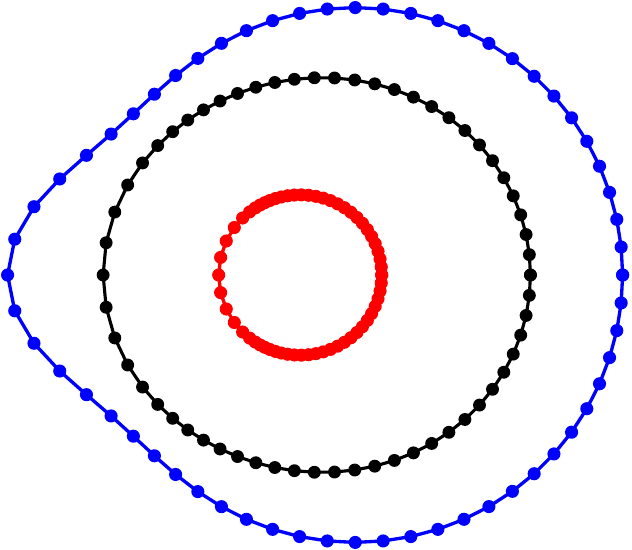}
		\hspace{0.08\linewidth} 
		\includegraphics[width=0.38\linewidth]{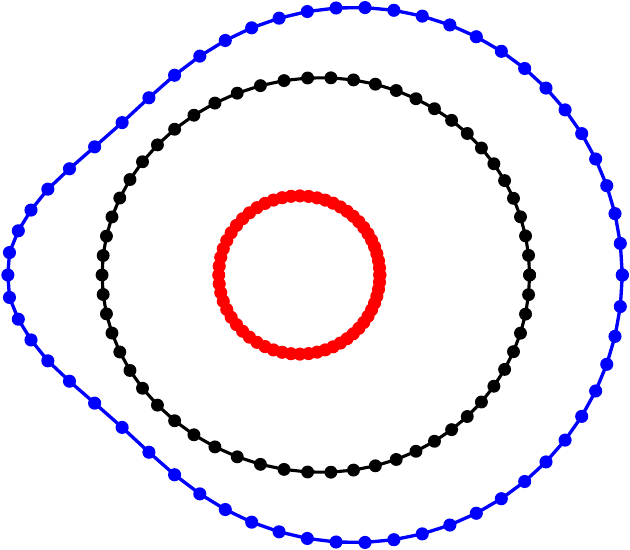}
		\caption{Curve evolution for the ISO-BDF1-FDM and ISO-A-BDF1-FDM, with the initial curve 
			$x = 4 + \left[1 + 0.3\exp\left(-\frac{(2\pi\rho - \pi)^{2}}{0.16}\right)\right]\cos (2\pi\rho), y = \left[1 + 0.3\exp\left(-\frac{(2\pi\rho - \pi)^{2}}{0.16}\right)\right]\sin (2\pi\rho)$.
			Left: ISO-BDF1-FDM, right: ISO-A-BDF1-FDM.
			Blue curves denote the initial curve,
			while black and red curves correspond to $t=0.25$, and $t=0.5$, respectively.}
		\label{fig:example2.1_evolution}
	\end{figure}
	
	\begin{figure}[htbp]
		\centering
		\includegraphics[width=0.44\linewidth]{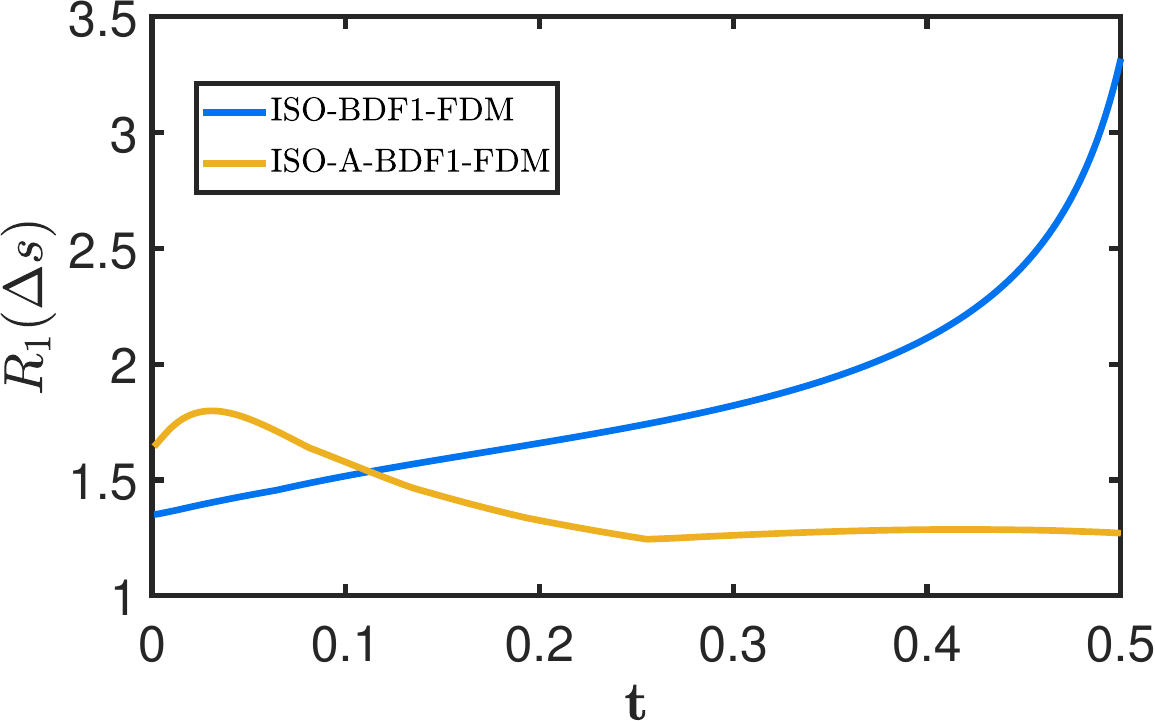}
		\hspace{0.04\linewidth} 
		\includegraphics[width=0.44\linewidth]{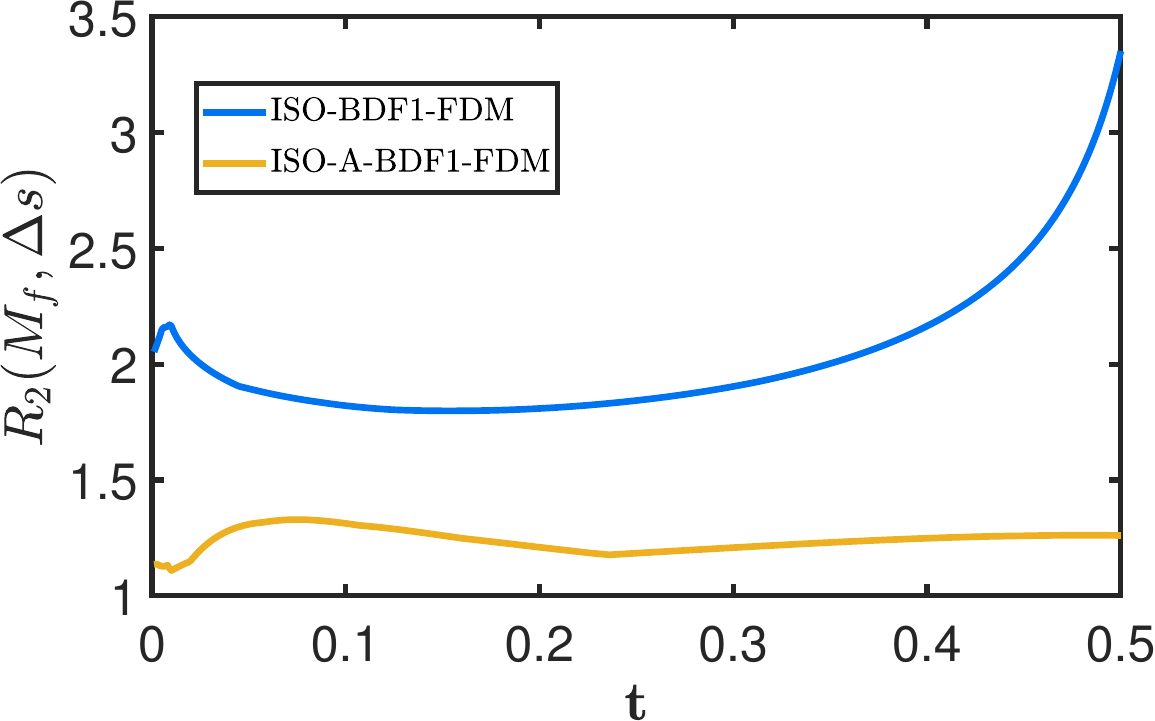}
		\caption{
			Time evolution of the mesh quality indicators
			$R_1(\Delta s)$ (left) and $R_2(M_f,\Delta s)$ (right)
			for the initial curve shown in Fig.~\ref{fig:example2.1_evolution}.
			The results correspond to the ISO-BDF1-FDM and the ISO-A-BDF1-FDM.
		}
		\label{fig:example2.1_mesh_quality}
	\end{figure}

	\begin{figure}[htbp]
		\centering
		\includegraphics[width=0.38\linewidth]{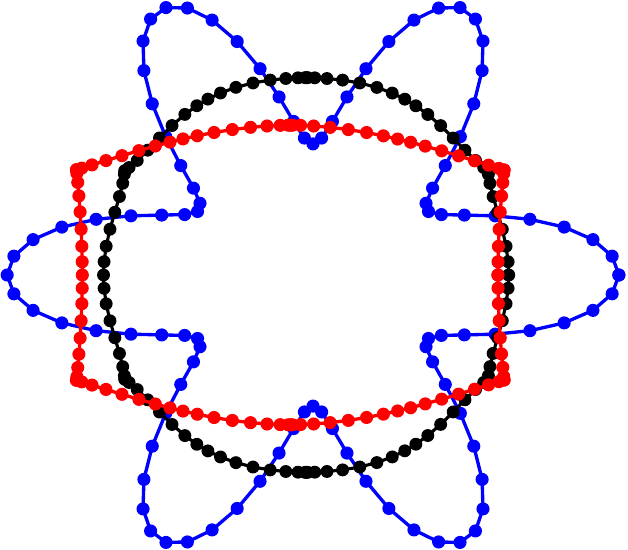}
		\hspace{0.08\linewidth} 
		\includegraphics[width=0.38\linewidth]{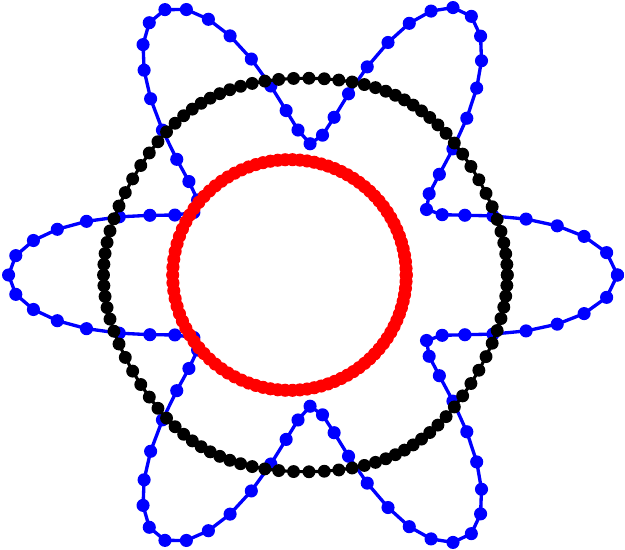}
		\caption{Curve evolution for the ISO-BDF1-FDM and ISO-A-BDF1-FDM, with the initial curve 
			$x = 4 + \left[1 + 0.4\cos\left(12\pi\rho\right)\right]\cos\left(2\pi\rho\right), 
			y = \left[1 + 0.4\cos\left(12\pi\rho\right)\right]\sin\left(2\pi\rho\right)$.
			Left: ISO-BDF1-FDM, right: ISO-A-BDF1-FDM.
			Blue curves denote the initial curve,
			while black and red curves correspond to $t=0.12$, and $t=0.4$, respectively.}
		\label{fig:example2.2_evolution}
	\end{figure}
	
	\begin{figure}[htbp]
		\centering
		\includegraphics[width=0.44\linewidth]{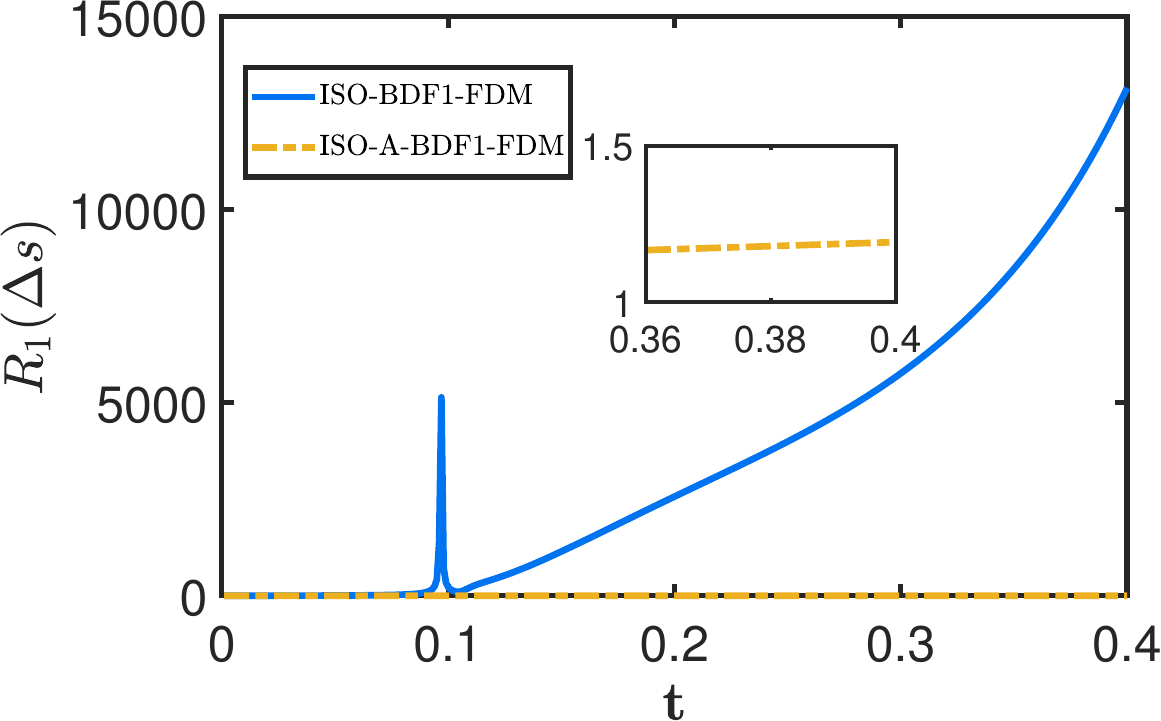}
		\hspace{0.04\linewidth} 
		\includegraphics[width=0.44\linewidth]{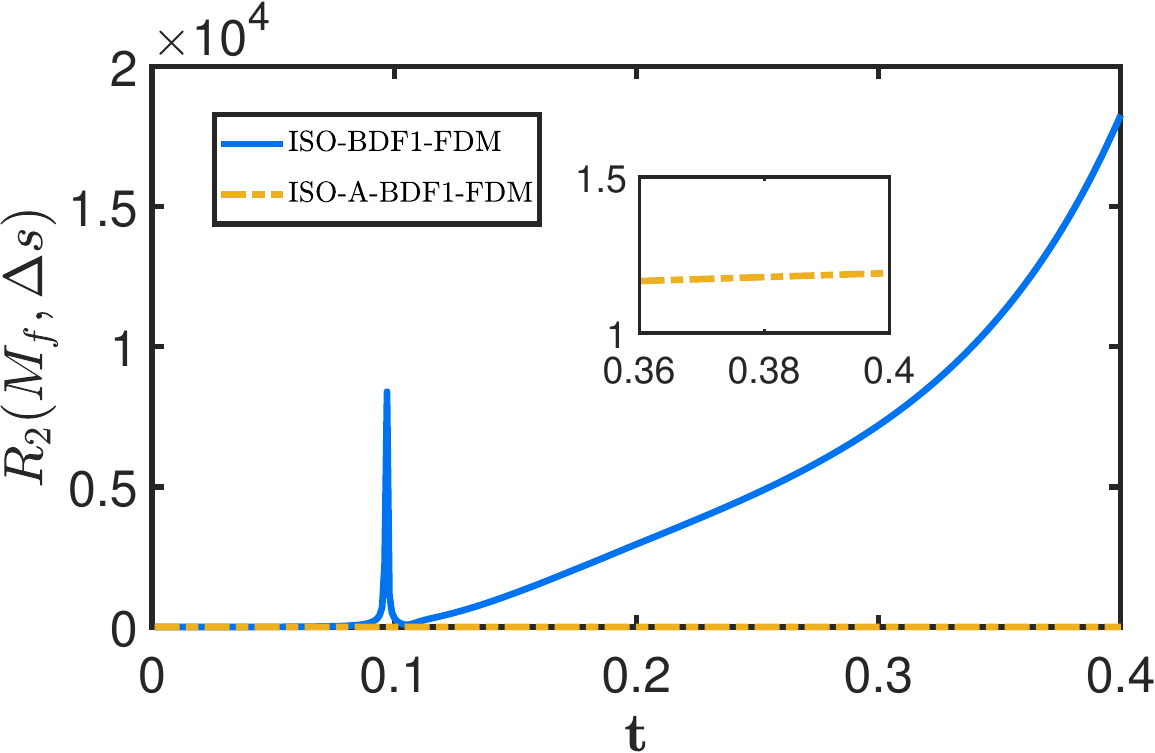}
		\caption{
			Time evolution of the mesh quality indicators
			$R_1(\Delta s)$ (left) and $R_2(M_f,\Delta s)$ (right)
			for the initial curve shown in Fig.~\ref{fig:example2.2_evolution}.
			The results correspond to the ISO-BDF1-FDM and the ISO-A-BDF1-FDM.
		}
		\label{fig:example2.2_mesh_quality}
	\end{figure}
	
	\begin{figure}[htbp]
		\centering
		\includegraphics[width=0.45\linewidth]{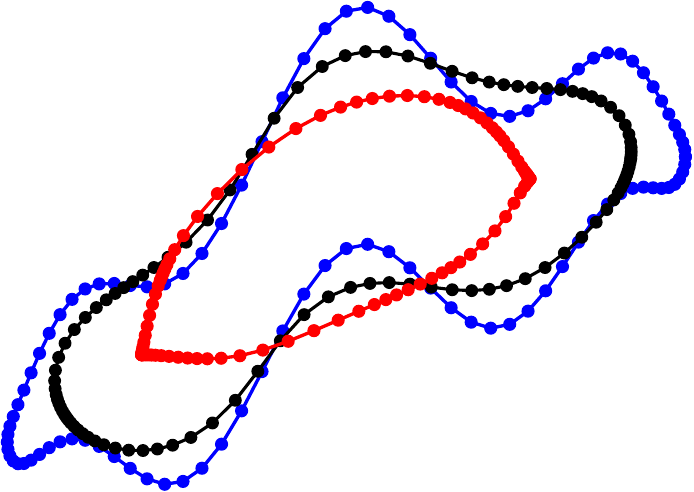}
		\hspace{0.08\linewidth} 
		\includegraphics[width=0.45\linewidth]{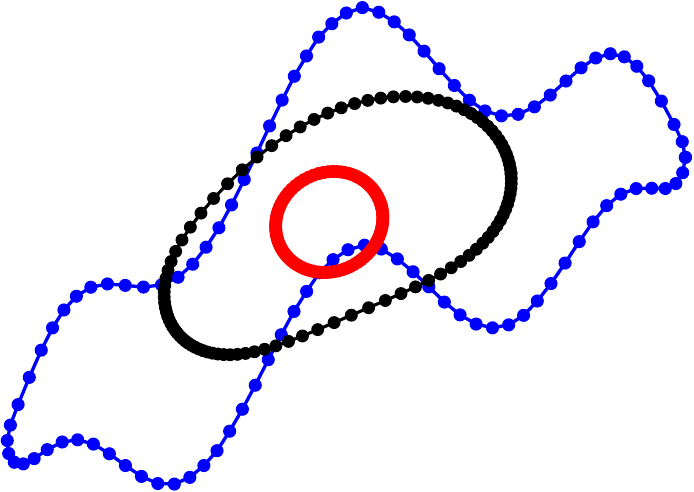}
		\caption{Curve evolution for the ISO-BDF1-FDM and ISO-A-BDF1-FDM, with the initial curve
			$x = 6 + 2\cos\left(2\pi\rho\right)$, $
			y = 0.5\sin\left(2\pi\rho\right)
			+ \sin\!\left(\cos\left(2\pi\rho\right)\right)
			+ \sin\left(2\pi\rho\right)\!\left[0.2 + \sin\left(2\pi\rho\right)\sin^{2}\left(6\pi\rho\right)\right].$ Left: ISO-BDF1-FDM; right: ISO-A-BDF1-FDM.
			Blue curves denote the initial curve,
			while black and red curves correspond to \(t = 0.1\), and \(t = 0.38\) on the left panel,
			and to \(t = 0.38\), and \(t = 0.66\) on the right panel, respectively.}
		\label{fig:example2.3_evolution}
	\end{figure}
	
	\begin{figure}[htbp]
		\centering
		\includegraphics[width=0.4\linewidth]{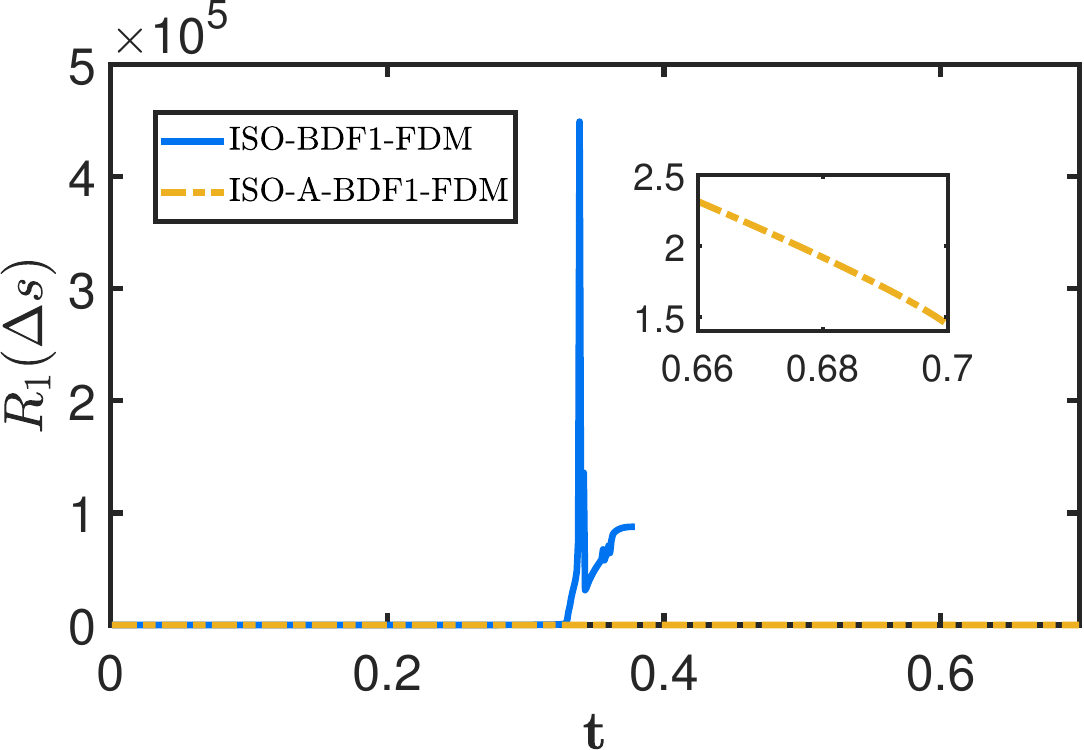}
		\hspace{0.08\linewidth} 
		\includegraphics[width=0.4\linewidth]{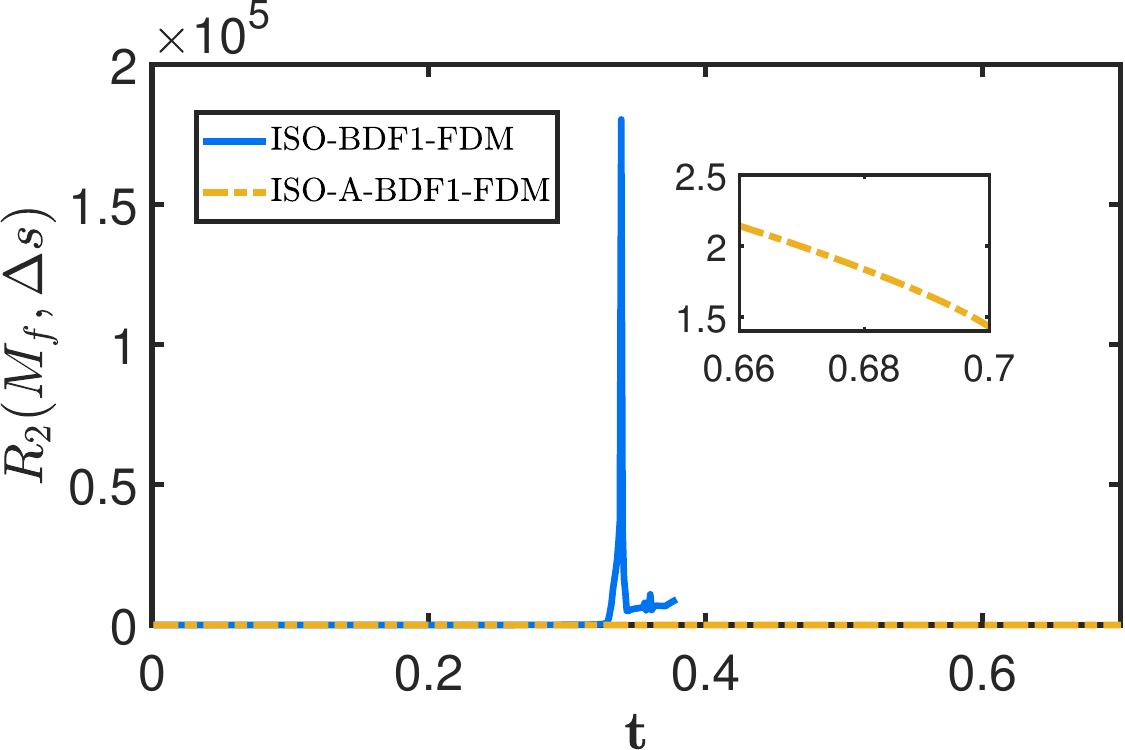}
		\caption{Time evolution of the mesh quality indicators
			$R_1(\Delta s)$ (left) and $R_2(M_f,\Delta s)$ (right)
			for the initial curve shown in Fig.~\ref{fig:example2.3_evolution}.
			The results correspond to the ISO-BDF1-FDM and the ISO-A-BDF1-FDM.}
		\label{fig:example2.3_mesh_quality}
	\end{figure}
\end{example}

\begin{example}
	In this example, we examine the area-decay property of the proposed
	structure-preserving schemes, namely the ISO-A-LM-BDFk-FDMs and the
	ISO-A-LM-CN-FDM, for the isotropic axisymmetric MCF, which
	inherit the intrinsic structure of the continuous model.
	The initial generating curve is chosen as
	\[
	x(\rho,0) = 6 + 3\cos(2\pi\rho), \qquad
	y(\rho,0) = 2\sin(2\pi\rho), \qquad \rho \in [0,1],
	\]
	which corresponds to an axisymmetric surface with a nonuniform curvature distribution.
	We define \( A^h(t)\big|_{t=t_m} := A^m \) as the area of the axisymmetric surface generated by the discrete curve, and we study the temporal evolution of the normalized area \( A^h(t)/A^h(0) \).
	Fig.~\ref{fig:ISO_energy_relative} displays the evolution of the normalized surface area
	$A^h(t)/A^h(0)$ computed by the ISO-A-LM-BDFk-FDMs and ISO-A-LM-CN-FDM.
	It is observed that the surface area decreases monotonically for all schemes over
	the entire time interval, which is consistent with our theoretical results.

	\begin{figure}[htbp]
		\centering
		\includegraphics[width=0.6\linewidth]{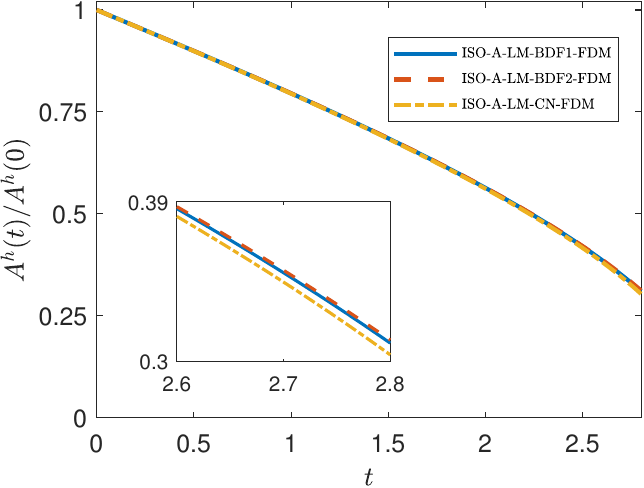}
		\caption{
			Time evolution of the normalized surface area $A^h(t)/A^h(0)$
			for the isotropic axisymmetric MCF computed by the
			ISO-A-LM-BDFk-FDMs and ISO-A-LM-CN-FDM.
			The initial generating curve is
			$x(\rho,0)=6+3\cos(2\pi\rho)$ and $y(\rho,0)=2\sin(2\pi\rho)$.
			The final time is $T=2.8$ with a uniform time step $\Delta t=0.01$.
			All schemes exhibit a clear area-decay behavior.
			The inset shows a local magnification over $t \in [2.6,2.8]$.
		}
		\label{fig:ISO_energy_relative}
	\end{figure}
\end{example}

\subsection{Numerical experiments for the anisotropic axisymmetric MCF}
\begin{example}[Convergence tests]
	In this example, we investigate the temporal-spatial convergence behavior
	of the proposed numerical schemes for the anisotropic axisymmetric MCF.
	The same initial generating curve~\eqref{eq:initial_curve} and final time
	$T=0.4$ as in the isotropic case are employed.
	Fig.~\ref{fig:error_time_ANISO} shows the log-log plots of the numerical
	errors versus the time step size for the anisotropic case.
	The results demonstrate that the ANISO-A-BDF1-FDM achieves first-order
	accuracy in time, while both the ANISO-A-BDF2-FDM and ANISO-A-CN-FDM achieve
	second-order temporal accuracy.
	Moreover, all schemes exhibit second-order spatial accuracy.
	These observations are fully consistent with the expected convergence rates.
	\begin{figure}[htbp]
		\label{fig:error_time_ANISO}
		\centering
		\includegraphics[width=0.97\linewidth]{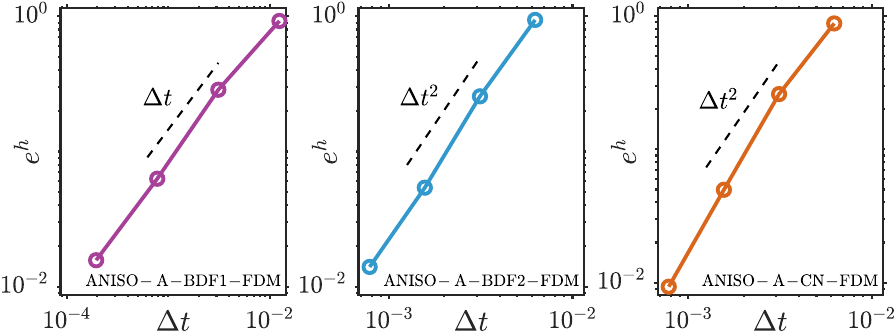}
		\caption{
			Plots of the numerical error \(e^h\) versus the time step \(\Delta t\).
			From left to right, the results obtained with the ANISO-A-BDF1-FDM,
			ANISO-A-BDF2-FDM, and ANISO-A-CN-FDM schemes are shown, respectively.
		}
	\end{figure}
\end{example}

\begin{example}
	In this example, we investigate the mesh quality of the moving-mesh strategy for the anisotropic
	axisymmetric MCF with a four-fold anisotropy
	\begin{equation*}
		\gamma(\theta)=1+0.04\cos\left(4\theta\right).
	\end{equation*}
	To simplify the comparison and focus on mesh quality behavior,
	we consider only the non-adaptive scheme ANISO-BDF1-FDM and its adaptive counterpart
	ANISO-A-BDF1-FDM.
	For each test, we report the mesh evolution together with the mesh quality indicators
	\(R_1(\Delta s)\) and \(R_2(M_f,\Delta s)\). For this example, we mainly do the following tests: 
	\begin{itemize}
		\item Fig.~\ref{fig:ANISO_example1_evolution} shows that the ANISO-BDF1-FDM rapidly develops
		a strongly nonuniform point distribution during the evolution, with increasingly uneven
		spatial resolution along the curve.
		In contrast, the ANISO-A-BDF1-FDM dynamically redistributes mesh points and maintains a visibly
		more balanced resolution throughout the entire evolution.
		This qualitative observation is quantitatively confirmed in
		Fig.~\ref{fig:ANISO_example1_mesh}, where the mesh quality indicators for the ANISO-BDF1-FDM
		grow steadily and increase sharply near the final time,
		while those for the adaptive scheme remain bounded and nearly flat. We can clearly see that the evolution curve based on ANISO-BDF1-FDM is incorrect.
		\item For the initial curve shown in Fig.~\ref{fig:ANISO_example2_evolution}, the non-adaptive
		ANISO-BDF1-FDM exhibits a clear \emph{loss of resolution} near the upper part of the curve:
		mesh points become noticeably sparse in this region, particularly in the upper-left part,
		whereas other regions of the curve remain over-resolved.
		By contrast, the ANISO-A-BDF1-FDM dynamically adjusts the point distribution and preserves
		a much more reasonable resolution around the localized geometric feature.
		The quantitative evidence is provided in Fig.~\ref{fig:ANISO_example2_mesh}, where
		both \(R_1(\Delta s)\) and \(R_2(M_f,\Delta s)\) for the ANISO-BDF1-FDM grow dramatically
		with a clear blow-up trend near the end of the simulation,
		while the corresponding curves for the adaptive scheme stay small and stable
		throughout the entire time interval. We can observe that, unlike the ANISO-A-BDF1-FDM, the evolution curve based on ANISO-BDF1-FDM is incorrect.
		\item For the oscillatory initial curve shown in Fig.~\ref{fig:ANISO_example3_evolution},
		the ANISO-BDF1-FDM quickly suffers from severe mesh distortion,
		leading to highly uneven point distributions and under-resolved geometric features.
		Accordingly, Fig.~\ref{fig:ANISO_example3_mesh} shows that both mesh quality indicators
		\(R_1(\Delta s)\) and \(R_2(M_f,\Delta s)\) increase rapidly and reach very large values.
		In contrast, the ANISO-A-BDF1-FDM keeps both indicators at much smaller levels,
		with only mild temporal variations, demonstrating robust mesh control even for
		highly oscillatory geometries. In this test, the evolution curve governed by the ANISO-BDF1-FDM is completely incorrect.
		
		\item The adaptive methods proposed in this work are initialized using the de Boor algorithm to generate a well-distributed initial mesh according to the prescribed monitor function. 
		This naturally raises the question of whether the mesh differences observed in the first three numerical experiments are merely a consequence of this initial mesh preprocessing. 
		To eliminate this concern, we conduct two additional sets of experiments, in which the de Boor algorithm is applied only to the initial mesh of the non-adaptive ANISO-BDF1-FDM, while the adaptive ANISO-A-BDF1-FDM starts from the original parameterization.
		As observed in  Fig.~\ref{fig:ANISO_example4_evolution},
		the de Boor initialization leads to a more uniform point distribution
		for the ANISO-BDF1-FDM at early times.
		However, as the evolution proceeds, the mesh quality of the non-adaptive scheme
		still deteriorates progressively. 
		The corresponding mesh quality indicators
		\(R_1(\Delta s)\) and \(R_2(M_f,\Delta s)\),
		shown in Fig.~\ref{fig:ANISO_example4_mesh},
		indicate that although the initial values are improved by the mesh preprocessing,
		both indicators exhibit sustained growth over time.
		In contrast, despite the absence of initial mesh preprocessing,
		the ANISO-A-BDF1-FDM gradually improves its mesh quality through adaptive
		tangential redistribution and maintains significantly smaller and more stable
		indicator values at later times.
		These results indicate that while appropriate initialization improves
		early-stage mesh quality, adaptive tangential redistribution is essential
		for maintaining robust mesh quality over long-time evolutions. 
		The evolution curve based on ANISO-BDF1-FDM is clearly incorrect, whereas the evolution curve based on ANISO-A-BDF1-FDM is evidently correct. This demonstrates that the mesh quality directly determines the accuracy of the evolution.
	\end{itemize}

	\begin{figure} [htbp]
		\centering
		\includegraphics[width=0.33\linewidth]{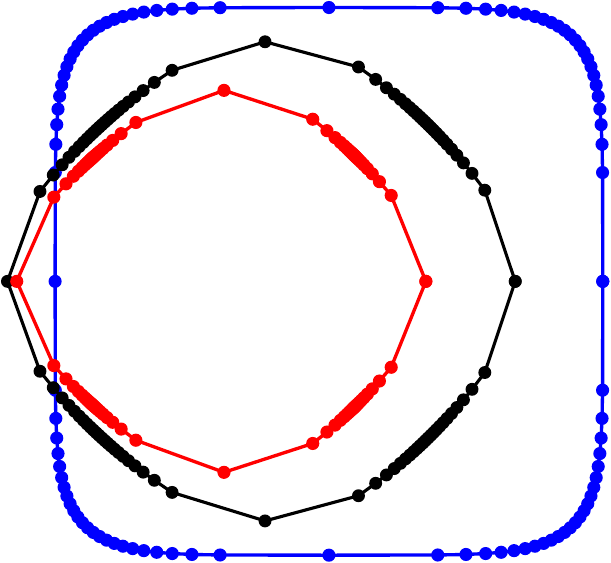}
		\hspace{0.2\linewidth} 
		\includegraphics[width=0.3\linewidth]{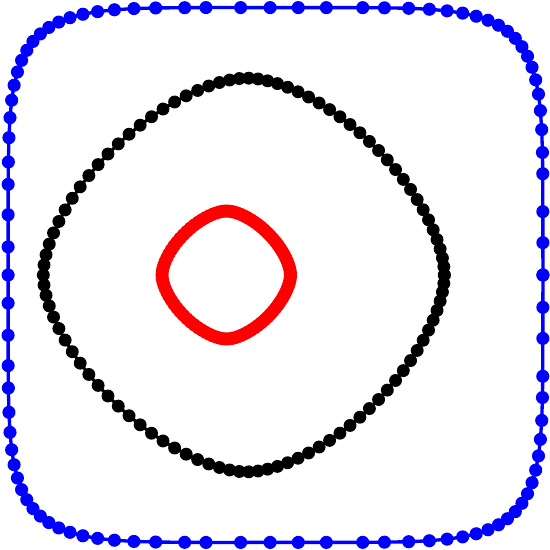}
		\caption{
			Curve evolution for the ANISO-BDF1-FDM and ANISO-A-BDF1-FDM, with  the initial curve $x = 8 + 2.3\,\mathrm{sign}\!\left(\cos\left(2\pi\rho\right)\right)
			\left|\cos\left(2\pi\rho\right)\right|^{1/3},
			$ $y = 2 + 2.3\,\mathrm{sign}\!\left(\sin\left(2\pi\rho\right)\right)
			\left|\sin\left(2\pi\rho\right)\right|^{1/3}.$
			The final time is \(T = 3.2\) with time step \(\Delta t = 0.001\).
			Left: ANISO-BDF1-FDM; right: ANISO-A-BDF1-FDM.
			Blue curves denote the initial curve,
			while black and red curves correspond to \(t = 2\), and \(t = 3.2\), respectively.}
		\label{fig:ANISO_example1_evolution}
	\end{figure}
	\begin{figure}[htbp]
		\centering
		\includegraphics[width=0.4\linewidth]{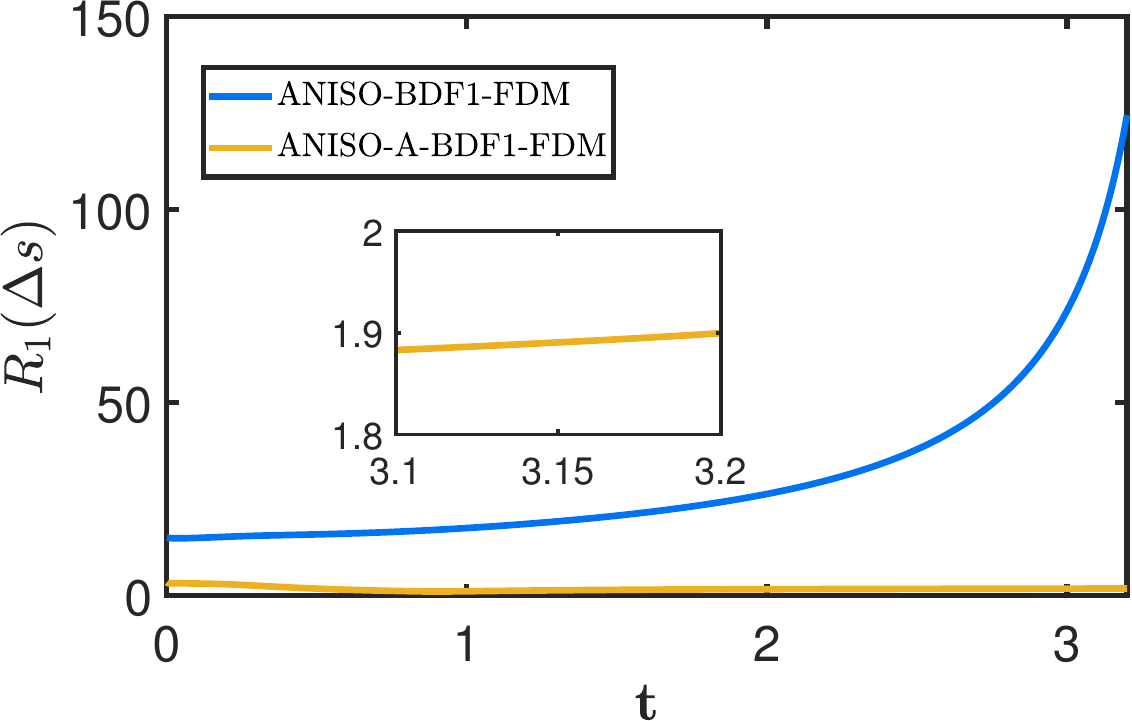}
		\hspace{0.08\linewidth} 
		\includegraphics[width=0.4\linewidth]{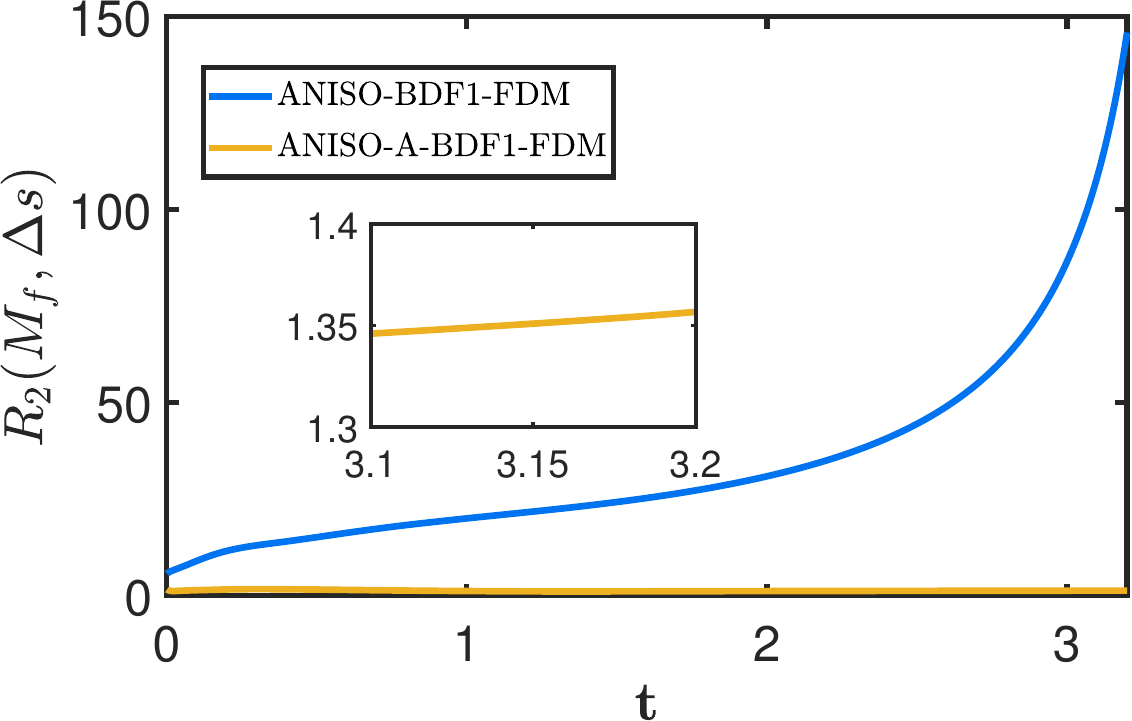}
		\caption{Time evolution of the mesh quality indicators
			$R_1(\Delta s)$ (left) and $R_2(M_f,\Delta s)$ (right)
			for the initial curve shown in Fig.~\ref{fig:ANISO_example1_evolution}.
			The results correspond to the ANISO-BDF1-FDM and the ANISO-A-BDF1-FDM.}
		\label{fig:ANISO_example1_mesh}
	\end{figure}
	\begin{figure}[htbp]
		\centering
		\includegraphics[width=0.3\linewidth]{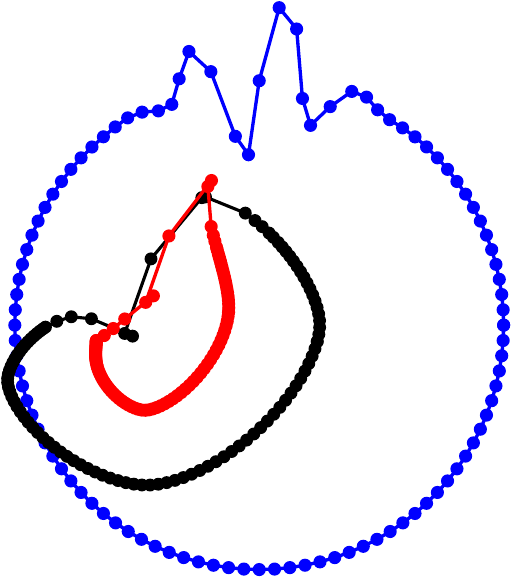}
		\hspace{0.2\linewidth} 
		\includegraphics[width=0.3\linewidth]{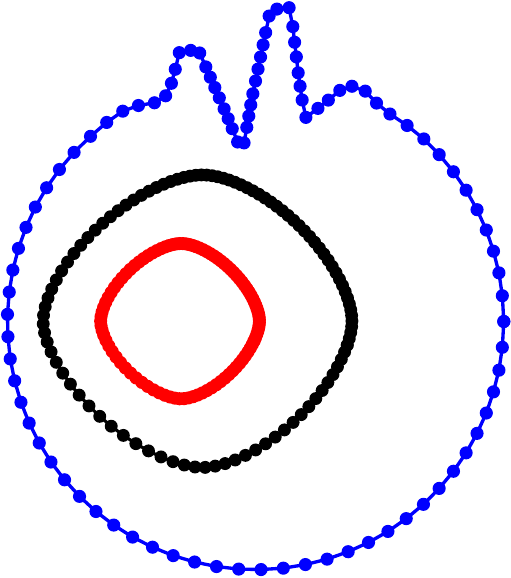}
		\caption{Curve evolution for the ANISO-BDF1-FDM and ANISO-A-BDF1-FDM, with  the initial curve 
			$x = 4 + 2\left[1 + 0.35\,\exp\!\left(-\left(\frac{\rho-0.25}{0.045}\right)^2\right)
			\sin\!\left(36\pi\rho\right)\right]\cos\!\left(2\pi\rho\right),
			y = 2 + 2\left[1 + 0.35\,\exp\!\left(-\left(\frac{\rho-0.25}{0.045}\right)^2\right)
			\sin\!\left(36\pi\rho\right)\right]\sin\!\left(2\pi\rho\right).$
			The final time is \(T = 2\) with time step \(\Delta t = 0.001\).
			Left: ANISO-BDF1-FDM; right: ANISO-A-BDF1-FDM.
			Blue curves denote the initial curve,
			while black and red curves correspond to\(t = 1.5\), and \(t = 2\), respectively.}
		\label{fig:ANISO_example2_evolution}
	\end{figure}
	\begin{figure}[htbp]
		\centering
		\includegraphics[width=0.4\linewidth]{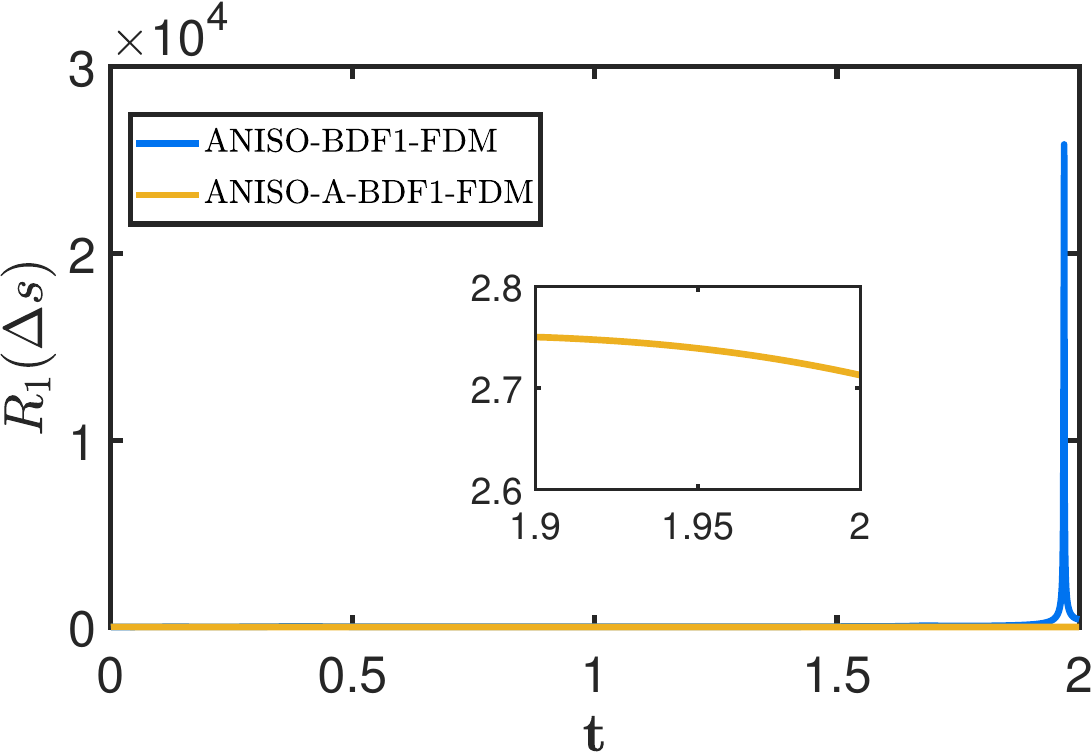}
		\hspace{0.08\linewidth} 
		\includegraphics[width=0.4\linewidth]{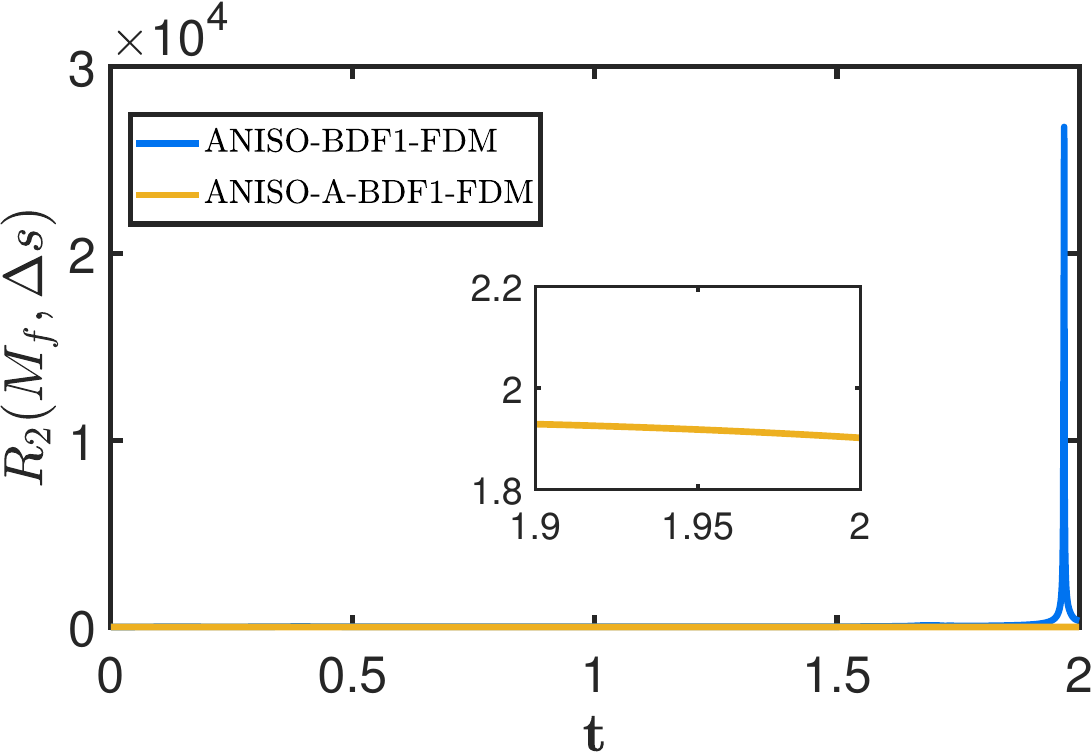}
		\caption{Time evolution of the mesh quality indicators
			$R_1(\Delta s)$ (left) and $R_2(M_f,\Delta s)$ (right)
			for the initial curve shown in Fig.~\ref{fig:ANISO_example2_evolution}.
			The results correspond to the ANISO-BDF1-FDM and the ANISO-A-BDF1-FDM.}
		\label{fig:ANISO_example2_mesh}
	\end{figure}
	
	\begin{figure}[htbp]
		\centering
		\includegraphics[width=0.3\linewidth]{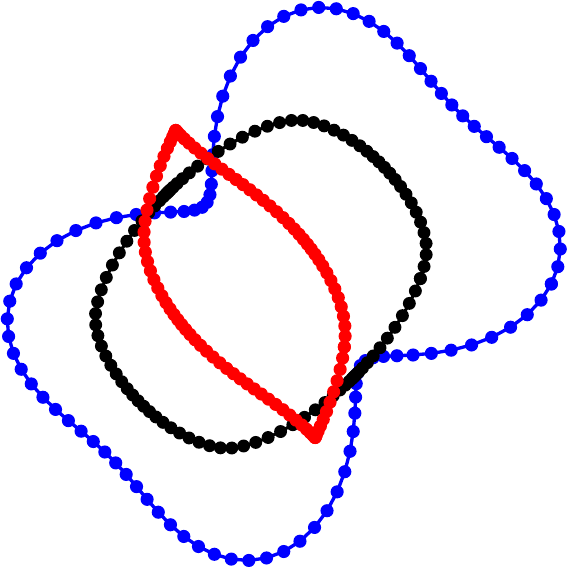}
		\hspace{0.2\linewidth} 
		\includegraphics[width=0.3\linewidth]{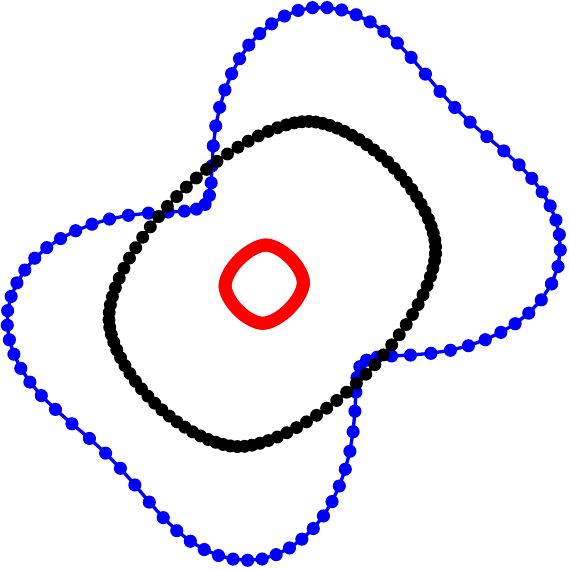}
		\caption{Curve evolution for the ANISO-BDF1-FDM and ANISO-A-BDF1-FDM, with  the initial curve 
			$x = 6 + \left[1 + 0.3\sin\!\left(4\pi\rho\right) + 0.2\cos\!\left(8\pi\rho\right)\right]
			\cos\!\left(2\pi\rho\right),$ $y = \left[1 + 0.3\sin\!\left(4\pi\rho\right) + 0.2\cos\!\left(8\pi\rho\right)\right]
			\sin\!\left(2\pi\rho\right).$
			The final time is \(T = 0.52\) with time step \(\Delta t = 0.001\).
			Left: ANISO-BDF1-FDM; right: ANISO-A-BDF1-FDM.
			Blue curves denote the initial curve,
			while black and red curves correspond to \(t = 0.3\), and \(t = 0.52\), respectively.}
		\label{fig:ANISO_example3_evolution}
	\end{figure}
	\begin{figure} [htbp]
		\centering
		\includegraphics[width=0.4\linewidth]{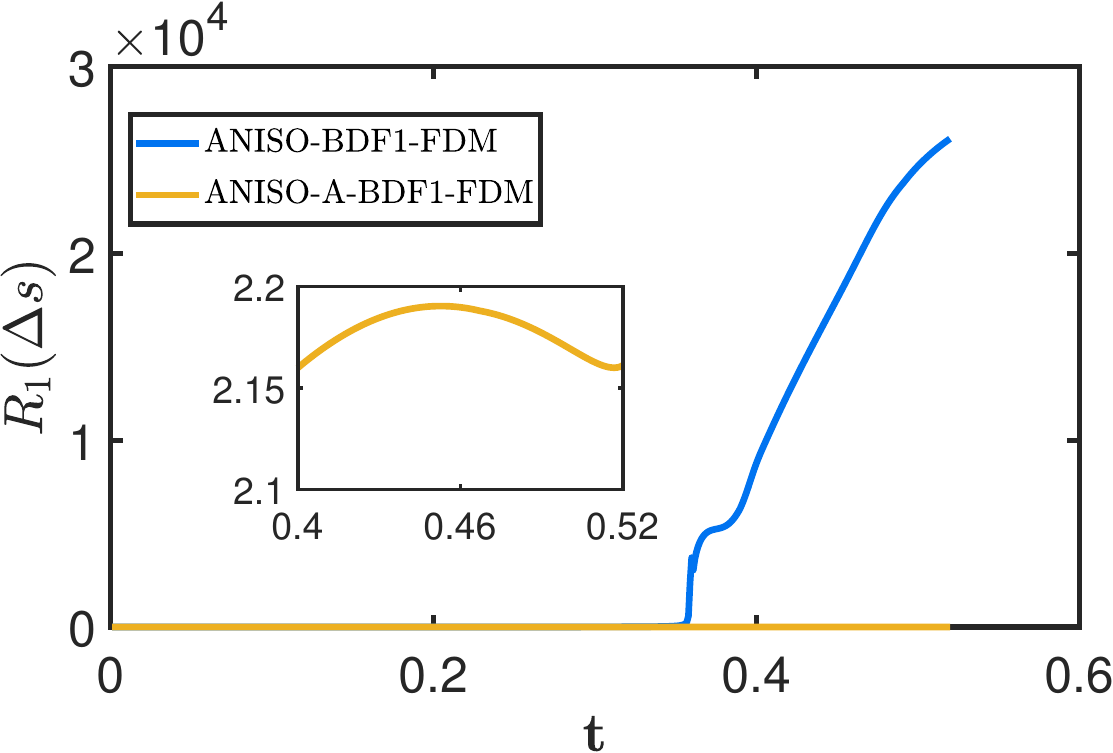}
		\hspace{0.08\linewidth} 
		\includegraphics[width=0.4\linewidth]{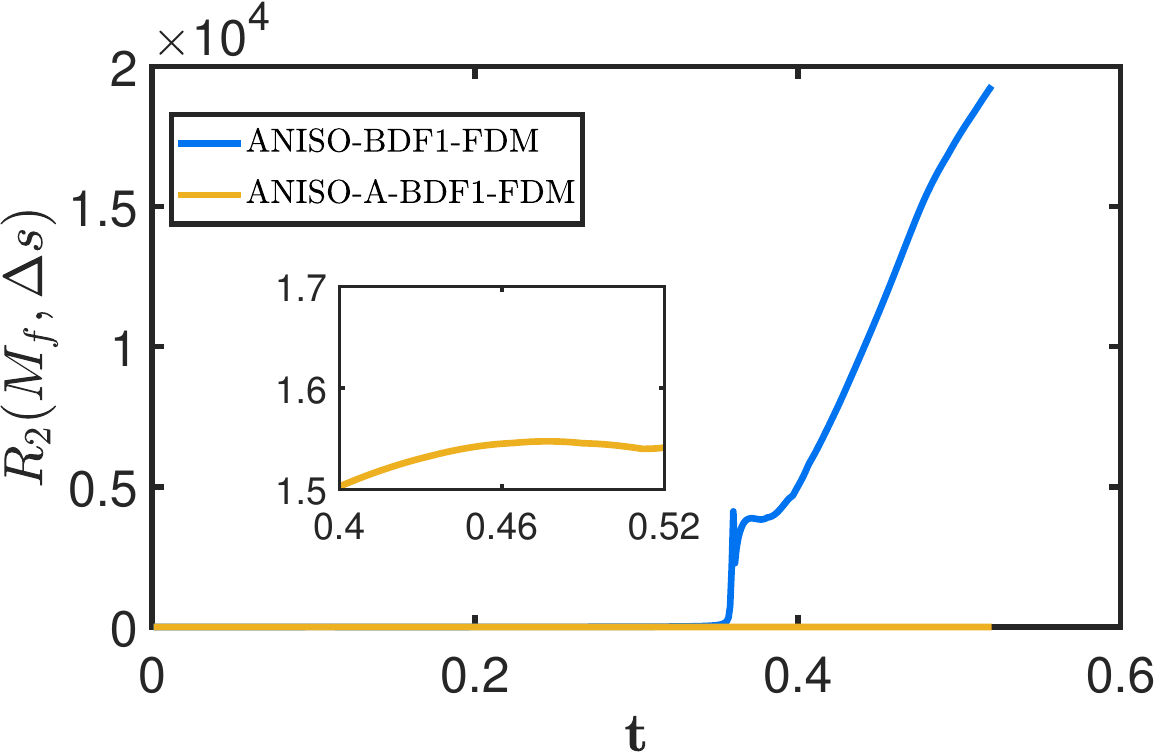}
		\caption{Time evolution of the mesh quality indicators
			$R_1(\Delta s)$ (left) and $R_2(M_f,\Delta s)$ (right)
			for the initial curve shown in Fig.~\ref{fig:ANISO_example3_evolution}.
			The results correspond to the ANISO-BDF1-FDM and the ANISO-A-BDF1-FDM.}
		\label{fig:ANISO_example3_mesh}
	\end{figure}
	\begin{figure} [htbp]
		\centering
		\includegraphics[width=0.31\linewidth]{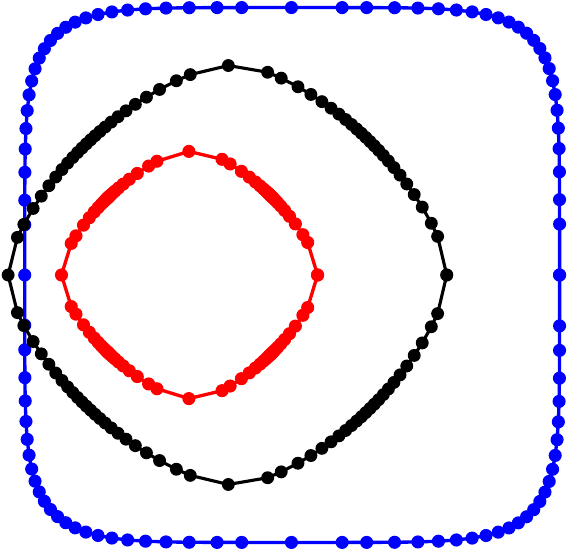}
		\hspace{0.2\linewidth} 
		\includegraphics[width=0.3\linewidth]{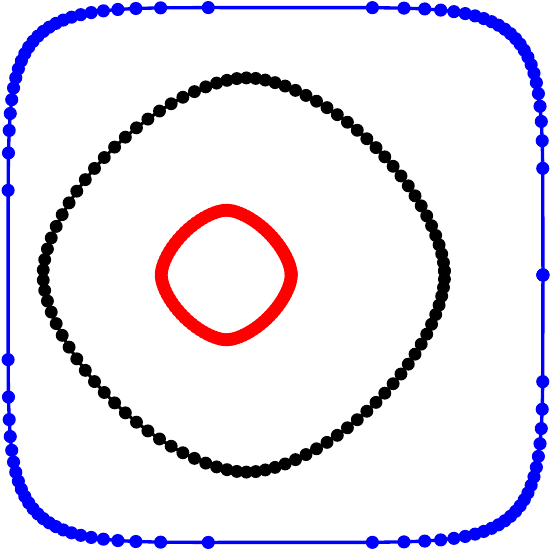}
		\caption{
			Curve evolution for the ANISO-BDF1-FDM and ANISO-A-BDF1-FDM, with  the initial curve 
			$x = 8 + 2.3\,\mathrm{sign}\!\left(\cos\left(2\pi\rho\right)\right)
			\left|\cos\left(2\pi\rho\right)\right|^{1/3},$ $y = 2 + 2.3\,\mathrm{sign}\!\left(\sin\left(2\pi\rho\right)\right)
			\left|\sin\left(2\pi\rho\right)\right|^{1/3}.$
			The final time is \(T = 3.2\) with time step \(\Delta t = 0.001\).
			Left: ANISO-BDF1-FDM; right: ANISO-A-BDF1-FDM.
			Blue curves denote the initial curve,
			while black and red curves correspond to \(t =0\), \(t = 2\), and \(t = 3.2\), respectively.}
		\label{fig:ANISO_example4_evolution}
	\end{figure}
	\begin{figure} [htbp]
		\centering
		\includegraphics[width=0.4\linewidth]{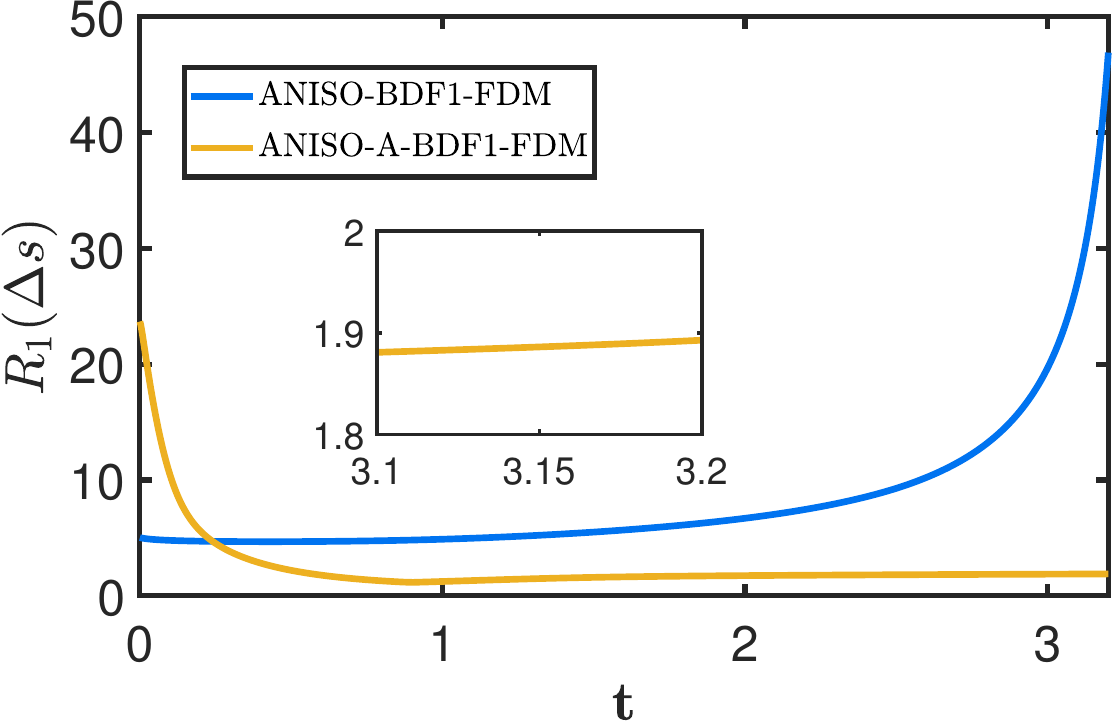}
		\hspace{0.08\linewidth} 
		\includegraphics[width=0.4\linewidth]{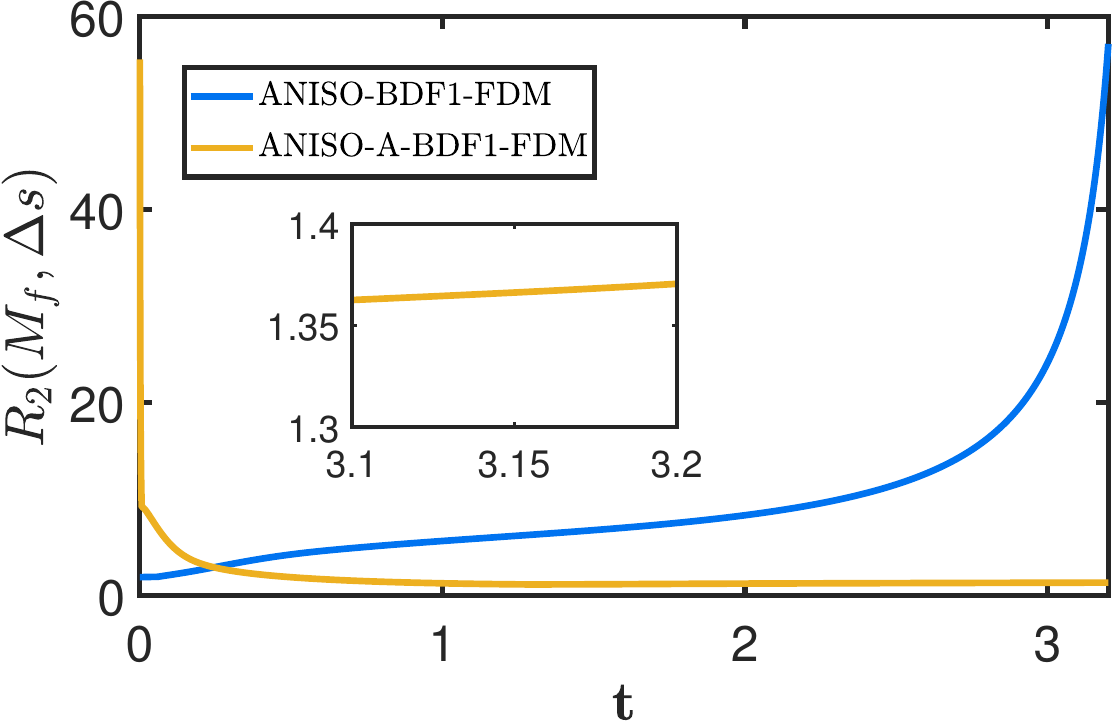}
		\caption{Time evolution of the mesh quality indicators
			$R_1(\Delta s)$ (left) and $R_2(M_f,\Delta s)$ (right)
			for the initial curve shown in Fig.~\ref{fig:ANISO_example4_evolution}.
			The results correspond to the ANISO-BDF1-FDM and the ANISO-A-BDF1-FDM.}
		\label{fig:ANISO_example4_mesh}
	\end{figure}
	
\end{example}

\begin{example}
	In this example, we compare the numerical performance of adaptive and non-adaptive schemes
	under different anisotropy strengths for the anisotropic axisymmetric MCF. 
	The same initial generating curve as in Fig.~\ref{fig:ANISO_example5_evolution} is employed,
	while the anisotropy parameter $\beta$ is varied to examine increasingly strong
	anisotropic effects.
	Fig. \ref{fig:ANISO_example5_evolution} shows the evolution computed by the ANISO-BDF1-FDM and ANISO-A-BDF1-FDM for the anisotropic system with 4-fold anisotropy. We select different anisotropy strengths, with \( \beta = 0.02, 0.04, \) and \( 0.06 \), respectively. As the anisotropy strength increases, the non-adaptive scheme exhibits a noticeable decline in mesh quality, characterized by an increasing nonuniformity in the mesh distribution throughout the evolution. The plots clearly illustrate that this mesh deterioration leads to an inaccurate representation of the evolving surface.
	This behavior is further corroborated by the mesh quality indicators presented in Fig. \ref{fig:ANISO_example5_mesh_quality}, where both \( R_1(\Delta s) \) and \( R_2(M_f, \Delta s) \) increase significantly as \( \beta \) rises. 
	In contrast, the adaptive scheme is able to maintain a well-distributed mesh for all tested
	anisotropy strengths.
	Even in the presence of relatively strong anisotropy, the mesh quality indicators remain
	well controlled and show no evident degradation over time.
	These results indicate that the proposed adaptive strategy is robust with respect to the
	anisotropy strength and can effectively preserve good mesh quality in anisotropic regimes.
	
	\begin{figure}[htbp]
		\centering
		\includegraphics[width=0.85\linewidth]{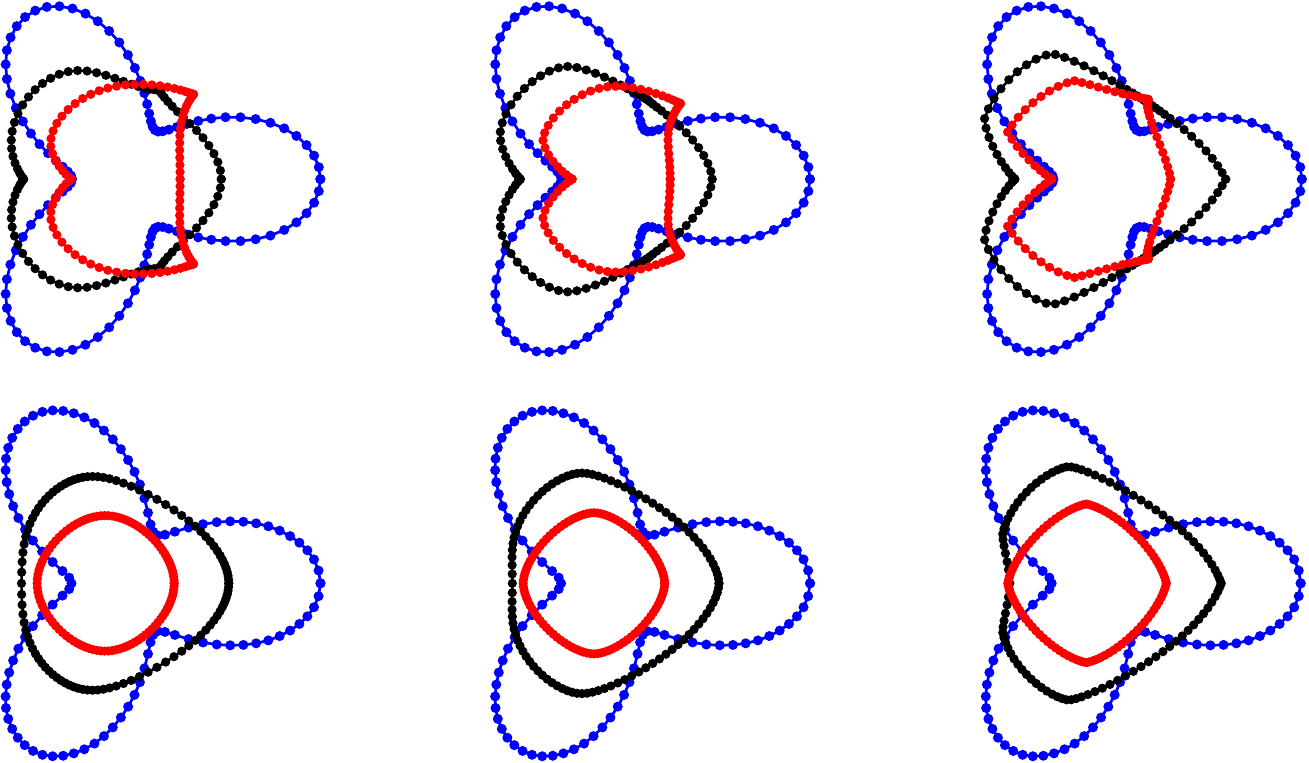}
		\caption{Curve evolution for the ANISO-BDF1-FDM and the ANISO-A-BDF1-FDM with different anisotropic strengths.
			The initial curve is selected as 
			$x = 6 + 2\left[1 + 0.55\cos\left(6\pi\rho\right)\right]\cos\left(2\pi\rho\right),$
			$y = 2\left[1 + 0.55\cos\left(6\pi\rho\right)\right]\sin\left(2\pi\rho\right).$
			The final time is $T = 1.8$ with time step $\Delta t = 0.01$.
			The top row shows the results obtained by ANISO-BDF1-FDM,
			while the bottom row corresponds to ANISO-A-BDF1-FDM,
			for $\beta = 0.02, 0.04,$ and $0.06$ from left to right.
			Blue curves denote the initial curve,
			while black and red curves correspond to $t = 1.0$ and $t = 1.8$, respectively.}
		\label{fig:ANISO_example5_evolution}
	\end{figure}

	\begin{figure}[htbp]
		\centering
		\includegraphics[width=0.94\linewidth]{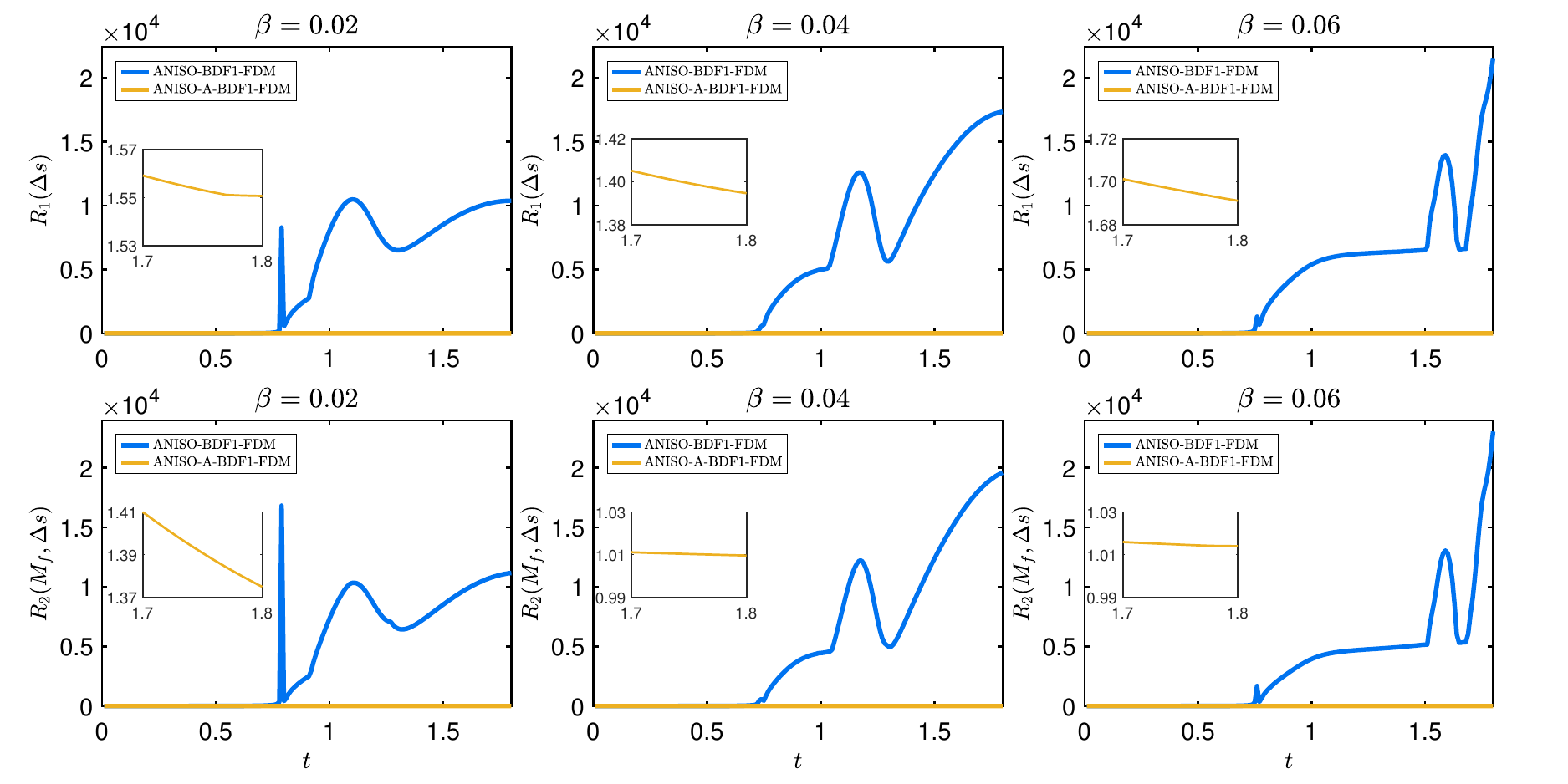}
		\caption{
			Time evolution of the mesh quality indicators \( R_1(\Delta s) \) (top row) and \( R_2(M_f,\Delta s) \) (bottom row) for the evolution shown in Fig.~\ref{fig:ANISO_example5_evolution}.
		}
		\label{fig:ANISO_example5_mesh_quality}
	\end{figure}
\end{example}
\begin{example}
	In this example, we check the energy stability of the proposed
	structure-preserving schemes, including the ANISO-A-LM-BDFk-FDMs and the
	ANISO-A-LM-CN-FDM, for the anisotropic axisymmetric MCF,
	which inherit the intrinsic structure of the continuous anisotropic model. 
	The initial generating curve is chosen as
	\[
	x(\rho,0) = 6 + 3\cos(2\pi\rho), \qquad
	y(\rho,0) = 2\sin(2\pi\rho), \qquad \rho \in [0,1].
	\]
	The anisotropic strength parameter is set to \( \beta = 0.04 \).
	We define \( W^h(t)\big|_{t=t_m} := W^m \) as the discrete anisotropic surface
	energy associated with the numerical generating curve, and we study the
	temporal evolution of the normalized energy \( W^h(t)/W^h(0) \).
	Fig.~\ref{fig:ANISO_energy_relative} displays the evolution of the normalized
	anisotropic energy \( W^h(t)/W^h(0) \) computed by the ANISO-A-LM-BDFk-FDMs and
	the ANISO-A-LM-CN-FDM. We observe that the discrete energy decreases
	monotonically for all schemes over the entire time interval, which is
	consistent with our theoretical results.
	
	\begin{figure} [htbp]
		\centering
		\includegraphics[width=0.6\linewidth]{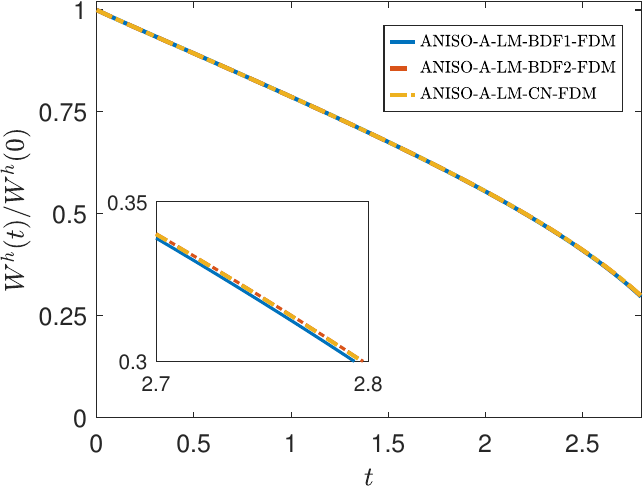}
		\caption{
			Time evolution of the normalized surface area $W^h(t)/W^h(0)$
			for the isotropic axisymmetric MCF computed by the
			ANISO-A-LM-BDFk-FDMs and ANISO-A-LM-CN-FDM.
			The initial generating curve is
			$x(\rho,0)=6+3\cos(2\pi\rho)$ and $y(\rho,0)=2\sin(2\pi\rho)$.
			The final time is $T=2.8$ with a uniform time step $\Delta t=0.01$.
		}
		\label{fig:ANISO_energy_relative}
	\end{figure}
\end{example}

\subsection{Numerical examples for axisymmetric surface evolution}
\begin{example}
	In this example, we illustrate the evolution of an axisymmetric surface under the
	isotropic axisymmetric MCF computed by the ISO-A-BDF1-FDM. 
		Fig.~\ref{fig:ISO_3D_example1}(a) shows the planar evolution of the generating curve
	at selected time instants \(t = 0,\,0.4,\,1.4,\) and \(2.15\).
	It can be observed that the curve evolves smoothly and shrinks inward.
	To quantitatively assess the mesh quality during the evolution,
	Fig.~\ref{fig:ISO_3D_example1}(b) presents the time histories of the mesh quality indicators
	\(R_1(\Delta s)\) and \(R_2(M\,\Delta s)\).
	The results demonstrate that both indicators remain bounded throughout the computation,
	indicating that the adaptive tangential redistribution effectively maintains a
	reasonable mesh distribution. 
	Finally, Fig.~\ref{fig:ISO_3D_example1}(c) displays the three-dimensional evolution of the
	axisymmetric surface obtained by revolving the generating curve about the rotation axis
	at the corresponding time instants.
	The numerical results confirm the robustness and stability of the proposed
	ISO-A-BDF1-FDM for simulating isotropic axisymmetric MCF.
	\begin{figure}[htbp]
		\begin{minipage}[t]{0.48\linewidth}
			\centering
			\includegraphics[width=\linewidth]{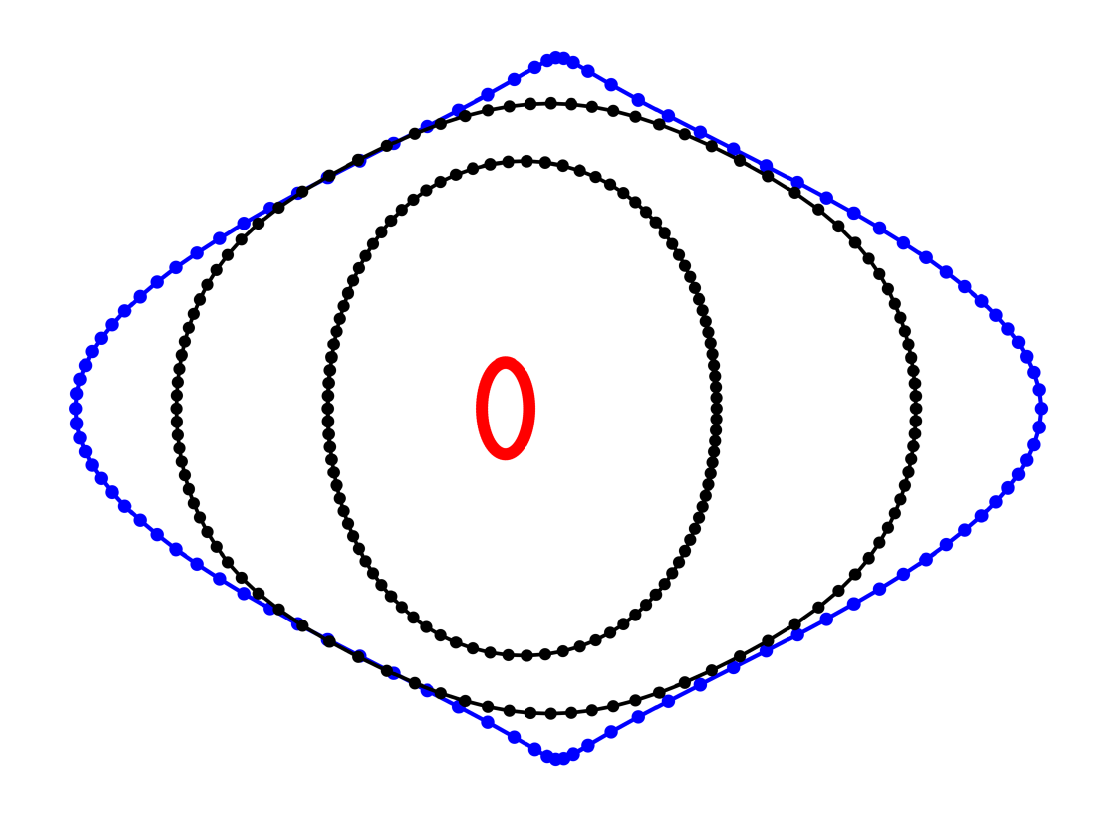}\\
			(a)
		\end{minipage}
		\hfill
		\begin{minipage}[t]{0.5\linewidth}
			\centering
			\includegraphics[width=\linewidth]{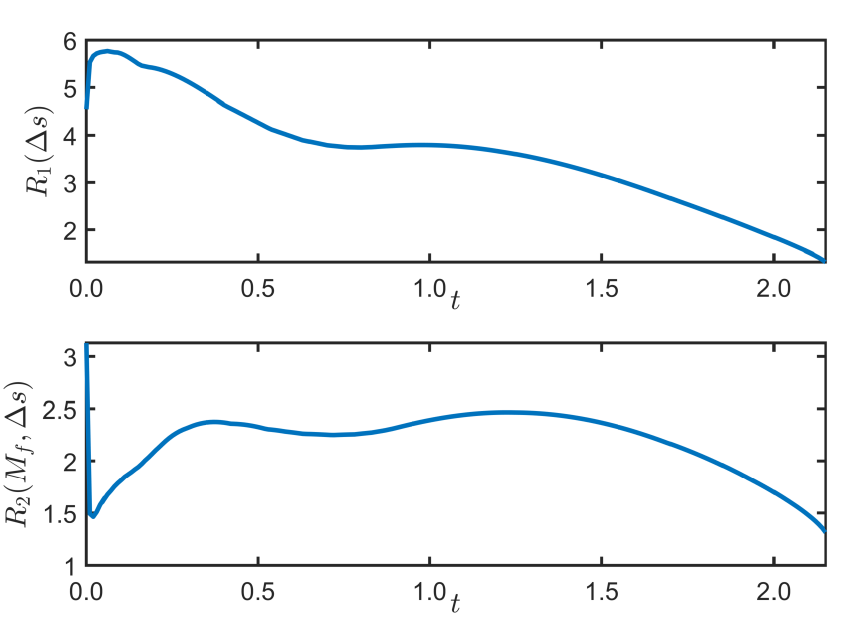}\\
			(b) 
		\end{minipage}
		
		\vspace{0.6em}
		
		\begin{minipage}[t]{1\linewidth}
			\centering
			\makebox[\linewidth][c]{%
				\hspace*{0.07\linewidth}%
				\includegraphics[width=\linewidth]{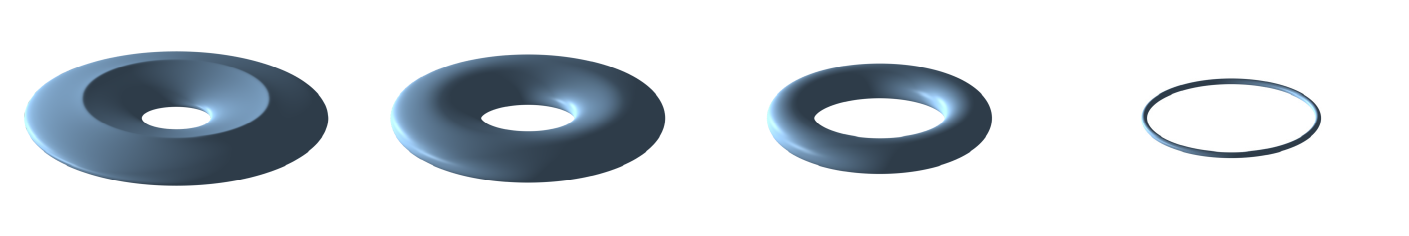}}\\
			(c) 
		\end{minipage}
		
		\caption{Evolution of an axisymmetric surface under the isotropic axisymmetric MCF
			computed by the ISO-A-BDF1-FDM.
			(a) Planar evolution of the generating curve at
			\(t = 0,\,0.4,\,1.4,\) and \(2.15\).
			The initial curve is given by
			$x(\rho,0) = 6 + 3\cos\left(2\pi\rho\right) + 0.8\cos\left(6\pi\rho\right),$ $y(\rho,0) = 1.4\sin\left(2\pi\rho\right),$
			and the final time is set to \(T = 2.15\) with time step \(\Delta t = 0.01\).
			(b) Time histories of the mesh quality indicators
			\(R_1(\Delta s)\) (top) and \(R_2(M\,\Delta s)\) (bottom).
			(c) Three-dimensional visualization of the evolving surface obtained by revolving
			the generating curve about the rotation axis at the corresponding time instants.}
		\label{fig:ISO_3D_example1}
	\end{figure}
	
\end{example}

\begin{example}
	In this example, we further investigate the evolution of an axisymmetric surface under the
	isotropic axisymmetric MCF using the ISO-A-BDF1-FDM,
	starting from a geometrically nonconvex generating curve.
	Fig.~\ref{fig:iso_3D_example2}(a) shows the planar evolution of the generating curve at
	the time instants \(t = 0,\,0.4,\,1.4,\) and \(2\).
	The initial curve exhibits pronounced nonconvex features and oscillatory geometric
	structures.
	As the evolution proceeds, these geometric oscillations and local shape variations
	are gradually attenuated. 
	The corresponding time histories of the mesh quality indicators
	\(R_1(\Delta s)\) and \(R_2(M\,\Delta s)\) are displayed in
	Fig.~\ref{fig:iso_3D_example2}(b).
	It can be observed that both indicators decrease rapidly during the initial stage
	and remain bounded for the remainder of the simulation,
	indicating that the adaptive tangential redistribution strategy effectively maintains
	a stable and well-distributed mesh throughout the evolution. 
	Finally, Fig.~\ref{fig:iso_3D_example2}(c) presents the three-dimensional evolution of the
	axisymmetric surface obtained by revolving the generating curve about the rotation axis
	at the corresponding time instants.
	The numerical results demonstrate that the ISO-A-BDF1-FDM is capable of robustly
	handling nonconvex initial geometries.
	\begin{figure}[htbp]
		\centering
		\begin{minipage}[t]{0.5\linewidth}
			\centering
			\includegraphics[width=\linewidth]{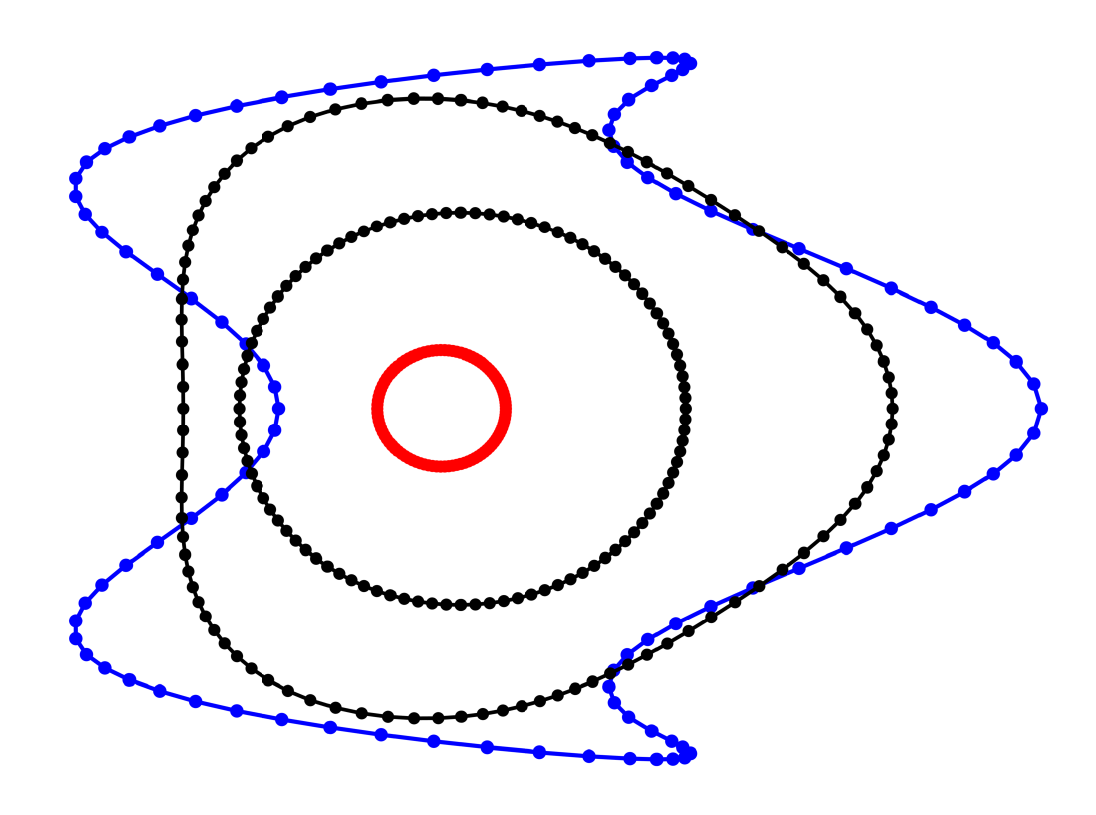}\\
			(a)
		\end{minipage}
		\hfill
		\begin{minipage}[t]{0.48\linewidth}
			\centering
			\includegraphics[width=\linewidth]{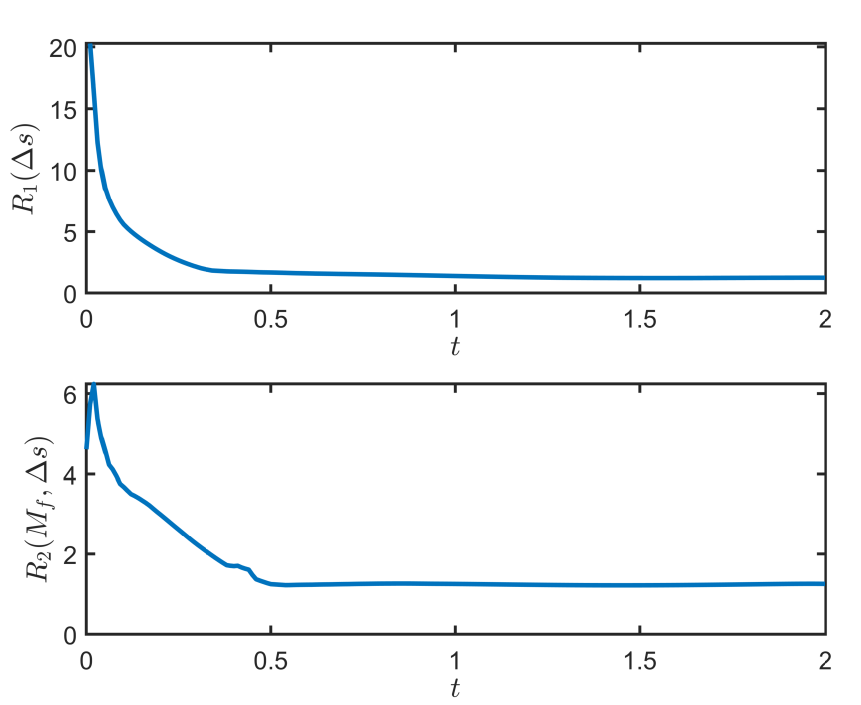}\\
			(b) 
		\end{minipage}
		
		\vspace{0.6em}
		
		\begin{minipage}[t]{1\linewidth}
			\centering
			\makebox[\linewidth][c]{%
				\hspace*{0.06\linewidth}%
				\includegraphics[width=\linewidth]{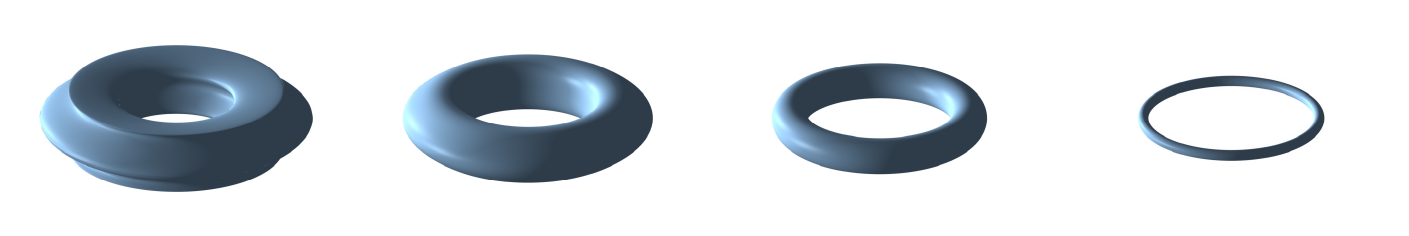}}\\
			(c) 
		\end{minipage}
		
		\caption{Evolution of an axisymmetric surface under the isotropic axisymmetric MCF
			computed by the ISO-A-BDF1-FDM.
			(a) Planar evolution of the generating curve at
			\(t = 0,\,0.4,\,1.4,\) and \(2\).
			The initial curve is given by
			$x(\rho,0) = 6 + 2\cos\!\left(2\pi\rho\right)
			+ 0.8\cos\!\left(8\pi\rho\right), $ $y(\rho,0) = 2\sin\!\left(2\pi\rho\right),$
			and the final time is set to \(T = 2\) with time step \(\Delta t = 0.02\).
			(b) Time histories of the mesh quality indicators
			\(R_1(\Delta s)\) (top) and \(R_2(M\,\Delta s)\) (bottom).
			(c) Three-dimensional visualization of the evolving surface obtained by revolving
			the generating curve about the rotation axis at the corresponding time instants.}
		\label{fig:iso_3D_example2}
	\end{figure}

\end{example}

\begin{example}
	In this example, we investigate the evolution of an axisymmetric surface under the
	four-fold anisotropic axisymmetric MCF using the
	ANISO-A-BDF1-FDM, with the anisotropy parameter set to \(\beta = 0.06\). 
	Fig.~\ref{fig:ANiso_3D_example1}(a) shows the planar evolution of the generating curve
	at the time instants \(t = 0,\,0.5,\,1.6,\) and \(2.1\).
	Due to the presence of four-fold anisotropy, the curve evolution exhibits
	direction-dependent geometric features that differ noticeably from the isotropic case. 
	The corresponding time histories of the mesh quality indicators
	\(R_1(\Delta s)\) and \(R_2(M\,\Delta s)\) are presented in
	Fig.~\ref{fig:ANiso_3D_example1}(b).
	Although the presence of anisotropy introduces additional geometric complexity,
	both indicators exhibit noticeable fluctuations during the early stage of the evolution,
	followed by an overall decreasing tendency, and eventually enter a regime with
	relatively mild temporal variations. 
	Finally, Fig.~\ref{fig:ANiso_3D_example1}(c) displays the three-dimensional evolution of
	the axisymmetric surface obtained by revolving the generating curve about the
	rotation axis at the corresponding time instants.
	The numerical results demonstrate that the proposed ANISO-A-BDF1-FDM
	can robustly simulate axisymmetric MCF with pronounced directional anisotropy.
	\begin{figure}[htbp]
		\centering
		
		\begin{minipage}[t]{0.48\linewidth}
			\centering
			\includegraphics[width=\linewidth]{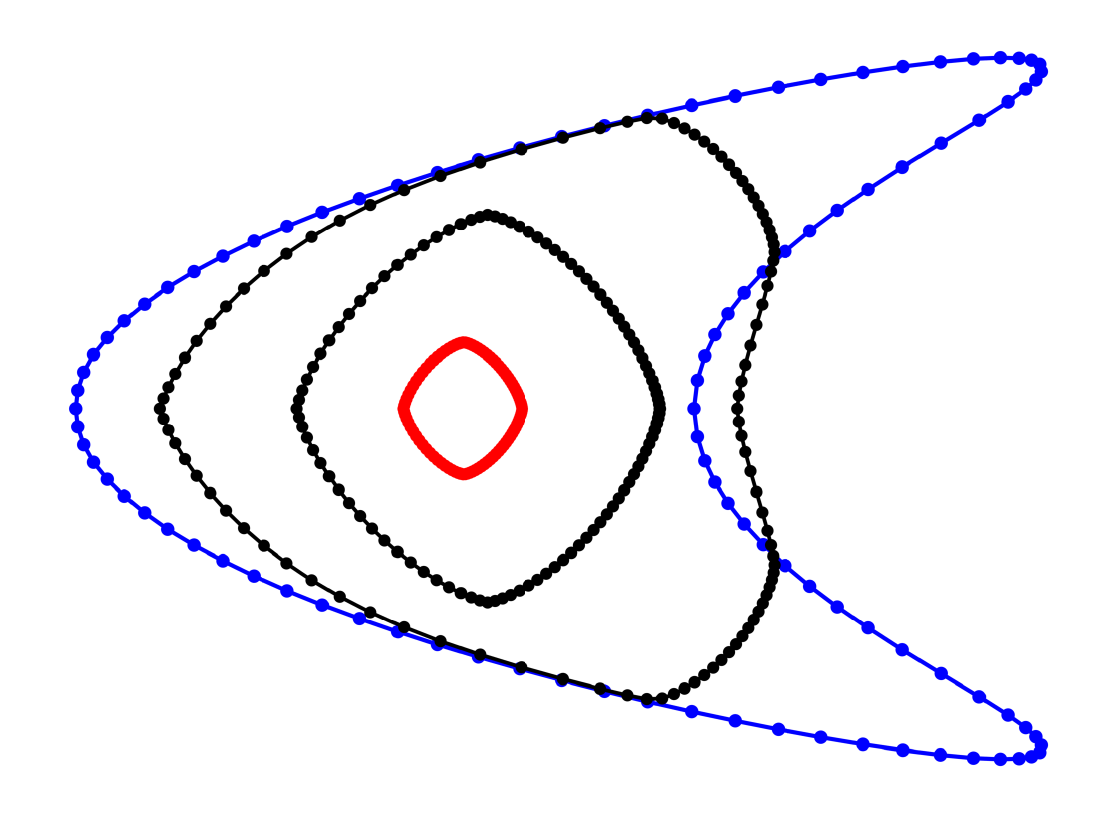}\\
			(a)
		\end{minipage}
		\hfill
		\begin{minipage}[t]{0.48\linewidth}
			\centering
			\includegraphics[width=\linewidth]{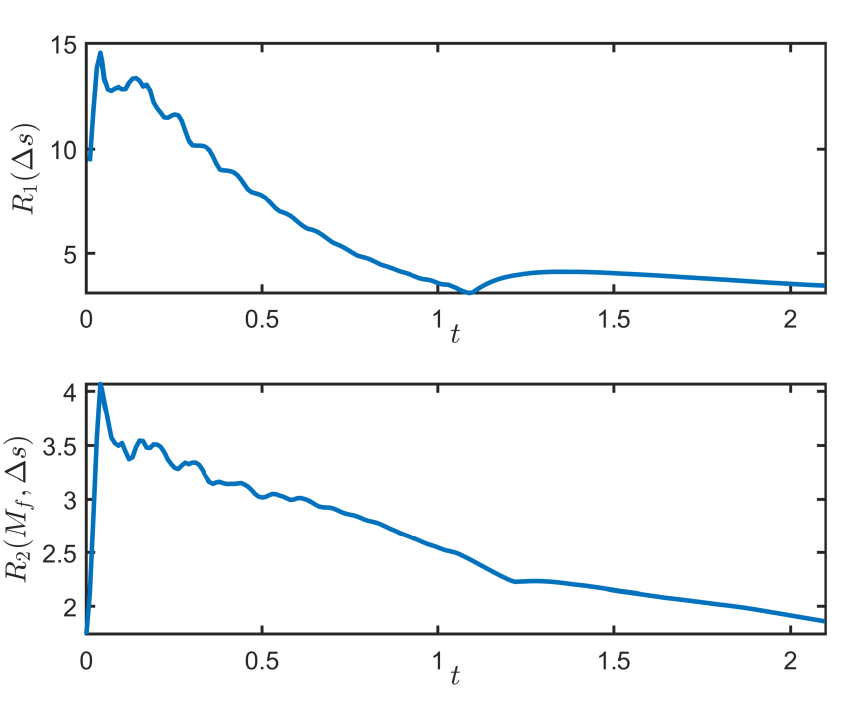}\\
			(b) 
		\end{minipage}
		
		\vspace{0.6em}
		
		\begin{minipage}[t]{1\linewidth}
			\centering
			\makebox[\linewidth][c]{%
				\hspace*{0.07\linewidth}%
				\includegraphics[width=\linewidth]{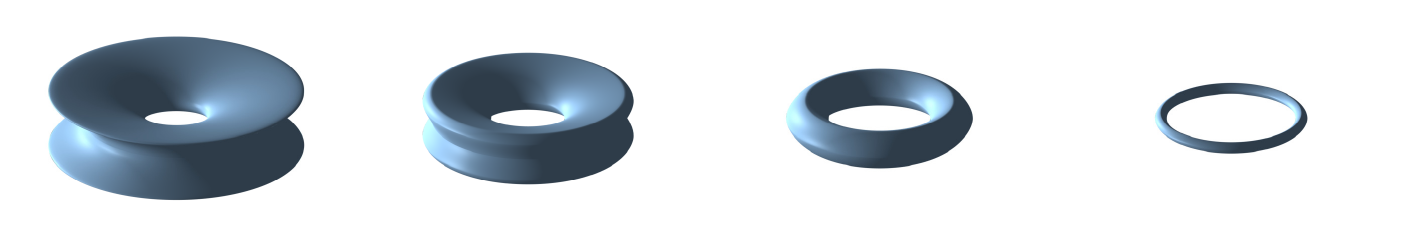}}\\
			(c) 
		\end{minipage}
		
		\caption{Evolution of an axisymmetric surface under the anisotropic axisymmetric MCF
			computed by the ANISO-A-BDF1-FDM.
			(a) Planar evolution of the generating curve at
			\(t = 0,\,0.6,\,1.6,\) and \(2.1\).
			The initial curve is given by
			$
			x(\rho,0) = 6 + 2\cos\!\left(2\pi\rho\right) - 2\cos\!\left(4\pi\rho\right),$ $y(\rho,0) = 2\sin\!\left(2\pi\rho\right),$
			and the final time is set to \(T = 2.1\) with time step \(\Delta t = 0.01\).
			(b) Time histories of the mesh quality indicators
			\(R_1(\Delta s)\) (top) and \(R_2(M\,\Delta s)\) (bottom).
			(c) Three-dimensional visualization of the evolving surface obtained by revolving
			the generating curve about the symmetry axis at the corresponding time instants.}
		\label{fig:ANiso_3D_example1}
	\end{figure}
\end{example}

\begin{example}
	In this example, we consider the evolution of an axisymmetric surface under the
	anisotropic axisymmetric MCF computed by the ANISO-A-BDF1-FDM.
	The anisotropy is characterized by a four-fold structure with the anisotropy parameter set to \(\beta = 0.06\).
	Fig.~\ref{fig:ANiso_3D_example2}(a) presents the planar evolution of the generating curve
	at the time instants \(t = 0,\,0.8,\,1.6,\) and \(2.1\).
	The initial curve contains multiple oscillatory modes and exhibits pronounced geometric
	undulations.
	During the evolution, these features interact with the anisotropic curvature effects,
	leading to a visibly direction-dependent deformation of the curve. 
	The corresponding time histories of the mesh quality indicators
	\(R_1(\Delta s)\) and \(R_2(M\,\Delta s)\) are shown in
	Fig.~\ref{fig:ANiso_3D_example2}(b).
	Both indicators display non-monotone behavior with noticeable fluctuations at the early
	stage, followed by a transition to a regime with comparatively mild temporal variations.
	This behavior reflects the interplay between geometric complexity and anisotropic effects
	during the evolution. 
	Finally, Fig.~\ref{fig:ANiso_3D_example2}(c) shows the three-dimensional evolution of the
	axisymmetric surface obtained by revolving the generating curve about the rotation axis
	at the corresponding time instants.
	The numerical results illustrate that the ANISO-A-BDF1-FDM is capable of robustly capturing
	anisotropy-induced geometric features for initial curves with multiple oscillatory modes.
	\begin{figure}[htbp]
		\centering
		
		\begin{minipage}[t]{0.48\linewidth}
			\centering
			\includegraphics[width=\linewidth]{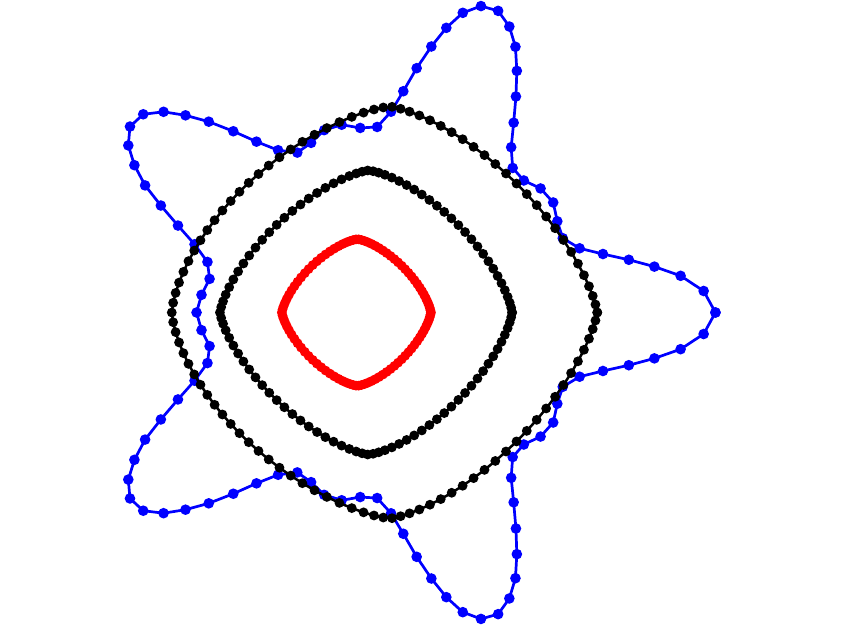}\\
			(a)
		\end{minipage}
		\hfill
		\begin{minipage}[t]{0.48\linewidth}
			\centering
			\includegraphics[width=\linewidth]{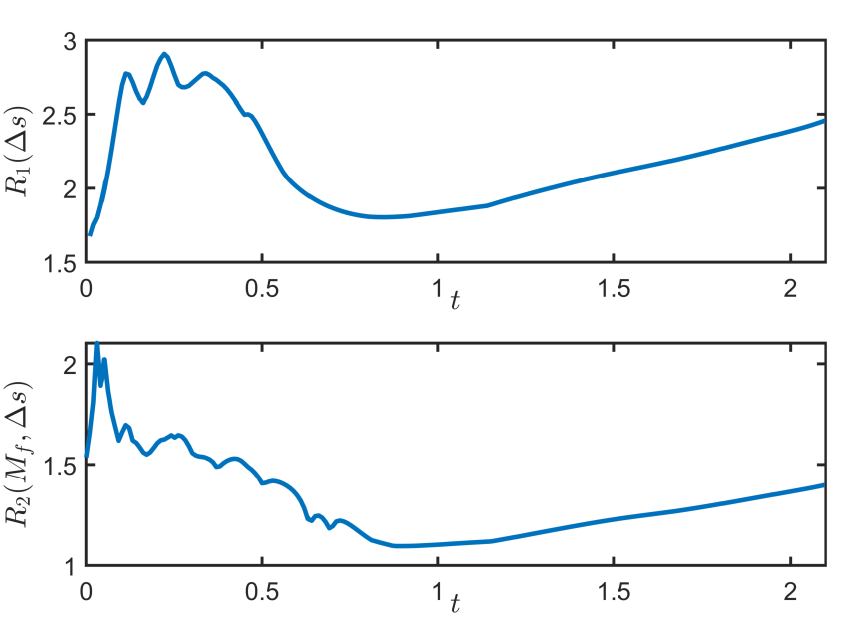}\\
			(b) 
		\end{minipage}
		
		\vspace{0.6em}
		
		\begin{minipage}[t]{1\linewidth}
			\centering
			\makebox[\linewidth][c]{%
				\hspace*{0.06\linewidth}%
				\includegraphics[width=\linewidth]{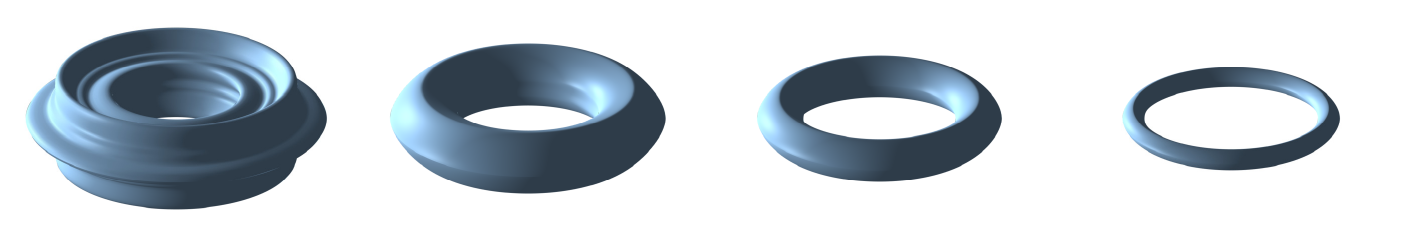}%
			}\\
			(c) 
		\end{minipage}
		
		\caption{Evolution of an axisymmetric surface under the anisotropic axisymmetric MCF
			computed by the ANISO-A-BDF1-FDM.
			(a) Planar evolution of the generating curve at
			\(t = 0,\,0.8,\,1.6,\) and \(2.1\).
			The initial curve is given by
			$x(\rho,0) = 6 + \left[2 + 0.55\cos\!\left(10\pi\rho\right)
			+ 0.25\cos\!\left(20\pi\rho\right)\right]\cos\!\left(2\pi\rho\right),$ 
			$y(\rho,0) = \left[2 + 0.55\cos\!\left(10\pi\rho\right)
			+ 0.25\cos\!\left(20\pi\rho\right)\right]\sin\!\left(2\pi\rho\right),$
			and the final time is set to \(T = 2.1\) with time step \(\Delta t = 0.01\).
			(b) Time histories of the mesh quality indicators
			\(R_1(\Delta s)\) (top) and \(R_2(M\,\Delta s)\) (bottom).
			(c) Three-dimensional visualization of the evolving surface obtained by revolving
			the generating curve about the symmetry axis at the corresponding time instants.}
		\label{fig:ANiso_3D_example2}
	\end{figure}
	
\end{example}

\section{Conclusions}\label{sec6}
In this work, we have developed adaptive moving mesh methods for the numerical
simulation of axisymmetric MCFs in both isotropic and anisotropic
settings.
The methods are formulated within the mesh equidistribution framework, in which
a tangential velocity driven by carefully designed monitor functions is introduced
to dynamically redistribute mesh points during the evolution.
The proposed monitor functions, incorporating curvature, its arc-length derivative,
and squared curvature, provide effective control of mesh quality and enable accurate
resolution of evolving geometric features.
The resulting numerical schemes are discretized in space using central finite
differences and in time by first- and second-order schemes, including the BDFk
($k=1,2$) and Crank-Nicolson methods.
By incorporating a Lagrange multiplier approach, the adaptive system preserves the
underlying geometric constraint and leads to energy-stable fully discrete schemes.
Numerical experiments confirm the convergence and energy stability of the proposed
methods.
More importantly, the results highlight the clear advantages of the adaptive strategy
in complex geometric evolutions.
Through dynamic mesh redistribution, the adaptive methods effectively capture
localized geometric features, prevent mesh degeneration, and significantly improve
numerical accuracy.
In particular, for anisotropic curvature evolution problems, the ability to maintain
good mesh quality enables stable and reliable long-time numerical simulations.
The proposed framework is flexible and can be extended to other geometric evolution
problems where anisotropy and strong curvature effects play a significant role.

\bibliographystyle{elsarticle-num}
\bibliography{ref}
\end{document}